\documentclass[11pt]{article}
\usepackage{amsmath,amsthm,amssymb}
\usepackage[usenames,dvipsnames]{xcolor}
\usepackage{enumerate}
\usepackage{graphicx}
\usepackage{cite}
\usepackage{comment}
\usepackage{oands}
\usepackage{tikz}
\usepackage{changepage}
\usepackage{bbm}
\usepackage{mathtools}
\usepackage[margin=1in]{geometry}
\usepackage[pagewise,mathlines]{lineno}
\usepackage{appendix}
\usepackage{stmaryrd} 
\usepackage{multicol}
\usepackage{microtype}
\usepackage[colorinlistoftodos]{todonotes}
\usepackage{enumitem}
\usepackage{ulem}
\newcommand{\old}[1]{{\color{cyan}{}}} 

\usepackage[pdftitle={Liouville quantum gravity with central charge greater than 1},
  pdfauthor={Ewain Gwynne, Nina Holden, Joshua Pfeffer, Guillaume Remy},
colorlinks=true,linkcolor=NavyBlue,urlcolor=RoyalBlue,citecolor=PineGreen,bookmarks=true,bookmarksopen=true,bookmarksopenlevel=2,unicode=true,linktocpage]{hyperref}

\theoremstyle{plain}
\newtheorem{thm}{Theorem}[section]

\newtheorem{lem}[thm]{Lemma}
\newtheorem{prop}[thm]{Proposition}
\newtheorem{conj}[thm]{Conjecture}

\def\@rst #1 #2other{#1}
\newcommand\MR[1]{\relax\ifhmode\unskip\spacefactor3000 \space\fi
  \MRhref{\expandafter\@rst #1 other}{#1}}
\newcommand{\MRhref}[2]{\href{http://www.ams.org/mathscinet-getitem?mr=#1}{MR#2}}

\theoremstyle{definition}
\newtheorem{defn}[thm]{Definition}
\newtheorem{remark}[thm]{Remark}

\newtheorem{ques}[thm]{Question}

\numberwithin{equation}{section}

\newcommand{\dsb}{\begin{adjustwidth}{2.5em}{0pt}
\begin{footnotesize}}
\newcommand{\dse}{\end{footnotesize}
\end{adjustwidth}}

\newcommand{\ssb}{\begin{adjustwidth}{2.5em}{0pt}}
\newcommand{\sse}{\end{adjustwidth}}

\newcommand{\aryb}{\begin{eqnarray*}}
\newcommand{\arye}{\end{eqnarray*}}
\def\alb#1\ale{\begin{align*}#1\end{align*}}
\def\allb#1\alle{\begin{align}#1\end{align}}
\newcommand{\eqb}{\begin{equation}}
\newcommand{\eqe}{\end{equation}}
\newcommand{\eqbn}{\begin{equation*}}
\newcommand{\eqen}{\end{equation*}}

\newcommand{\BB}{\mathbbm}
\newcommand{\ol}{\overline}
\newcommand{\ul}{\underline}
\newcommand{\op}{\operatorname}

\newcommand{\frk}{\mathfrak}
\newcommand{\eqD}{\overset{d}{=}}
\newcommand{\ep}{\epsilon}
\newcommand{\rta}{\rightarrow}

\newcommand{\wt}{\widetilde}
\newcommand{\wh}{\widehat} 
\newcommand{\mcl}{\mathcal}

\newcommand{\bdy}{\partial}

\newcommand{\ccL}{{\mathbf{c}_{\mathrm L}}}
\newcommand{\ccM}{{\mathbf{c}_{\mathrm M}}}

\newcommand{\tr}{{\mathrm{tr}}}
\newcommand{\LFPP}{{\textnormal{\tiny{\textsc{LFPP}}}}}

\let\originalleft\left
\let\originalright\right
\renewcommand{\left}{\mathopen{}\mathclose\bgroup\originalleft}
\renewcommand{\right}{\aftergroup\egroup\originalright}

\title{Liouville quantum gravity with matter central charge in $(1,25)$: a probabilistic approach}

\date{  }
\author{
\begin{tabular}{c} Ewain Gwynne\\[-4pt]\small Cambridge \end{tabular} 
\begin{tabular}{c} Nina Holden\\[-4pt]\small ETH Z\"urich \end{tabular}
\begin{tabular}{c} Joshua Pfeffer\\[-4pt]\small MIT \end{tabular}
\begin{tabular}{c} Guillaume Remy \\[-4pt]\small Columbia \end{tabular}
}

\begin{document}

\maketitle
  
\begin{abstract} 
There is a substantial literature concerning Liouville quantum gravity (LQG) in two dimensions with conformal matter field of central charge ${\mathbf{c}}_{\mathrm M} \in (-\infty,1]$. 
Via the DDK ansatz, LQG can equivalently be described as the random geometry obtained by exponentiating $\gamma$ times a variant of the planar Gaussian free field, where $\gamma \in (0,2]$ satisfies $\mathbf c_{\mathrm M} = 25 - 6(2/\gamma + \gamma/2)^2$. 
Physics considerations suggest that LQG should also make sense in the regime when $\mathbf c_{\mathrm M} > 1$.
However, the behavior in this regime is rather mysterious in part because the corresponding value of $\gamma$ is complex, so analytic continuations of various formulas give complex answers which are difficult to interpret in a probabilistic setting. 

We introduce and study a discretization of LQG which makes sense for all values of $\mathbf c_{\mathrm M} \in (-\infty,25)$. Our discretization consists of a random planar map, defined as the adjacency graph of a tiling of the plane by dyadic squares which all have approximately the same ``LQG size" with respect to the Gaussian free field. We prove that several formulas for dimension-related quantities are still valid for $\mathbf c_{\mathrm M} \in (1,25)$, with the caveat that the dimension is infinite when the formulas give a complex answer. In particular, we prove an extension of the (geometric) KPZ formula for $\mathbf c_{\mathrm M}  \in (1,25)$, which gives a finite quantum dimension if and only if the Euclidean dimension is at most $(25-\mathbf c_{\mathrm M} )/12$. We also show that the graph distance between typical points with respect to our discrete model grows polynomially whereas the cardinality of a graph distance ball of radius $r$ grows faster than any power of $r$ (which suggests that the Hausdorff dimension of LQG in the case when $\mathbf c_{\mathrm M}  \in(1,25)$ is infinite). 

We include a substantial list of open problems. 
\end{abstract}

\tableofcontents

\section{Introduction}
\label{sec-intro}

\subsection{Overview}
\label{sec-overview}

Liouville quantum gravity (LQG) is a one-parameter family of random surfaces which describe two-dimensional quantum gravity coupled with conformal matter fields. It was first introduced in physics by Polyakov~\cite{polyakov-qg1} in the context of bosonic string theory in order to define a ``sum over Riemannian metrics" in two dimensions. 
To define LQG, consider a parameter $\ccM \in (-\infty,1]$ which we call the \emph{matter central charge}, corresponding in physics to the central charge of the conformal field theory (CFT) given by the matter fields. 
For a compact surface $\mcl D$ and a Riemannian metric $g$ on $\mcl D$, let $\Delta_g$ be the corresponding Laplace-Beltrami operator. 
Heuristically speaking, an LQG surface with the $\mcl D $ topology and matter central charge $\ccM$ is a random surface sampled from the measure on Riemannian metric tensors $g$ on $\mcl D $ whose probability density with respect to the ``Lebesgue measure on the space of metrics $g$ on $\mcl D $" is proportional to $(\det \Delta_g)^{-\ccM/2}$. The determinant  $(\det  \Delta_g)^{-\ccM/2}$ can be thought of as the partition function of a statistical mechanics model whose scaling limit is described by a CFT with central charge $\ccM$.

It was argued by David~\cite{david-conformal-gauge} and Distler-Kawai~\cite{dk-qg} via the so-called \emph{DDK ansatz} that the partition function of LQG --- which in principle governs the law of the random metric $g$ on $\mcl D $ --- can be constructed by integrating the product of the partition function of the matter fields (which should behave like $(\det \Delta_g)^{-\ccM/2}$) times the one of Liouville conformal field theory (LCFT) over the moduli space of $\mcl D $. See also \cite{dp-multiloop} for more explanations. In this context physicists define LCFT by the path integral formalism using the Liouville action, see Section \ref{sec-action} for more detail. LCFT is also parametrized by one real parameter which we denote by $\ccL$, the central charge of LCFT.  In order for LQG to be conformally invariant the DDK ansatz further imposes the relation 
\eqb
\ccM + \ccL - 26 =0.
\label{eq-cM-cL-sum}
\eqe

In this article we will be primarily interested in the local geometry of LQG surfaces so we restrict our attention to the case of simply connected surfaces,\footnote{
See~\cite{drv-torus,remy-annulus,grv-higher-genus} for works concerning LQG with non-simply connected topologies.
} in which case the moduli space of $\mcl D $ is trivial and the Riemannian metric tensor of the LQG surface is simply given by
\eqb \label{eqn-lqg-metric-tensor}
e^{\gamma h} (dx^2 + dy^2), 
\eqe
where $h$ is a variant of Gaussian free field (GFF), $dx^2+dy^2$ is the standard Euclidean metric tensor on $\mcl D $, and $\gamma$ is the unique solution in $(0,2]$ of the equation\footnote{In the physics literature it is common to use $b = \frac{\gamma}{2}$ instead of $\gamma$.}
\eqb \label{eqn-Q-c-gamma}
\ccM = 25 -6Q^2 \quad \text{where} \quad Q = \frac{2}{\gamma}  +\frac{\gamma}{2} . 
\eqe 

Since \eqref{eqn-Q-c-gamma} gives a bijection between $\ccM \in(-\infty,1]$ and $\gamma\in(0,2]$, we can parametrize LQG using $\gamma$ instead of $\ccM$ and talk about $\gamma$-LQG. 
The parameter $Q  \geq  2$ is called the \emph{background charge} and will play a fundamental role in this article. 
We also record the formula $\ccL = 1 + 6 Q^2 $, which is obtained by combining \eqref{eq-cM-cL-sum} and \eqref{eqn-Q-c-gamma}. It was shown rigorously in~\cite{dkrv-lqg-sphere} that the partition function of LCFT transforms under conformal rescaling as the partition function of a CFT with central charge $\ccL$ for $\ccL \geq 25$, which corresponds to $\ccM \in(-\infty,1]$, in the case of the sphere topology (see~\cite{hrv-disk,drv-torus,grv-higher-genus} for other topologies). From the point of view of constructive quantum field theory, this means that LCFT is a CFT with central charge $\ccL$. Figure~\ref{fig-c-Q-gamma-table} is a table of the relationships between $\ccL,\ccM,Q,$ and $\gamma$.

\begin{figure}[t!]
 \begin{center}
\includegraphics[scale=1.1]{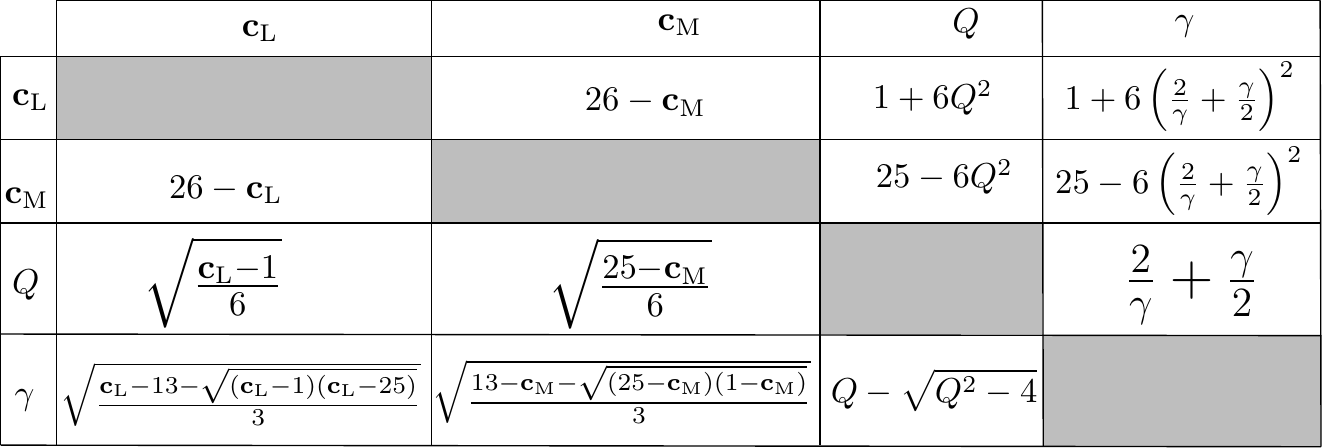}
\vspace{-0.01\textheight}
\caption{Table of relationships between the values of $\ccL$, $\ccM$, $Q$, and $\gamma$. Note that $\gamma$ is complex when $\ccM \in (1,25]$, but $Q$ is real for $\ccM \in (-\infty,25]$.  
}\label{fig-c-Q-gamma-table}
\end{center}
\vspace{-1em}
\end{figure} 

The Riemann metric tensor definition~\eqref{eqn-lqg-metric-tensor} does not make literal sense since the GFF $h$ is only a generalized function, not a true function, so it does not have well-defined pointwise values and cannot be exponentiated pointwise. 
Nevertheless, it is possible to make rigorous sense of this object via various regularization procedures. 
For example, one can construct for each $\gamma \in (0,2]$ a random measure $\mu_h^\gamma$ on $\mcl D$ which is the limit of regularized versions of $e^{\gamma h(z)} \, d^2 z$, where $ d^2 z$ denotes Lebesgue measure. The construction of this measure is a special case of a general theory of regularized random measures called \emph{Gaussian multiplicative chaos} (GMC), which was initiated by Kahane~\cite{kahane}; see~\cite{rhodes-vargas-review,berestycki-gmt-elementary,aru-gmc-survey} for reviews of this theory.
The particular construction of the measure which is most relevant to this paper (using circle averages of the GFF) was introduced in~\cite{shef-kpz}. 
See~\cite{shef-deriv-mart,shef-renormalization} for the construction of the measure in the critical case $\gamma=2$. 

It was shown very recently in~\cite{dddf-lfpp,gm-uniqueness} that for $\gamma \in (0,2)$, an LQG surface also gives rise to a random distance function\footnote{Here we mean distance function in the sense of a metric which gives the distance between any pair of points. We use the phrase ``distance function'' rather than ``metric'' to avoid confusion with the Riemannian metric tensor $e^{\gamma h}(dx^2+dy^2)$. } $D_h^\gamma$ on $\mcl D$ which is the limit of regularized versions of the Riemannian distance function associated with~\eqref{eqn-lqg-metric-tensor}. 
The distance function in the special case when $\gamma=\sqrt{8/3}$ ($\ccM  =0$) was constructed in earlier work~\cite{lqg-tbm1,lqg-tbm2,lqg-tbm3} using a completely different method, in which case the resulting metric space (for a certain a special choice of $h$) is isometric to the Brownian map~\cite{legall-uniqueness,miermont-brownian-map}.

Another possible approach to LQG, based on the Laplacian determinant definition, is to look at a discretized version of the random surface known as a random planar map. 
Recall that a planar map (with the sphere topology) is a graph embedded into the Riemann sphere, viewed modulo orientation-preserving homeomorphisms. 
See e.g.\ \cite{miermont-st-flour,legall-sphere-survey,curien-peeling-notes} for an introduction to random planar maps.  

For each $\ccM \in(-\infty,1)$, there are statistical mechanics models on planar maps whose partition functions are known or expected to behave like the $-\ccM/2$-power of the discrete Laplacian determinant when the number of edges of the planar map is large. 
Examples include the uniform spanning tree ($\ccM = -2$) and the critical Ising model ($\ccM = 1/2$). 
Suppose $M^n$ is a planar map with $n$ edges sampled with probability proportional to one of these partition functions. 
We can think of $M^n$ as a discrete random surface by endowing each face with the surface structure of a polygon.
From this perspective, it is natural to believe that when $n$ is large, $M^n$ is a good approximation of LQG with the spherical topology and matter central charge $\ccM$. 
Indeed, $M^n$ and variants thereof are used throughout the physics literature as a heuristic approximation of LQG with matter central charge $\ccM$; see, e.g.,~\cite{djkp-critical-exponents,adf-critical-dimensions,bkkm-analytical-study,bd-triangulated-surfaces,adjt-c-ge1}. 
Informally, we say that a random planar map model belongs to the \emph{matter central charge $\ccM$ universality class} if it is expected to approximate LQG with matter central charge $\ccM$. 

So far, the random planar map approach has met with much success in the special case when $\ccM =0$ (i.e., $\gamma=\sqrt{8/3}$), in which case we are dealing with uniform random planar maps. Indeed, it is known that uniform random planar maps converge to $\sqrt{8/3}$-LQG surfaces both in the Gromov-Hausdorff sense~\cite{lqg-tbm1,lqg-tbm2,lqg-tbm3,legall-uniqueness,miermont-brownian-map} and under certain embeddings into the plane~\cite{hs-cardy-embedding}. 
There are also some weaker convergence results for general $\ccM \in (-\infty,1)$ based on the paper~\cite{wedges}; see~\cite{ghs-mating-survey} for a recent survey.

There is a substantial mathematical literature on LQG with matter central charge $\ccM \leq 1$, concerning topics such as the connection between LQG surfaces and random planar maps (see references above), the relationships between LQG surfaces and Schramm-Loewner evolutions~\cite{shef-zipper,wedges}, exact formulas for various quantities related to LQG surfaces~\cite{krv-dozz}, LQG surfaces with arbitrary genus~\cite{grv-higher-genus}, and the stochastic Ricci flow~\cite{ds-ricci-flow} (which is related to stochastic quantization). We will not attempt a comprehensive survey of this literature here, but the reader may consult the above-cited works and the references therein.
 
There is also significant physical interest in LQG with matter central charge $\ccM \in (1,25)$, or equivalently with background charge $Q \in (0,2) $. 
 By~\eqref{eqn-Q-c-gamma}, $\ccM \in (1,25)$ corresponds to a complex value of $\gamma$, with modulus 2. 
The probabilistic and geometric aspects of LQG in this phase are much less well-understood than in the case when $\ccM \leq 1$, even from a physics perspective. 
Nevertheless, it is reasonable to believe that LQG with matter central charge $\ccM \in (1,25)$ can still be realized as some sort of random geometry connected to the Gaussian free field. 
Moreover, a number of works have successfully analyzed LQG with $\ccM \in (1,25)$~\cite{bh-c-ge1-matrix,david-c>1-barrier,zam05} and LCFT with $\ccL \in (1,25)$~\cite{suzuki1997note,fkv-c>1-I,fk-c>1-II,teschner04,ribault-cft,rs15,ijs16,ribault2018minimal}. 

In this paper, we will introduce and study a discrete geometric/probabilistic model of LQG in the phase when $\ccM \in (1,25)$. A key idea behind our approach is the observation that the definition of an LQG surface as an equivalence class of surfaces transforming as dictated by the parameter $Q$ (a well-known fact in physics, see for instance~\cite{zz-dozz}, that was mathematically implemented in~\cite{shef-kpz,shef-zipper,wedges}), extends immediately to the case where $Q\in(0,2)$; see Section~\ref{sec-lqg-surface}. Our model takes the form of a one-parameter family of random planar maps, indexed by $\ccM \in (-\infty,25)$, which are defined as the adjacency graphs of a family of dyadic tilings of the plane constructed from the Gaussian free field. See Section~\ref{sec-tiling-def} for a precise definition. When $\ccM \in (-\infty,1]$, our dyadic tiling is closely related to LQG with matter central charge $\ccM \in (-\infty,1]$ as studied elsewhere in the literature. 
When $\ccM \in (1,25)$, we expect that the tiling should converge, in a certain sense, to a continuum model of LQG with matter central charge $\ccM \in (1,25)$. 

One of the major difficulties in the study of LQG with matter central charge $\ccM \notin (-\infty,1]$ is that a number of formulas for quantities related to LQG surfaces --- such as the (geometric) KPZ formula~\cite{kpz-scaling,shef-kpz}, the DOZZ formula~\cite{do-dozz,zz-dozz,krv-dozz}, and predictions for the Hausdorff dimension of LQG~\cite{watabiki-lqg,dg-lqg-dim} --- yield complex answers when $\ccM > 1$, which are difficult to interpret in a probabilistic framework. 
We will show that for our model for $\ccM \in (1,25)$, several notions of ``dimension" (such as the quantum dimension in the KPZ formula and the ball volume exponent for certain random planar maps) are given by analytic continuation of the formulas in the case when $\ccM \leq 1$, with the caveat that whenever the formulas give a non-real answer, the corresponding dimension is infinite. 
This means that the phase transition at $\ccM = 1$ is in some ways discontinuous, in the sense that the limits of certain dimensions as $\ccM \rta 1^-$ are finite, whereas the dimensions for $\ccM  \in (1,25)$ are infinite.

We refer to Section~\ref{sec-discussion} for additional discussion about the motivation for our model and its connection to other results in the literature.

\begin{figure}[t!]
 \begin{center}
\includegraphics[scale=1.1]{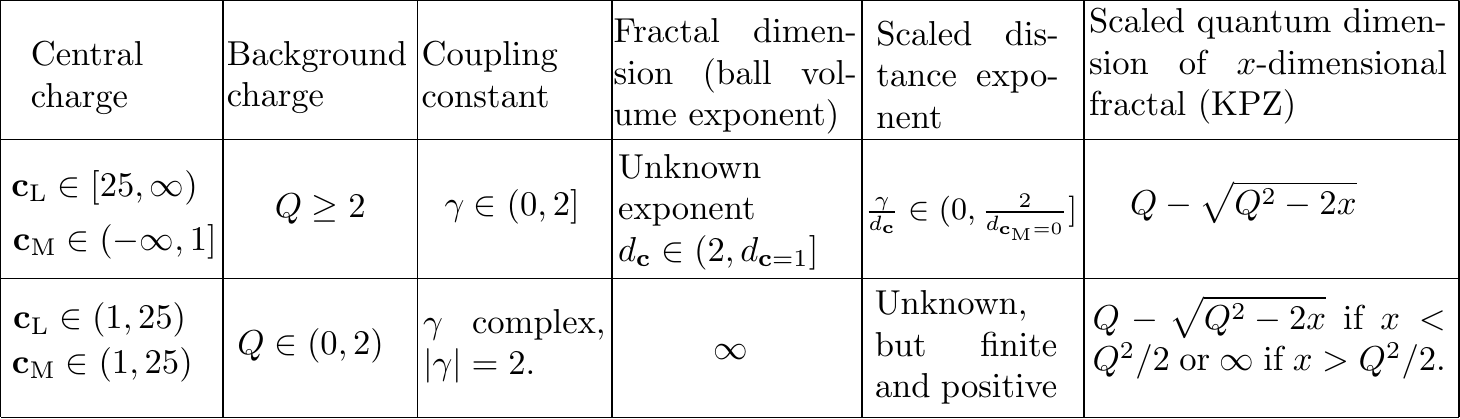}
\vspace{-0.01\textheight}
\caption{Table listing various features of the two phases of LQG relevant to this paper. The case when $\ccM \leq 1$ is the one considered in most previous mathematical works on LQG. The phase $\ccM \in (1,25)$ is the main focus of the present paper. The last three columns correspond to Theorem~\ref{thm-ball-infty}, Proposition~\ref{prop-ptwise-distance}, and Theorem~\ref{thm-kpz}, respectively.  
}\label{fig-c-phases-table}
\end{center}
\vspace{-1em}
\end{figure}

\subsection{Definition of the model} 
\label{sec-tiling-def}

For a domain $\mcl D\subset \BB C$, the \emph{zero-boundary GFF} on $\mcl D$ is the centered Gaussian random distribution $h^{\mcl D}$ on $\mcl D$ with covariances
\eqb \label{eqn-zero-bdy}
\op{Cov}\left( h^{\mcl D}(z) , h^{\mcl D}(w) \right) = \op{Gr}_{\mcl D}(z,w) ,
\eqe
where $\op{Gr}_{\mcl D}$ is the Green's function on $\mcl D$ with zero boundary conditions. 
The \emph{whole-plane GFF} (normalized so that its average over $\bdy\BB D$ is equal to zero) is the centered Gaussian random distribution $h = h^{\BB C}$ on $\BB C$ with covariances
\eqb \label{eqn-whole-plane}
\op{Cov}\left( h (z) , h (w) \right) = \log \frac{|z|_+ |w|_+}{|z-w|}   \quad \text{where} \quad |z|_+ := \max(|z|,1) .
\eqe
We say that a random distribution $h$ on a domain $\mcl D \subset\BB C$ is a \emph{GFF plus a continuous function} if there is a coupling of $h$ with a whole-plane or zero-boundary GFF $\wt h$ on $\mcl D$ such that $h-\wt h$  is a.s.\ equal to a continuous function. 
We refer to~\cite{shef-gff} and the introductory sections of~\cite{ss-contour,shef-zipper,ig1,ig4} for more background on the GFF.
We are mainly interested in the local geometry of LQG surfaces, so it will not matter for us which particular choice of $h$ we consider.
For convenience we will most often consider the whole-plane GFF.

Fix $Q > 0$ and let $\ccM = 25-6Q^2 \in (1,\infty)$ be the corresponding matter central charge.
Also let $h$ be a GFF on $\mcl D$ plus a continuous function, as above. 
We will define a dyadic tiling associated with $h$ which we will think of as a graph approximation to an LQG surface with matter central charge $\ccM$.

\begin{defn}[Notation for squares] \label{def-square-notation}
For a square $S\subset \BB C$, we write $|S|$ for its side length and $v_S$ for its center. 
We also write $\BB S = [0,1]^2$ for the unit square. 
We say that a closed square with sides parallel to the coordinate axes is \emph{dyadic} if for some $n\in\BB Z$, $|S| = 2^{-n}$ and the corners of $S$ lie in $2^{-n}\BB Z^2$. 
\end{defn}

We will only ever consider closed squares with sides parallel to the coordinate axes, i.e., sets of the form  $[a,a+s]\times[b,b+s]$ for $a,b\in\BB R$ and $s>0$.

For a square $S \subset \mcl D $, we define
\eqb \label{eqn-mass-def}
M_h(S) := e^{h_{|S|/2}(v_S)} |S|^Q ,
\eqe 
where here $h_r(z)$ for $z\in\BB C$ and $r>0$ is the average of $h$ over the circle $\bdy B_r(z)$ (see~\cite[Section 3.1]{shef-kpz} for basic properties of the circle average).
In the case when $Q  =2/\gamma + \gamma/2 > 2$, the quantity $(M_h(S))^\gamma$ is a good approximation for the $\gamma$-LQG area of $S$ (in a sense which is made precise in~\cite{shef-kpz}, see e.g.\ \cite[Lemmas 4.6 and 4.7]{shef-kpz}).
In general, we think of $M_h(S)$ as representing the ``LQG size" of $S$.

For a set $U\subset\mcl D$ and $\ep  > 0$, let
\allb \label{eqn-dyadic-tiling-def}
\mcl S_h^\ep(U) &:= \big\{ \text{dyadic squares $S\subset U$ with $M_h(S) \leq \ep$ and $M_h(S') > \ep$,} \notag \\
&\qquad \qquad \qquad \text{$\forall$ dyadic ancestors $S' \subset U$ of $S$} \big\}.
\alle
We abbreviate $\mcl S^\ep_h(\BB C) =: \mcl S_h^\ep$. See Figure~\ref{fig-singularity} for an illustration. 
We view $\mcl S_h^\ep(U)$ as a planar map with two squares in $\mcl S_h^\ep(U)$ considered to be adjacent if they intersect along a non-trivial connected line segment (so intersections at a single corner do not count). Then $\mcl S_h^\ep$ for $\ccM \leq 1$ ($Q \geq 2$) consists of dyadic squares with LQG mass approximately $\ep^\gamma$, so should (at least heuristically) behave like a planar map in the matter central charge $\ccM$ universality class; see Remark~\ref{remark-c<1-ball-growth} for further discussion of this point.

For $z,w\in U$, we let $D_h^\ep(z,w ; U)$ be the minimal $\mcl S_h^\ep(U)$-graph distance from a square which contains $z$ to a square which contains $w$ (by convention, this infimum is $\infty$ if either $z$ or $w$ is not contained in a square of $\mcl S^\ep_h(U)$). In other words,
$$
D_h^\ep(z,w ; U) = \inf_{P : z\rta w} \# P ,
$$
where we take the infimum over all paths $P=(S_0,\dots,S_{\#P})$ with $S_j\in\mcl S^\ep_h(U)$, $S_j$ is adjacent to $S_{j-1}$ for $j=1,\dots,\#P$, $z\in S_0$, and $w\in S_{\#P}$. 
Note that this is only a pseudometric, not a true metric (i.e., distance function), since the $D_h^\ep$-distance between two points in the same square is zero.
For $A,B\subset U$, we write
\eqb \label{eqn-dist-set}
D_h^\ep(A,B ; U ) = \inf_{z \in A} \inf_{w\in B} D_h^\ep(z,w; U) .
\eqe 
In the case when $\mcl D = U = \BB C$, we abbreviate $D_h^\ep(\cdot,\cdot ;\BB C) =: D_h^\ep(\cdot,\cdot)$. 

We make some trivial observations about the above definitions. If $U\subset V$, then for each $z\in U$ the square of $\mcl S_h^\ep(U)$ containing $z$ is contained in the square of $\mcl S_h^\ep(V)$ containing $z$, with equality if and only if the latter square is contained in $U$. This implies in particular that $D_h^\ep(z,w ; U) \geq D_h^\ep(z,w ; V)$ for all $z,w\in U$. Furthermore, if $\ep' \in (0,\ep)$, then each square in $\mcl S_h^{\ep'}(U)$ is contained in a square of $\mcl S_h^\ep(U)$, so $D_h^\ep(z,w ; U) \leq D_h^{\ep'}(z,w ; U)$ for all $z,w\in U$. The distance $D_h^\ep$ also satisfy a scaling property: if $C>0$ is dyadic and $f : U\rta \BB R$, then 
\eqb \label{eqn-dist-scaling}
D_{h+f}^\ep(z,w ; U) \leq D_{h(\cdot/C)}^{T \ep}(Cz , Cw ; C U) \quad \text{for} \quad T := C^Q \exp\left( - \sup_{z\in U} f(z) \right) ,\quad\forall z,w\in U. 
\eqe

The key difference between the behavior of $\mcl S_h^\ep$ in the regime $Q \geq 2$ and the regime $Q \in (0,2)$ is that it is locally finite in the former regime but not in the latter.

\begin{defn} \label{def-singularity}
We say that $z\in U$ is a \emph{singularity} of $\mcl S_h^\ep(U)$ if $z$ is in the interior of $U$ and is not contained in any square of $\mcl S_h^\ep(U)$. 
\end{defn}

\begin{figure}[t!]
 \begin{center}
\includegraphics[scale=1]{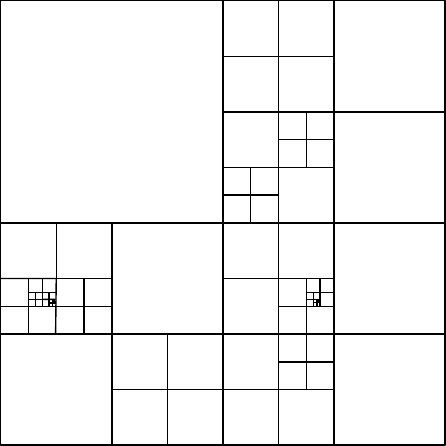}
\vspace{-0.01\textheight}
\caption{Possible realization of the set $\mcl S_h^\ep(\BB S)$ for $\ccM \in (1,25)$ consisting of maximal dyadic squares $S\subset\BB S$ with $M_h(S) \leq \ep$. Singularities are points $z$ which are not contained in any square $S\subset\BB S$ for which $M_h(S) \leq \ep$, i.e., places where we subdivide forever. Such points exist for $\ccM \in (1,25)$ with probability converging to 1 as $\ep\rta 0$ but not for $\ccM  < 1$.  
}\label{fig-singularity}
\end{center}
\vspace{-1em}
\end{figure} 

We observe that if $z\in U$ is fixed, then a.s.\ $z$ is not a singularity of $\mcl S_h^\ep(U)$. 
Indeed, if $h$ is a whole-plane GFF normalized so that $h_1(0) =0$, then since each $h_{|S|/2}(v_S)$ is Gaussian with variance $\log(2/|S|) + O(1)$, the desired statement is easily seen from the Gaussian tail bound and a union bound over the dyadic squares contained in $U$ which contain $z$. 
The corresponding statement for other variants of the GFF follows by local absolute continuity. 
From this, it is easily seen that a.s.\ every singularity is an accumulation point of arbitrarily small squares of $\mcl S_h^\ep(U)$.

Singularities are closely related to $\alpha$-thick points of $h$ for $\alpha >  Q$.

\begin{defn}[Thick points] \label{def-thick-pt}
For $\alpha \in \BB R$, a point $z \in \BB C$ is called an \emph{$\alpha$-thick point of $h$} if the following convergence holds:
\begin{equation}
\lim_{\delta \rightarrow 0} \frac{h_{\delta}(z)}{\log \delta^{-1}} = \alpha.
\label{thick_points_def_eq}
\end{equation}
\end{defn}

It is shown in~\cite{hmp-thick-pts} that a.s.\ the Hausdorff dimension of the set of $\alpha$-thick points is $2-\alpha^2/2$ if $\alpha \in [-2,2]$ and a.s.\ the set of thick points is empty if $|\alpha| > 2$.  
Using this, it is easy to see that if $Q <2$, then with probability tending to 1 as $\ep\rta 0$ there are uncountably many singularities of $\mcl S_h^\ep(U)$ and if $Q\geq 2$, then a.s.\ there are no singularities. Roughly speaking, the reason for this is as follows (a much more general version of this statement is contained in Theorem~\ref{thm-kpz} below). We say that a random point $z \in \BB C$ is a \emph{typical $\alpha$-thick point} if the conditional law of $z$ given $h$ is absolutely continuous with respect to the measure $\mu_h^\alpha$.
Near a typical $\alpha$-thick point $z$, the field behaves like $h^\alpha := h - \alpha\log|\cdot  -z|$, where $h$ is a GFF (this observation dates back to Kahane~\cite{kahane}; see, e.g.,~\cite[Section 3.3]{shef-kpz} for the precise statement we are using here). 
If we let $\{S_n\}_{n\in\BB N}$ be the sequence of dyadic squares of side length less than 1 containing $z$, enumerated so that $|S_n| = 2^{-n}$, then $h_{|S_n|/2}^\alpha(z)$ is Gaussian with variance of order $ \log 2^n$ and mean of order $\alpha \log 2^n$. 
Hence $\log M_{h^\alpha}(S_n)$ behaves like a re-parametrized random walk with drift $ (\alpha -Q) t  $, so a.s.\ drifts to $ \infty$ if $\alpha > Q$ and a.s.\ drifts to $-\infty$ if $\alpha <Q$. 
Thus, if $Q<2$ then $\alpha$-thick points for $\alpha \in (Q,2)$ should give rise to singularities of $\mcl S_h^\ep(U)$ for sufficiently small $\ep$. 

A major motivation for considering $\mcl S_h^\ep$ is the following conjecture, which is motivated by the fact that for $\ccM < 1$, $\mcl S_h^\ep$ behaves like a random planar map in the matter central charge $\ccM$ universality class. See Section~\ref{sec-discussion} and Remark \ref{remark-c<1-ball-growth} for additional context and motivation. 

\begin{conj}[Distance function scaling limit] \label{conj-metric-lim}
For $\ccM  \in (-\infty,25)$, it holds as $\ep\rta 0$ that the graphs $\mcl S_h^\ep$, equipped with their graph distance, converge a.s.\ in the scaling limit to a non-trivial random metric space $(X , \frk d_h)$  with respect to the local Gromov-Hausdorff topology.
\end{conj}

A proof of Conjecture~\ref{conj-metric-lim} would in particular allow us to define LQG with matter central charge $\ccM \in (1,25)$ as a metric space. 
A distance function (metric) on LQG with $\ccM < 1$ has previously been constructed in~\cite{dddf-lfpp,gm-uniqueness}, which should be the Gromov-Hausdorff scaling limit of $\mcl S_h^\ep$ for $\ccM < 1$. Conjecture~\ref{conj-metric-lim} has not been proven even in this case (the construction of the distance function uses a different approximation scheme), but see~\cite{ding-dunlap-lgd} for a proof of tightness for an approximation of LQG distances which is closely related to the graph distance on $\mcl S_h^\ep$.  

For $\ccM \in (1,25)$, we expect that the limiting metric space in Conjecture~\ref{conj-metric-lim} has infinite diameter and infinitely many ends (corresponding to the singularities of $\mcl S_h^\ep$). 
The point-to-point distance in $\mcl S_h^\ep$ grows polynomially (Proposition~\ref{prop-ptwise-distance}) which suggests that the scaling factor for distances should be a power of $\ep$, possibly with a slowly varying correction. 

\subsection{Main results}
\label{sec-results}

To be concrete, throughout this subsection, we assume that $h$ is a whole-plane GFF normalized so that its circle average over $\bdy\BB D$ is zero. 
Our results can be transferred to other variants of the GFF using local absolute continuity considerations and~\eqref{eqn-dist-scaling} (to deal with distributions which differ from the GFF by a continuous function).  

One of the most important features of Liouville quantum gravity is the KPZ formula~\cite{kpz-scaling}, which relates the Euclidean and ``quantum" dimensions of a fractal $X\subset \BB C$.\footnote{The original KPZ formula in~\cite{kpz-scaling} described what the primary fields of the matter field CFT become when they are coupled to quantum gravity. This question seems to be mathematically out of reach so far but mathematicians have proved a weaker formulation~\cite{shef-kpz,rhodes-vargas-log-kpz} which relates the fractal dimension of a set sampled independently from the GFF as measured with the Euclidean metric to the fractal dimension of the same set as measured by the random distance function corresponding to \eqref{eqn-lqg-metric-tensor}. This weaker formulation also goes under the name KPZ formula and is sometimes called the geometric KPZ formula. In this paper we use the terms KPZ formula and geometric KPZ formula interchangeably and we are only concerned with the formula that relates the two notions of dimension for random fractals.} The first rigorous versions of this formula were proven by Duplantier and Sheffield~\cite{shef-kpz} and Rhodes and Vargas~\cite{rhodes-vargas-log-kpz}. Several other versions of the KPZ formula for LQG in the case when $\ccM \leq 1$ are obtained in~\cite{benjamini-schramm-cascades,bjrv-gmt-duality,shef-renormalization,aru-kpz,ghm-kpz,grv-kpz,gwynne-miller-char,wedges,gp-kpz}. Our first main result is an extension of the KPZ formula to the case when $\ccM \in (1,25)$.

\begin{thm}[KPZ formula] \label{thm-kpz}
Let $Q > 0$, equivalently $\ccM \in (-\infty,25)$.  
Let $X\subset \BB C$ be a deterministic or random set which is independent from $h$ and which is a.s.\ contained in some deterministic compact subset of $\BB C$. 
For $\ep \in (0,1)$, let $  N_h^\ep(X)$ be the number of squares of $\mcl S_h^\ep$ which intersect $X$.  
\begin{itemize}
\item (Upper bound) Let $N_0^\delta(X)$ denote the number of dyadic squares of side length $2^{-\lfloor \log_2 \delta^{-1} \rfloor}$ which intersect $X$. 
If $x  < Q^2/2$ and
\eqb \label{eqn-mink-dim}
\limsup_{\delta \rta 0} \frac{\log \BB E\left[ N_0^\delta(X) \right]}{\log \delta^{-1}} \leq x,
\eqe  
then 
\allb \label{eqn-quantum-dim}
&\limsup_{\ep\rta 0} \frac{\log \BB E \left[ N_h^\ep(X) \right]}{\log \ep^{-1}} \leq Q - \sqrt{Q^2 - 2 x} 
\quad \text{and} \quad \notag\\
&\limsup_{\ep\rta 0} \frac{\log  N_h^\ep(X) }{\log \ep^{-1}} \leq Q - \sqrt{Q^2 - 2 x} , \quad a.s.
\alle
\item (Lower bound) If the Hausdorff dimension of $X$ is a.s.\ at least $x \in [0,2]$, then if $x < Q^2/2$, 
\eqb \label{eqn-quantum-dim-as}
\liminf_{\ep\rta 0} \frac{\log   N_h^\ep(X) }{\log \ep^{-1}} \geq Q - \sqrt{Q^2 - 2 x}  ,\quad a.s.\ 
\eqe
and if $x > Q^2/2$, then a.s.\ for each sufficiently small $\ep > 0$, $X$ intersects infinitely many singularities of $\mcl S_h^\ep$ and $N_h^\ep(X) = \infty$. 
\end{itemize}
In particular, if the Minkowski dimension and the Hausdorff dimension of $X$ are each a.s.\ equal to $x\in [0,2]$, then a.s.\ $N_h^\ep(X) = \ep^{- (Q - \sqrt{Q^2-2x}) + o_\ep(1)}$ if $x < Q^2/2$ and a.s.\ $N_h^\ep(X) = \infty$ for each sufficiently small $\ep >0$ if $x > Q^2/2$. 
\end{thm}

Theorem~\ref{thm-kpz} is proven in Section~\ref{sec-kpz} via a short argument based on an analysis of the circle average process of $h$. 
The upper bound for $N_h^\ep(X)$ in Theorem~\ref{thm-kpz} is in terms of the Minkowski expectation dimension of $X$, whereas the lower bound is in terms of the Hausdorff dimension of $X$. We note that if the Hausdorff dimension of $X$ is a.s.\ at least $x$, then $ \liminf_{\delta\rta 0} \log  \BB E[ N_0^\delta(X)]   / \log \delta^{-1} \geq x$.  
Theorem~\ref{thm-kpz} is not true with the Hausdorff dimension of $X$ used in the upper bound and/or the Minkowski expectation dimension of $X$ used in the lower bound; see Remark~\ref{remark-not-true}. 
However, for many interesting fractals, including SLE curves~\cite{beffara-dim,lawler-rezai-nat}, the Hausdorff and Minkowski dimensions are known to be equal.

Note that we do not treat the case when $x = Q^2/2$. This case is more delicate and we expect that even if the Minkowski dimension and Hausdorff dimension of $X$ are each equal to $Q^2/2$, $\lim_{\ep\rta 0} N_h^\ep(X) / \log\ep^{-1}$ can be either finite or infinite depending on the rate of convergence in~\eqref{eqn-mink-dim} and the minimal ``gauge function" with respect to which $X$ has finite Hausdorff content.

Let us now explain how Theorem~\ref{thm-kpz} can be viewed as an extension of the KPZ formula for $\ccM   \leq 1$.
Recall that for $\ccM = 25- 6(2/\gamma + \gamma/2)^2  > 25$ and a dyadic square $S$, the quantity $(M_h(S))^\gamma$ is a good approximation for the $\gamma$-LQG mass of $S$. 
In other words, squares of $\mcl S_h^\ep$ typically have $\gamma$-LQG mass approximately $\ep^{\gamma}$. 
Consequently, the KPZ formula (e.g., in the form of~\cite[Corollary~1.7]{shef-kpz}) shows that for $Q >2$, a fractal $X$ satisfying~\eqref{eqn-mink-dim} should also satisfy
\eqb \label{eqn-shef-kpz}
\lim_{\ep\rta 0} \frac{\log N_h^\ep(X)}{\log \ep^{-1}}  = \gamma (1- \Delta) \quad \text{where $\Delta$ solves} \quad 1 - \frac{x}{2} = \frac{\gamma^2}{4} \Delta^2  + \left(1 - \frac{\gamma^2}{4} \right) \Delta . 
\eqe 
Note that our $x$ corresponds to $2-2x$ in the notation of~\cite{shef-kpz}. Expressing~\eqref{eqn-shef-kpz} in terms of $Q$ gives~\eqref{eqn-quantum-dim} for $Q > 2$. 
On the other hand, for $Q \in (0,2)$ the quantity $Q^2 - \sqrt{Q^2 -2x}$ is real if and only if $x \leq Q^2/2$. Hence Theorem~\ref{thm-kpz} says that the ``quantum dimension" of $X$ is given by analytic continuation of the KPZ formula whenever this analytic continuation gives a real answer, and is infinite otherwise. 

As mentioned above, it has recently been proved that an LQG surface with $\ccM  < 1$ has a distance function (i.e., a metric). The Hausdorff dimension $d_{\ccM}$ of the resulting metric space is unknown except in the case $\ccM = 0$, where it is equal to 4. There are several exponents defined in terms of various approximations of LQG or in terms of the continuum LQG  distance function itself which can be expressed in terms of $d_{\ccM}$~\cite{ghs-map-dist,dzz-heat-kernel,dg-lqg-dim,lqg-metric-estimates,gp-kpz,gwynne-ball-bdy}. See~\cite{dg-lqg-dim,ghs-map-dist,ang-discrete-lfpp} for upper and lower bounds for $d_{\ccM}$.

For our purposes, the most relevant of these exponents is the ball volume exponent for a random planar map $M$ in the matter central charge $\ccM$ universality class: for many such planar maps $M$, it is shown in~\cite[Theorem 1.6]{dg-lqg-dim} that the graph distance ball of radius $r \in \BB N$ centered at the root vertex of $M$ typically contains order $r^{d_{\ccM}}$ vertices.\footnote{One reason why it is natural for this ball volume exponent to coincide with the dimension of LQG is that the Minkowski dimension $d$ of a metric space can be defined by the condition that the number of metric balls of radius $\delta >0$ needed to cover a metric ball of radius 1 is of order $\delta^{-d}$. For $M$, the number of graph distance balls of radius 1 (i.e., singleton sets of vertices) needed to cover the ball of radius $r$ is its cardinality.}
We will show that for our model, the ball growth exponent is infinite when $\ccM \in (1,25)$, which suggests that the LQG distance function in this regime (if it can be shown to exist) should have infinite Hausdorff dimension. 

In the statement of the next theorem and in what follows, for $r\in\BB N$ we let $\mcl B_r^{\mcl S^1_h}(0)$ be the graph-distance ball in the dyadic tiling $\mcl S^1_h$ centered at 0, i.e., the set of all squares of $\mcl S_h^1$ lying at $D_{ h}^1$-distance at most $r$ from an origin-containing square. 
We write $ \#\mcl B_r^{\mcl S^1_h}(0)$ for the cardinality of this ball. 

\begin{thm}[Superpolynomial ball volume growth] \label{thm-ball-infty}
Let $\ccM \in (1,25)$, equivalently $Q\in (0,2)$.  
Almost surely,
\eqb \label{eqn-ball-infty}
\lim_{r\rta\infty} \frac{\log \#\mcl B_r^{\mcl S^1_h}(0)}{ \log r } = \infty. 
\eqe 
\end{thm} 

The proof of Theorem~\ref{thm-ball-infty}, which is carried out in Section~\ref{sec-ball-infty}, is the most involved part of the paper.
In order to establish the theorem, we will need to prove a number of estimates for distances in $\mcl S_h^\ep$ which are of independent interest.
See the beginning of Section~\ref{sec-ball-infty} for an outline of the proof and the estimates involved. 

Theorem~\ref{thm-ball-infty} is consistent with the general principle that whenever we analytically continue a formula for $\ccM < 1$ and get a complex answer, the corresponding quantity for $\ccM \in (1,25)$ should be degenerate. Indeed, even though we do not know the dimension $d_{\ccM}$ for $\ccM  < 1$, there are several predictions which appear to have some degree of validity. For example, Watabiki~\cite{watabiki-lqg} predicted that for $\ccM < 1$, 
\eqb \label{eqn-watabiki}
d_{\ccM} = 2\frac{\sqrt{49-\ccM}  + \sqrt{1-\ccM}}{\sqrt{25-\ccM} + \sqrt{1-\ccM}}  .
\eqe
This formula is known to be false when $\ccM$ is very negative~\cite{ding-goswami-watabiki}, but it appears to match up reasonably well with numerical simulations~\cite{ambjorn-budd-lqg-geodesic}. Alternatively,~\cite[Equation (1.16)]{dg-lqg-dim} proposes an alternative formula for $d_{\ccM}$ when $\ccM  < 1$, namely
\eqb \label{eqn-dg-guess}
d_\ccM = \gamma Q + \frac{\gamma}{\sqrt 6} =  \frac16 \left(25-\ccM  -  \sqrt{26 + 2\sqrt{( 25-\ccM ) (1-\ccM)} - 2\ccM} + \sqrt{( 25-\ccM ) (1-\ccM)}  \right) .
\eqe
Recent numerical simulations~\cite{bb-lqg-dim} fit much better with~\eqref{eqn-dg-guess} than with~\eqref{eqn-watabiki}; and unlike~\eqref{eqn-watabiki} the formula~\eqref{eqn-dg-guess} is consistent with all rigorously known bounds. However, there is currently no theoretical justification for~\eqref{eqn-dg-guess} (even at a heuristic level). 
If we take $\ccM \in (1,25)$ in either~\eqref{eqn-watabiki},~\eqref{eqn-dg-guess}, or in the upper or lower bounds from~\cite{dg-lqg-dim}, we get a complex value for $d_{\ccM}$, which is consistent with the infinite ball growth exponent in Theorem~\ref{thm-ball-infty}. 

\begin{remark} \label{remark-c<1-ball-growth}
When $\ccM < 1$, it is not hard to show, using similar arguments to those in~\cite[Section 3.1]{dzz-heat-kernel} and~\cite[Section 3.2]{dg-lqg-dim}, that a.s.\ $\#\mcl B_r^{\mcl S_h^1}(0) = r^{d_{\ccM} + o_r(1)}$. We will not carry this out here since our main interest is in the case when $\ccM \in (1,25)$ but we remark that this provides one indication that the planar map $\mcl S_h^\ep$ behaves like a random planar map in the matter central charge $\ccM$ universality class. 
	
Another indication that $\mcl S_h^\ep$ is for $\ccM < 1$ is related to LQG is provided by \cite{gms-tutte}. This work studies a discretization of an LQG surface with matter central charge $\ccM < 1$ which is closely related to $\mcl S_h^\ep$, except that instead of subdividing into squares the authors divide into sets of the form $\eta([x-\ep,x])$ for $x\in\ep\BB Z$, where $\eta$ a space-filling SLE curve parametrized so that it traverses one unit of LQG mass in one unit of time. Using a general theorem from~\cite{gms-random-walk}, it is proven that the counting measure on the vertices of the random planar map defined by this subdivision converges to the area measure on the original LQG surface under the so-called Tutte embedding (a.k.a.\ the harmonic embedding or the barycentric embedding). Again, we expect that similar arguments to those in~\cite{gms-tutte} can be used to show that the counting measure on vertices of $\mcl S_h^\ep$ under the Tutte embedding converges to the LQG area measure when $\ccM < 1$, but we do not carry this out here.
\end{remark}

In contrast to Theorem~\ref{thm-ball-infty}, the $D_h^\ep$-distance between two typical points of $\BB C$ a.s.\ grows polynomially in $\ep$. 

\begin{prop}[Point-to-point distances grow polynomially] \label{prop-ptwise-distance}
Let $\ccM \in (-\infty,25)$, equivalently $Q > 0$. There are finite positive numbers $\ul\xi  ,\ol \xi  > 0$, depending on $\ccM$, such that for each fixed distinct $z,w\in\BB C$, a.s.\
\eqb \label{eqn-ptwise-distance}
\ep^{-\ul\xi + o_\ep(1)} \leq  D_h^\ep(z,w)   \leq  \ep^{-\ol\xi -  o_\ep(1)} \quad \text{as $\ep\rta 0$}  .
\eqe   
\end{prop}
 
The lower bound in Proposition~\ref{prop-ptwise-distance} (with $\ul \xi = 1/(2+Q)$) is essentially obvious --- it follows since a basic Gaussian tail estimate shows that the maximal side length of the squares of $\mcl S_h^\ep$ which intersect any fixed compact set is at most $\ep^{ 1/(2+Q) + o_\ep(1)}$. 
The upper bound in the case when $\ccM < 13$ ($Q > \sqrt 2$) with $\ol \xi = Q - \sqrt{Q^2-2 }$ is an immediate consequence of Theorem~\ref{thm-kpz} applied with $X$ equal to a straight line from $z$ to $w$. 
The upper bound in the case when $\ccM \in [13,25)$ takes a little more thought.
Indeed, in this case the number of squares of $\mcl S_h^\ep$ which intersect any deterministic continuous path is a.s.\ infinite for small enough $\ep$, so we need to consider random paths. We will look at a path of squares which follows a ``level line" of the GFF in the sense of~\cite{ss-contour}, which is an SLE$_4$ curve coupled with the field in such a way that the field values along the curve are of constant order. See Lemma~\ref{lem-dist-along-line}. 
We do not emphasize the particular values of $\ul\xi$ and $\ol \xi$ in Proposition~\ref{prop-ptwise-distance} since we expect that these bounds are far from optimal.

We expect that one can show that in fact there exists an exponent $\xi  = \xi(\ccM) >0$ such that for any fixed $z,w\in\BB C$, a.s.\ $D_h^\ep(z,w) = \ep^{\xi + o_\ep(1)}$, using similar techniques to those in~\cite{dzz-heat-kernel} (which proves the existence of such an exponent for several similar models in the case when $\ccM < 1$). We will not carry this out here, however.

We will now explain how Proposition~\ref{prop-ptwise-distance} is related to analytic continuation. 
In the case when $\ccM < 1$, it follows from results in~\cite{dzz-heat-kernel,dg-lqg-dim} that for fixed distinct points $z,w\in\BB C$, a.s.\ 
\eqb \label{eqn-c<1-dist}
D_h^\ep(z,w) = \ep^{-\gamma/d_{\ccM} + o_\ep(1)} , 
\eqe 
where $d_{\ccM}$ is the Hausdorff dimension LQG with matter central charge $\ccM$, as above, and $\gamma  = \gamma(\ccM) \in (0,2)$ is the coupling constant.\footnote{To be more precise,~\cite[Section 5]{dzz-heat-kernel} shows that~\eqref{eqn-c<1-dist} holds for the variant of $D_h^\ep$ where we replace the circle average by the truncated white-noise decomposition of the field (here we note that $\chi = 2/d_{\ccM}$ in the notation of~\cite{dzz-heat-kernel}, see~\cite{dg-lqg-dim}; and that the parameter $\ep$ in~\cite{dzz-heat-kernel} corresponds to $\ep^{2/\gamma}$ in our setting). 
One can compare this variant of $D_h^\ep$ to $D_h^\ep$ itself using Lemma~\ref{lem-tr-compare-square} below. }
If we plug in a reasonable guess for $d_{\ccM}$ (e.g.,~\eqref{eqn-watabiki} or~\eqref{eqn-dg-guess}) then $d_{\ccM}$ and $\gamma$ are both complex for $\ccM \in (1,25)$, but the ratio $\gamma/d_{\ccM}$ is real. It is natural to guess that the (still unknown) formula for $\gamma/d_{\ccM}$ analytically continues to the case when $\ccM \in (1,25)$ and the relation~\eqref{eqn-c<1-dist} remains valid in this regime (but we would not go so far as to make this a conjecture). This is consistent with the polynomial growth of point-to-point distances observed in Proposition~\ref{prop-ptwise-distance}.

\subsection{Basic notation}
\label{sec-basic}

\noindent
We write $\BB N = \{1,2,3,\dots\}$ and $\BB N_0 = \BB N \cup \{0\}$. 
\medskip

\noindent
For $a < b$, we define the discrete interval $[a,b]_{\BB Z}:= [a,b]\cap\BB Z$. 
\medskip

\noindent
If $f  :(0,\infty) \rta \BB R$ and $g : (0,\infty) \rta (0,\infty)$, we say that $f(\ep) = O_\ep(g(\ep))$ (resp.\ $f(\ep) = o_\ep(g(\ep))$) as $\ep\rta 0$ if $f(\ep)/g(\ep)$ remains bounded (resp.\ tends to zero) as $\ep\rta 0$. We similarly define $O(\cdot)$ and $o(\cdot)$ errors as a parameter goes to infinity. 
\medskip

\noindent
If $f,g : (0,\infty) \rta [0,\infty)$, we say that $f(\ep) \preceq g(\ep)$ if there is a constant $C>0$ (independent from $\ep$ and possibly from other parameters of interest) such that $f(\ep) \leq  C g(\ep)$. We write $f(\ep) \asymp g(\ep)$ if $f(\ep) \preceq g(\ep)$ and $g(\ep) \preceq f(\ep)$. 
\medskip

\noindent
Let $\{E^\ep\}_{\ep>0}$ be a one-parameter family of events. We say that $E^\ep$ occurs with
\begin{itemize}
\item \emph{polynomially high probability} as $\ep\rta 0$ if there is a $p > 0$ (independent from $\ep$ and possibly from other parameters of interest) such that  $\BB P[E^\ep] \geq 1 - O_\ep(\ep^p)$. 
\item \emph{superpolynomially high probability} as $\ep\rta 0$ if $\BB P[E^\ep] \geq 1 - O_\ep(\ep^p)$ for every $p>0$. 
\item \emph{exponentially high probability} as $\ep\rta 0$ if there exists $\lambda >0$ (independent from $\ep$ and possibly from other parameters of interest) $\BB P[E^\ep] \geq 1 - O_\ep(e^{-\lambda/\ep})$. 
\end{itemize}
We similarly define events which occur with polynomially, superpolynomially, and exponentially high probability as a parameter tends to $\infty$. 
\medskip

\noindent
We will often specify any requirements on the dependencies on rates of convergence in $O(\cdot)$ and $o(\cdot)$ errors, implicit constants in $\preceq$, etc., in the statements of lemmas/propositions/theorems, in which case we implicitly require that errors, implicit constants, etc., appearing in the proof satisfy the same dependencies.

\subsubsection*{Acknowledgments}

We are grateful to several individuals for helpful discussions, including Timothy Budd, Jian Ding, Bertrand Duplantier, Antti Kupiainen, Greg Lawler, Eveliina Peltola, R\'emi Rhodes, Scott Sheffield, Xin Sun, and Vincent Vargas.  
We thank Scott Sheffield for suggesting the idea of using square subdivisions to approximate LQG for $\ccM   \in (1,25)$.
We also thank the anonymous referee for numerous helpful suggestions and comments.
E.G.\ was partially supported by a Herchel Smith fellowship and a Trinity College junior research fellowship.
N.H.\ was supported by Dr.\ Max R\"ossler, the Walter Haefner Foundation, and the ETH Z\"urich Foundation.
G.R.\ was partially supported by a National Science Foundation mathematical sciences postdoctoral research fellowship.
J.P.\ was partially supported
by the National Science Foundation Graduate Research Fellowship under Grant No.\ 1122374.

\section{Further discussion concerning LQG with matter central charge $\ccM \in (1,25)$}
\label{sec-discussion}

This section includes additional discussion concerning the meaning of ``LQG with matter central charge $\ccM \in (1,25)$", and the relationship between this work and other parts of the mathematics and physics literature. Nothing in this section is needed to understand the proofs of our main results.

\subsection{Liouville quantum gravity surfaces for $\ccM < 25$}
\label{sec-lqg-surface}
Since $Q > 0$ for $\ccM < 25$, one may extend the formal definition of a Liouville quantum gravity surface for $\ccM \leq 1$ from~\cite{shef-kpz,shef-zipper,wedges} verbatim to the case when $\ccM \in (1,25)$. 

\begin{defn} \label{def-lqg-surface}
A \emph{Liouville quantum gravity surface} with matter central charge $\ccM  < 25$ is an equivalence class of pairs $(\mcl D , h)$ where $\mcl D\subset \BB C$ is an open domain, $h$ is a distribution (generalized function) on $\mcl D$ (which we will always take to be some variant of the GFF), and two such pairs $(\mcl D , h)$ and $(\wt{\mcl D} , \wt h)$ are declared to be equivalent if there is a conformal map $f : \wt{\mcl D} \rta \mcl D$ such that $\wt h = h\circ f + Q\log |f'| $. 
\end{defn}

We think of two equivalent pairs as being two different parametrizations of the same LQG surface. In the case when $\ccM \leq 1$, one of the major motivations for the above definition of LQG surfaces is that the $\gamma$-LQG area measure is preserved under coordinate changes, i.e., a.s.\ $\mu_{h\circ f +Q\log |f'|}^\gamma(f^{-1}(A)) = \mu_h^\gamma(A)$ for each $A\subset \mcl D$~\cite[Proposition 2.1]{shef-kpz} (see~\cite[Theorem 13]{shef-renormalization} for the case $\ccM = 1$).  
It is shown in~\cite{gm-coord-change} that for $\ccM < 1$, the LQG distance function is conformally covariant in a similar sense.

See Section~\ref{sec-action} for a discussion about the relationship between the coordinate change relation of Definition~\ref{def-lqg-surface} and the Liouville action. We emphasize that, unlike in the case when $\ccM \leq 1$, we do not have a definitive heuristic derivation of the coordinate change formula from the Liouville action in the case when $\ccM > 1$. Nevertheless, as we discuss below, the coordinate change formula for $\ccM > 1$ is consistent with our proposed discrete model of LQG. 

\subsubsection{Relationship to the square subdivision model}
\label{sec-subdivision-coord}

The discretization considered in Section~\ref{sec-tiling-def} is not exactly preserved under coordinate changes of the form considered in Definition~\ref{def-lqg-surface} (i.e., we do not have $D_{\wt h}^\ep(f^{-1}(z) , f^{-1}(w)) = D_h^\ep(z,w)$), but it is approximately preserved in the following sense. With $M_h(S)$ as in~\eqref{eqn-mass-def} and $C>0$, one has $M_h(S) = C^Q M_{h(\cdot/C)}(C S)$. Therefore, if $C$ is dyadic (so that $C  S$ is a dyadic square whenever $S$ is a dyadic square), one has $D_{h(\cdot/C) + Q\log(1/C)}^\ep(C z , C w) = D_h^\ep(z,w)$ for each $z,w \in \mcl D$ and each $\ep > 0$.   
This suggests that subsequential limits of $D_h^\ep$ (if they can be shown to exist) should be covariant with respect to spatial scaling (at least by dyadic scaling factors).

Although the law of $\mcl S_h^\ep$ itself is not rotationally invariant (except by angles which are integer multiples of $\pi/2$), we expect that geodesics of subsequential limits of $D_h^\ep$ should have no local notion of direction. This should allow one to show that the limiting distance function is also rotationally invariant.
Indeed, it is shown in~\cite[Corollary 1.3]{gm-uniqueness} (see also~\cite[Remark 1.6]{gm-uniqueness}) that a distance function associated with $\gamma$-LQG for $\gamma \in (0,2)$ must be rotationally invariant once it is known to be scale invariant and to satisfy a few other natural axioms.
We expect a similar situation for LQG with central charge in $(1,25)$. 
Once one has scale and rotational covariance, one can deduce conformal covariance using that general conformal maps locally look like complex affine maps.
See~\cite{gm-coord-change} for a proof that the $\gamma$-LQG distance function for $\gamma\in (0,2)$ is conformally covariant starting from the fact that it is covariant with respect to scaling, translation, and rotation.

\subsubsection{LQG measures on lower-dimensional fractals}
\label{sec-low-dim-fractal}

Our model suggests that for $\ccM \in (1,25)$, there is not a natural locally finite measure on $\BB C$ corresponding to LQG with matter central charge $\ccM$ (see Section~\ref{sec-open-problems} for some related open problems). Indeed, this is because for any fixed open set $V\subset\BB C$, the set $V$ will a.s.\ contain a singularity of $\mcl S_h^\ep$  --- and hence will contain infinitely many squares of $\mcl S_h^\ep$ --- for small enough $\ep > 0$ since the set of $\alpha$-thick points with $\alpha\in(Q,2)$ is dense. 
However, one can define LQG measures on fractals of dimension less than 2 which are invariant under coordinate changes, as we now explain. 

Let $X\subset\BB C$ and suppose that for some $x \in(0,2]$ the $x$-dimensional Minkowski content of $X$ is well-defined and defines a locally finite and non-trivial measure $\frk m$ that has finite $x'$-dimensional energy for any $x'\in(0,x)$, i.e.,  
\eqbn
\int_X \int_X  |z-w|^{-x'} \frk m(dz)\frk m(dw)<\infty.
\eqen
An example of a fractal satisfying this property is an SLE$_\kappa$ curve for $\kappa \in (0,8)$; see~\cite{lawler-rezai-nat}.  

For any such fractal $X$ which is either deterministic or independent of the field $h$ and any $\gamma'<\sqrt{2 x}$ 
one may define the Gaussian multiplicative chaos measure $\mu_{h,\frk m}^{\gamma'}=e^{\gamma' h}\,d\frk m$  
by a similar regularization procedure as for the area measure, see e.g.\ \cite{rhodes-vargas-review,berestycki-gmt-elementary}. The following proposition says that the resulting LQG measure is invariant under coordinate changes with 
\eqb
	Q=x/\gamma'+\gamma'/2. 
\label{eq:Q-d-gamma}
\eqe
Therefore, one can think of $\mu_{h,\frk m}^{\gamma'}$ as the LQG measure on $X$ corresponding to LQG with background charge $Q$.

\begin{prop}
	Let $\ccM  < 25$. Let $X$, $\frk m$, and $x\in(0, Q^2/2 )\cap(0,2]$ be as above, and choose $\gamma'$ such that \eqref{eq:Q-d-gamma} is satisfied. Let $\mcl D\subset\BB C$, $\wt{\mcl D}\subset\BB C$, $f:\wt{\mcl D}\to\mcl D$, and $\wt h=h\circ f +Q\log |f'|$ be as in Definition \ref{def-lqg-surface}, and let $\wt{\frk m}$ denote the Minkowski content of $f^{-1}(X)$. Almost surely, for every Borel set $A\subset \mcl D$,
	\eqb
	\mu_{\wt h,\wt{\frk m}}^{\gamma'}(f^{-1}(A)) = \mu_{h,\frk m}^{\gamma'}(A).
	\label{eq:coordinate-change}
	\eqe
	\label{prop:coordinatechange}
\end{prop}
The proposition illustrates that our definition of an LQG surface with matter central charge $\ccM  \in (1,25)$ (Definition \ref{def-lqg-surface}) is a natural extension of LQG for central charge $\ccM \leq 1$ since it can be equipped with natural coordinate-invariant measures.
Note that given $Q\in(0,\infty)$, there exists $\gamma'\in(0,\sqrt{2x})$ satisfying \eqref{eq:Q-d-gamma} if and only if $x\in (0,Q^2/2)$ and $x\leq 2$. See \cite[Proposition 2.1]{shef-kpz} for a proof of the proposition in the case where $\frk m$ is Lebesgue area measure.
For $Q\in (0,2)$ (equivalently, $\ccM \in (1,25)$), there does not seem to be a ``canonical" choice of the random fractal $X$ so unlike in the case when $\ccM \leq 1$ it is not clear how to use the measures $\mu_{h,\frk m}^{\gamma'} $ to make rigorous sense of the partition function for Liouville CFT.

\begin{proof}[Proof of Proposition \ref{prop:coordinatechange}]
	First observe that, by the monotone class theorem, it is sufficient to prove \eqref{eq:coordinate-change} for a fixed set $A$. When applying the monotone class theorem we use that the Borel $\sigma$-algebra on $\mcl D$ is generated by the collection of dyadic rectangles contained in $\mcl D$ (i.e., rectangles with dyadic vertices whose sides are parallel to the coordinate axes) and that the collection of sets $A$ for which \eqref{eq:coordinate-change} holds is closed under monotone intersections and unions by continuity of the measures.
	
	We will prove the result for $h$ an instance of the Gaussian free field in $\mcl D$ since the result for other fields follow by local absolute continuity. For $\phi_1,\phi_2,\dots$ a smooth orthonormal basis for the Dirichlet inner product, we can find a coupling of $h$ and i.i.d.\ standard normal random variables $\alpha_1,\alpha_2,\dots$ such that $h=\alpha_1 \phi_1+\alpha_2\phi_2+\dots$. Define $h^n=\alpha_1 \phi_1+\dots+\alpha_n\phi_n$, and let $\mcl F_n$ be the $\sigma$-algebra generated by $\alpha_1,\dots,\alpha_n$. Recall that for $z\in\mcl D$, the conformal radius of $\mcl D$ viewed from $z$ is defined by $R(z,\mcl D)=|g'(z)|^{-1}$, where $g:\mcl D\to\BB D$ is a conformal map to the unit disk with $g(z)=0$. Define the measure $\mu_{h^n,\frk m}^{\gamma'}$ as follows for $A\subset\mcl D$ 
	$$
	\mu_{h^n,\frk m}^{\gamma'}(A)
	= 
	\int_A e^{ \gamma' h^n(z)-\frac{(\gamma')^2}{2}\BB E[h^n(z)^2] +\frac{(\gamma')^2}{2}\log R(z,\mcl D)}\, \frk m(dz).
	$$
	In the proof of uniqueness of \cite[Theorem 1]{berestycki-gmt-elementary} it is proved that for any fixed $A\subset\mcl D$ we have $\mu_{h^n,\frk m}^{\gamma'}(A)=\BB E[\mu_{h,\frk m}^{\gamma'}(A)\,|\,\mcl F_n ]$. Defining $\wt h^n$ by $\wt h^n=h^n\circ f+Q\log|f'|$, the same argument gives that $\mu_{\wt h^n,\wt{\frk m}}^{\gamma'}(f^{-1}(A))=\BB E\big[\mu_{\wt h,\wt{\frk m}}^{\gamma'}(f^{-1}(A))\,|\,\mcl F_n \big]$. Therefore, to conclude the proof of the proposition it is sufficient to show the following for any $n\in\BB N$
	\eqbn
	\mu_{\wt h^n,\wt{\frk m}}^{\gamma'}(f^{-1}(A))
	=
	\mu_{h^n,\frk m}^{\gamma'}(A).
	\eqen
	This identity is immediate by using the definition of $\mu_{h^n,\frk m}^{\gamma'}(A)$ and $\mu_{\wt h^n,\wt{\frk m}}^{\gamma'}(f^{-1}(A))$, that $\log(f(z),{\mcl D})-\log R(z,\wt{\mcl D})=\log|f'(z)|$, and that for any $U\subset A$ we have $\frk m(U)=\int_{f(U)} |f'(z)|^x\,\wt{\frk m}(dz)$ (by the definition of Minkowski content). In particular, the term $\frac{\gamma'}{2}\log|f'(z)|$ in $Q\log|f'(z)|$ accounts for the change in conformal radius, and the term $\frac{x}{\gamma}\log|f'(z)|$ accounts for the rescaling of the measure $\frk m$.
\end{proof}

For $\ccM \leq 1$ we get a good approximation to the LQG area measure of a set by counting the number of squares in the dyadic tiling $\mcl S_h^\ep$ which intersect the set, since each square in $\mcl S_h^\ep$ has LQG area approximately $\ep^{\gamma}$. We expect that we can use a similar method to approximate the LQG measure of other sets $X\subset\BB C$. In this setting the LQG measure of $X$ should be approximately given by 
\eqb
N_h^\ep(X) \ep^{Q-\sqrt{Q^2-2x}}, 
\label{eq:measure-approx}
\eqe
where $N_h^\ep(X)$ is the number of squares of $\mcl S_h^\ep$ that intersect $X$ and $x$ is the Euclidean dimension of $X$.\footnote{Due to lattice effects of the dyadic tiling we do not expect the approximation \eqref{eq:measure-approx} to converge exactly to the LQG measure of $X$ as defined earlier in this subsection for all fixed choices of $X$. For example, the line segment $[0,1] \times \{0\}$ typically intersects approximately twice as many squares as the line segment $[0,1]\times \{r\}$ for $r\in \BB R\setminus \BB Q$ close to 1, but these two line segments should have approximately the same LQG measure. However, we believe that certain variants of \eqref{eq:measure-approx} do converge to the LQG measure of $X$, e.g.\ we can consider versions of the square subdivision where the set of possible boxes are translated by $z\in [0,1]^2$ and average over $z$.} We do not work this out carefully here, however. Note that the exponent $Q - \sqrt{Q^2-2x}$ comes from Theorem~\ref{thm-kpz}. 
 
\subsubsection{Liouville action}
\label{sec-action}

Recall that $Q = \sqrt{\frac{25-\ccM}{6}}$. 
One has the following formula for the Liouville action, where $\mcl D $ is a compact, boundaryless, surface, $g$ is a Riemannian metric on $\mcl D $ with associated curvature $R_g$, gradient $\partial^g $,  volume form $d \lambda_g$, and $\mu \geq 0$:
\eqbn
S_{\mcl D}(\varphi) := \frac{1}{4\pi} \int_{\mcl D}  \left( |\partial^g \varphi |^2 + QR_g\varphi  + 4 \pi \mu e^{\gamma \varphi} \right)\, d \lambda_g.
\eqen
The definition of the exponential term $ \mu e^{\gamma \varphi} $ poses problem in the regime $\ccM \in (1, 25)$ since $\gamma$ is complex (c.f.\ the discussion in the next section regarding the complex Gaussian multiplicative chaos). This problem is also related to the fact that the term $ \mu e^{\gamma \varphi} $ should be seen in a certain sense as conditioning the volume to be finite, which we are not doing since the number of squares of $\mathcal{S}_h^{\epsilon}$ which intersect a fixed bounded open set is a.s.\ infinite for small enough $\ep$. Therefore constructing LCFT using $S_{\mcl D}(\varphi)$ as performed in \cite{dkrv-lqg-sphere} seems to require some novel ideas.

In the regime $\ccM<1$, where $\gamma \in (0,2)$, there is a well-known argument in the physics literature (see for instance~\cite{zz-dozz}) to justify the coordinate change formula. The idea is the following. Let $\mcl D, \wt{\mcl D} \subset\BB C$ be two domains and $f:\mcl D\to\wt{\mcl D}$ a conformal map. Take on $\wt{\mcl D}$ the metric $\wt{g}$ defined as the push-forward of $g$ by $f$. Let $\wt{\varphi} = \varphi \circ f + Q \log |f'|$. Then one can check by performing an integration by parts that
\eqbn
S_{\wt{\mcl D}}(\wt{\varphi}) = S_{\mcl D}(\varphi)  +  c(\mcl D, g,f),
\eqen
where $c(\mcl D, g,f)$ is an unimportant constant independent of $\varphi$.\footnote{Here at the classical level the exponential term $e^{\gamma \varphi}$ transforms correctly under the coordinate change provided that $Q = \frac{2}{\gamma}$, but in the quantum case when $\varphi$ is a GFF type distribution the correct value of $Q$ is $ \frac{\gamma}{2} + \frac{2}{\gamma}$. See the coordinate change formula \eqref{eq:coordinate-change}.} Heuristically, let $\wt h$ be the distribution on $\wt{\mcl D}$ sampled from  $e^{-S(\wt h) }\,d\wt h$, where $d\wt h$ is ``uniform measure on the space of all functions'', and similarly for $h$ on $D$ using $e^{-S( h) }\,d h$; note that $h$ and $\wt h$ can only be made sense of as GFF-like distributions rather than true functions, and that we need to (for example) fix a global constant for the field in order to get a probability measure. Also notice that the term $c(\mcl D, g,f) $ obtained above will disappear in the normalization of the law of our fields. Therefore requiring that  $e^{-S(\wt h) }\,d\wt h$ and $e^{-S( h) }\,d h$ define the same field is equivalent to imposing the coordinate change formula $\wt{h} = h \circ f + Q \log |f'|$.

The above heuristic derivation of the coordinate change formula cannot be straightforwardly applied to the $\ccM >1$ case because the meaning of $e^{\gamma \varphi}$ becomes unclear. One option to extend to $\ccM >1$ is simply to set $\mu =0$ in the argument of the previous paragraph (as performed in \cite[Section 2]{shef-kpz}) in which case one still lands on the same coordinate change formula. The downside is one no longer sees the relation $Q = \frac{\gamma}{2} + \frac{2}{\gamma}$ coming from the presence of $e^{\gamma \varphi}$. On the other hand, in the $\ccM <1$ case, the field obtained for $\mu>0$ behaves locally like the GFF~\cite{wedges,dkrv-lqg-sphere}. It is not clear whether this absolute continuity should extend to the $\ccM \in (1,25)$ regime, but if we assume it is the case, then, for the purposes of the theorems proved in this paper --- which require $h$ to behave locally like a GFF but do not require any precise information about the global properties of $h$ --- it would not matter whether we work with $\mu =0$ or $\mu>0$.

\subsection{Laplacian determinant and branched polymers}
\label{sec-laplacian}  
 
As alluded to in Section~\ref{sec-overview}, it has been proven that a uniform random planar map or a uniform $q$-angulation  (for $q=3$ or $q \geq 4$ even) with a fixed number of edges converges in law in the Gromov-Hausdorff sense, with the appropriate rescaling, to an LQG surface with matter central charge $\ccM = 0$ as the number of edges tends to infinity~\cite{legall-uniqueness, miermont-brownian-map,bjm-uniform,lqg-tbm1,lqg-tbm2,lqg-tbm3}. Moreover, uniform triangulations also converge to LQG with matter central charge 0 when embedded into the plane via the discrete conformal embedding called the Cardy embedding; see~\cite{hs-cardy-embedding}.
Due to the ``Laplacian determinant" definition of LQG it is natural to conjecture that random planar maps sampled with probability proportional to $(\det \Delta)^{-\ccM/2}$, where $\Delta$ is the discrete Laplacian determinant, converge in some sense to LQG surfaces with matter central charge $\ccM$.  
Note that when $\ccM$ is a positive integer, weighting by $(\det \Delta)^{-\ccM/2}$ is equivalent to coupling the random planar map with $\ccM$ independent discrete Gaussian free fields or ``massless free bosons''.  

Often using random planar maps weighted by $(\det \Delta)^{-\ccM/2}$ as a heuristic approximation of LQG,
many works~\cite{cates-branched-polymer,david-c>1-barrier,adjt-c-ge1,ckr-c-ge1,bh-c-ge1-matrix,dfj-critical-behavior,bj-potts-sim,adf-critical-dimensions} have suggested that an LQG surface for $\ccM > 1$ (or at least for $\ccM  \geq 12$) should behave like a so-called branched polymer. 
 Mathematically, this means that such surfaces should be similar to the continuum random tree (CRT)~\cite{aldous-crt1,aldous-crt2,aldous-crt3}. This prediction is supported by numerical simulations and heuristics (see the above references) which suggest that the scaling limit of  random planar maps  weighted by  $(\det \Delta)^{-\ccM/2}$ should be the CRT.  Thus, the heuristic of approximating LQG with matter central charge $\ccM > 1$ by random planar maps weighted by  $(\det \Delta)^{-\ccM/2}$ yields predictions for the behavior of LQG in the $\ccM>1$ regime that are very different from the behavior of our proposed discrete model $\mcl S^{\ep}_h$.   

Describing the behavior of LQG for $\ccM>1$ heuristically in terms of random planar maps weighted by  $(\det \Delta)^{-\ccM/2}$ does not just conflict with the behavior of our model; it is also rather unsatisfying. Indeed, it would suggest that the large-scale geometry for LQG with matter central charge $\ccM  >1$ is tree-like and does not depend on $\ccM$, and therefore apparently unrelated to conformal field theory models (since trees have no conformal structure). As a result, a number of works in the physics literature have looked for ways to associate a non-trivial geometry with LQG for $\ccM  > 1$ despite the apparent obstacles; see~\cite{ambjorn-remarks} for a survey of some of these works. 

In this subsection, we will discuss, on a heuristic level, the relationship in the $\ccM > 1$ phase between random planar maps with a large number of vertices weighted by  $(\det \Delta)^{-\ccM/2}$ and   $\mcl S^{\ep}_h$.  We will argue heuristically that, the former model corresponds, in some sense, to conditioning the law of the latter model on the unlikely event that the number of squares is large but finite, in which case all of the  nontrivial geometry that the latter model possesses ``disappears''.  Thus, if the latter model that we have proposed indeed converges to a continuum LQG surface with matter central charge $\ccM > 1$, then finite random planar maps weighted by  $(\det \Delta)^{-\ccM/2}$ should not be a good heuristic approximation for LQG in this phase.

Our explanation is partially based on forthcoming work by Ang, Park, Pfeffer, and Sheffield that establishes the first rigorous version of the heuristic relating LQG to weighting random planar maps by the power of a Laplacian determinant. 

Before we explain the connection between their work and the $\ccM > 1$ setting, let us first outline the result of Ang, Park, Pfeffer, and Sheffield in the case $\ccM \leq 1$.  For each $\ccM \leq 1$, they define a random planar map, which we will denote here by $\mcl{R}^{\ep}_{h, \ccM}$, using a square subdivision procedure identical to the subdivision $\mcl S^{\ep}_h$ considered in this paper, except with circle averages replaced by averages of the field over dyadic squares.  Like $\mcl S^{\ep}_h$, the map $\mcl{R}^{\ep}_{h, \ccM}$ should be a good discrete  approximation of an LQG surface with matter central charge $\ccM$ when $\ep$ is small. The map $\mcl{R}^{\ep}_{h, \ccM}$ is constructed in such a way that it is naturally coupled to a smooth approximation of the (heuristic) LQG metric tensor.  What they are able to show is that, if we weight the law of $\mcl{R}^{\ep}_{h, \ccM}$ conditioned on $\{\# \mcl{R}^{\ep}_{h, \ccM} = n\}$ by the $(-\ccM'/2)$-th power of the determinant of the Laplacian of the approximating metric, then the law of the resulting random planar map is that of $\mcl{R}^{\ep}_{h, \ccM+\ccM'}$ conditioned on $\{\# \mcl{R}^{\ep}_{h, \ccM+\ccM'} = n\}$. Thus, we have an explicit family of random planar map approximations of LQG surfaces for $\ccM \leq 1$ such that weighting one member of this family by the appropriate power of the determinant of the Laplacian (defined with respect to an appropriate approximating smooth metric) yields the law of another member of the family. 

Like $\mcl S^{\ep}_h$, the random planar map $\mcl{R}^{\ep}_{h, \ccM}$ can be defined for $\ccM \in (1,25)$ as well. In this case, the same result of Ang, Park, Pfeffer, and Sheffield still holds. It follows that the law of $\mcl{R}^{\ep}_{h, \ccM}$ conditioned on $\{\# \mcl{R}^{\ep}_{h, \ccM} = n\}$ should be in the same universality class (for $n$ sufficiently large) as that of random planar maps with $n$ vertices weighted by  $(\det(-\Delta))^{-\ccM/2}$.  And, just as numerical simulations and heuristics in the physics literature suggest that the latter model converges to a CRT, we expect the law of $\mcl{R}^{\ep}_{h, \ccM}$ conditioned on $\{\# \mcl{R}^{\ep}_{h, \ccM} = n\}$ to converge to that of a CRT as well.  (See Question~\ref{ques-crt} below, where this conjecture is formulated in terms of $\mcl S_h^\ep$.)  The key difference between the $\ccM \in (1,25)$ and $\ccM \leq 1$ regimes is that, when $\ccM \in (1,25)$,  the law of $\mcl{R}^{\ep}_{h, \ccM}$ conditioned on $\{\# \mcl{R}^{\ep}_{h, \ccM} = n\}$ has very different geometric properties from $\mcl{R}^{\ep}_{h, \ccM}$ without this conditioning.  When $\ccM \in (1,25)$, the random planar map $\mcl{R}^{\ep}_{h, \ccM}$ will be infinite with positive probability; indeed, when $\ep  $ is small, $\{\# \mcl{R}^{\ep}_{h, \ccM} = n\}$ is a very unlikely event (i.e., $\mcl{R}^{\ep}_{h, \ccM}$ typically have at least one singularity).  Thus, if we believe that $\mcl{R}^{\ep}_{h, \ccM}$ converges in a metric sense as $\ep \rta 0$ to a continuum model of LQG with matter central charge $\ccM$, then we would expect the geometry of this continuum LQG object to be very different from that of a CRT; e.g., we would expect the limiting metric space to have infinite diameter and infinitely many ends.

This situation is somewhat analogous to the situation for supercritical Galton-Watson trees. If we let $T$ be such a tree, with a geometric offspring distribution, and we condition on $\{\# T = n\}$ then $T$ is uniform over all trees with $n$ vertices and hence behaves like a CRT when $n$ is large. On the other hand, if we do not condition on $\{\# T = n\}$ then $T$ typically looks very different from a CRT, e.g., in the sense that it has exponential ball volume growth and infinitely many ends with positive probability. The law of a random planar map weighted by $\det(-\Delta)^{-\ccM/2}$, with a fixed number of vertices, still depends on $\ccM$, but it is expected that the macroscopic structure does not depend on $\ccM$ when $\ccM > 1$. 

Though it certainly is interesting to have a continuum theory of LQG for $\ccM > 1$ that is nontrivial and depends on $\ccM$, we would ideally also want such an object to arise as the limit of natural random planar map models (and not just the limit of discrete models like $\mcl S^{\ep}_h$ defined in terms of the continuum GFF).  See Question~\ref{ques-discrete}.

\subsection{Related (and not-so-related) models}
\label{sec-related}

\noindent\textbf{Complex Gaussian multiplicative chaos.} 
A natural approach to constructing an ``LQG measure" for matter central charge $\ccM \in (1,25)$ is to consider the corresponding value of $\gamma$, which is complex with $|\gamma|  =2$, and try to make sense of $\exp\left( \gamma h(z) - \frac{\gamma}{2} \op{Var} h(z) \right) d^2 z$  using Gaussian multiplicative chaos (GMC) theory;\footnote{This object may only make sense as a distribution, not a (complex) measure; see~\cite{jsw-imaginary-gmc} for the case of purely imaginary $\gamma$.} see Question~\ref{ques-complex-measure}. 
Complex GMC was first studied in the case where the real and imaginary parts of the field are independent in~\cite{rhodes-vargas-complex-gmt}. The authors obtain a region of $\mathbb{C}$ where the standard renormalization procedure leads to a non-trivial limit. GMC with a purely imaginary value of $\gamma$ was then further studied in~\cite{jsw-imaginary-gmc}. Very recently in~\cite{jsw-decompositions}, a novel decomposition of log-correlated fields allowed the authors to handle the case where real and imaginary parts come from the same field. Unfortunately we don't expect any of these works to be directly related to LQG with matter central charge $\ccM \in (1,25)$ because the case $|\gamma|  =2$ lies outside the feasible region described in~\cite{rhodes-vargas-complex-gmt, jsw-decompositions}. On the other hand in recent work on the sine-Gordon model~\cite{lrv-sine-gordon}, in the case of purely imaginary $\gamma$, a method of renormalization was developed to go beyond the previous region up to a threshold that seems to correspond to $Q=0$ ($\ccM = 25$). However, they do not treat the case of a general complex $\gamma$ with $|\gamma|=2$ and they only consider the one dimensional model so novel ideas are required to handle our case of $\ccM \in (1,25)$ and two dimensions. 
\medskip

\noindent\textbf{Complex weights.} Huang~\cite{huang-complex-insertion} studies the correlation functions of LCFT with real coupling constant $\gamma \in (0,2)$ but with complex weights (equivalently, complex log singularities), building on~\cite{dkrv-lqg-sphere,krv-local,krv-dozz}. This gives an analytic continuation of LCFT in a different direction than the one considered here, but does not deal with the case when $\ccL < 25$ (which corresponds to LQG with $\ccM > 1$).  
\medskip

\noindent\textbf{Atomic LQG measures.} In addition to LQG with matter central charge $\ccM > 1$, there is another natural extension of LQG beyond the $\gamma \in (0,2]$ phase: the purely atomic $\gamma$-LQG measures for $\gamma > 2$ considered, e.g., in~\cite{dup-dual-lqg,bjrv-gmt-duality,rhodes-vargas-review,wedges}.
Such atomic measures are closely related to trees of LQG surfaces with the dual parameter $\gamma' = 4/\gamma \in (0,2)$, where adjacent surfaces touch at a single point, as considered in~\cite{klebanov-touching}. See~\cite[Section 10]{wedges} for further discussion of this point. 
The matter central charge corresponding to $\gamma  >2$ is the same as the matter central charge corresponding to the dual parameter $\gamma' = 4/\gamma$. 
In particular, these atomic measures do \emph{not} correspond to $\ccM > 1$ and should not be confused with the extension of LQG beyond $\gamma\in(0,2]$ considered in this paper.
\medskip
 
\noindent \textbf{Special LQG surfaces and stochastic quantization.} In the case when $\ccM < 1$, there are certain special LQG surfaces (corresponding to special variants of the GFF $h$) which are in some sense canonical. 
These GFF variants for these special LQG surfaces are the focus of Liouville conformal field theory. Examples of such surfaces include the quantum sphere, quantum disk, quantum cone, etc.; see, e.g.,~\cite{dkrv-lqg-sphere,hrv-disk,drv-torus,remy-annulus,grv-higher-genus,wedges} for the definitions of some of these special LQG surfaces.
One way in which these special LQG surfaces are canonical is that they are the ones which arise (or are conjectured to arise) as the scaling limits of natural random planar map models; see the above references for further discussion. 

The definitions of the infinite-volume special LQG surfaces from~\cite[Section 4]{wedges}, namely the quantum cone and quantum wedge, extend immediately to the case when $\ccM  \in (1,25)$ (equivalently, $Q \in (0,2)$). However, the definitions of finite-volume LQG surfaces, such as quantum sphere, disks, and torii, do not have obvious analogs for $\ccM \in (1,25)$ since we do not have a canonical finite volume measure in this case.
 
One possible, but highly speculative, way to construct $\ccM \in (1,25)$-analogs of these surfaces is via \emph{stochastic quantization}. This would involve constructing a canonical Markov chain on the space of random distributions on a given Riemann surface which has a unique stationary distribution.
This stationary distribution should then be the law of the field associated with the canonical LQG surface with the given conformal structure. 
Recent progress on stochastic quantization in the $\ccM \leq 1$ case has been made in works by Garban~\cite{garban-dynamical} and by Dub\'edat and Shen~\cite{ds-ricci-flow}, who studied, respectively, a natural dynamics for the LQG measure and the so-called \emph{stochastic Ricci flow}. 
\medskip
 
\noindent\textbf{Liouville first passage percolation.} Let $h$ be a GFF-type distribution on $\BB C$, let $\{h_\delta : z\in\BB C , \delta>0\}$ be its circle average process, and for $\xi > 0$, $\delta >0$, define a distance function on $\BB C$ by 
\eqbn
D_{h,\LFPP}^{\xi,\delta}(z,w) = \inf_{P : z\rta w} \int_0^1 e^{\xi h_\delta(P(t))} |P'(t)| \,dt , 
\eqen
where the infimum is over all piecewise continuously differentable paths $P$ from $z$ to $w$. 
We call this the \emph{$\delta$-Liouville first passage percolation (LFPP) metric (distance function)} with parameter $\xi$. 
For $\ccM  < 1 $ and corresponding coupling constant $\gamma \in (0,2)$, if we set $\xi = \gamma/d_\ccM$ (where $d_\ccM$ is the Hausdorff dimension of LQG with matter central charge $\ccM$, as above) then it is shown\footnote{For technical reasons,~\cite{dddf-lfpp,gm-uniqueness} define LFPP using the convolution of $h$ with the heat kernel rather than the circle average approximation but the scaling limit of both versions of LFPP should be the same.} in~\cite{dddf-lfpp,gm-uniqueness} that $D_{h,\LFPP}^{\xi,\delta}$, re-scaled appropriately, converges as $\delta\rta 0$ to the LQG distance function for matter central charge $\ccM$. 

We expect that LFPP with $\xi > 2 / d_{\ccM=1}$ should be related to the square subdivision model of the present paper in a manner which is directly analogous to the relationship between LFPP for $\xi \leq 2 / d_{\ccM=1}$ and Liouville graph distance, as established in~\cite[Theorem 1.5]{dg-lqg-dim}. 
In particular, if $Q >0$ and $\xi = \xi(Q) > 0$ is chosen so that, in the notation of Section~\ref{sec-tiling-def}, typically $D_h^\ep(z,w)  =\ep^{-\xi + o_\ep(1)}$ for any two fixed distinct points $z,w\in\BB C$, then it should be the case that typically $D_{h,\LFPP}^{\xi,\delta}(z,w)  = \delta^{1-\xi Q + o_\delta(1)}$ for any two fixed distinct points $z,w\in\BB C$. 
See~\cite[Section 2.3]{dg-lqg-dim} for a heuristic argument supporting this relationship.  
Note that we have not yet shown that such an exponent $\xi(Q)$ exists for $Q \in (0,2)$, but we expect that this can be done using similar techniques to those in~\cite{dzz-heat-kernel}.   
 
We conjecture that if $\xi$ and $Q$ are related as above, then LFPP with parameter $\xi$ converges (under appropriate scaling) to the same limiting distance function as the square subdivision distance $D_h^\ep$ (see Conjecture~\ref{conj-metric-lim}). It may be possible to prove the convergence of LFPP for $\xi > 2/d_{\ccM=1}$ by improving on the techniques of~\cite{dddf-lfpp,gm-uniqueness}, but substantial new ideas would be necessary since the LQG distance function for matter central charge $\ccM \in (1,25)$ is not expected to induce the same topology as the Euclidean metric.

See~\cite{gp-lfpp-bounds} for more on LFPP with $\xi  >2/d_{\ccM=1}$. 
\medskip

\section{The KPZ formula}
\label{sec-kpz}

In this section we will prove Theorem \ref{thm-kpz}. 
The proof differs from the proof of the KPZ formula in~\cite{shef-kpz} since, in the case $Q< 2$, we have no Liouville measure $e^{\gamma h}\, d^2 z $ to work with. 
We first establish the  upper bound in Theorem \ref{thm-kpz}.

\begin{proof}[Proof of Theorem~\ref{thm-kpz}, upper bound.]
We fix a bounded open set $U\subset\BB C$ with $X\subset U$ a.s.  
For $n\in\BB N$, let $\mcl D_n(U)$ be the set of dyadic squares with side length $2^{-n}$ which intersect $U$. 
The probability that $U$ intersects a square of $\mcl S_h^\ep$ with side length larger than $1/2$ tends to zero as $\ep\rta 0$. 
Since the squares of $\mcl S_h^\ep$ intersect only along their boundaries, the number of such squares that $U$ intersects is at most a constant depending only on $U$. 
In particular, the expected number of squares in $\mcl S_h^\ep$ of side length larger than $1/2$ which intersect $U$ tends to zero as $\ep\rta 0$. 
Since $X$ and $h$ are independent, we therefore have
\eqb \label{eqn-upperbound-split} 
\mathbb{E}\left[ N_h^\ep(X)  \right] 
\leq \sum_{n = 1}^{\infty} \sum_{S \in \mcl D_n(U) } \BB P \left[ S\cap X \not=\emptyset \right] \BB P\left[ S \in\mathcal{S}^{\epsilon}_h  \right] + o_\ep(1) ,
\eqe 
where the $o_\ep(1)$ term comes from the squares in $\mcl S_h^\ep$ of side length larger than $1/2$ which intersect $U$. 

If $S \in  \mathcal{S}^{\epsilon}_h $, then by the definition~\eqref{eqn-dyadic-tiling-def} of $\mcl S_h^\ep $, the dyadic parent $\wt{S}$ of $S$ with $|\wt{S}| = 2^{-n+1}$ must satisfy $M_h(\wt{S}) = e^{h_{2^{-n}}(v_{\wt{S}})} (2^{-n+1})^Q > \epsilon $. 
To bound the probability that this is the case, let $m = m(\ep) \in\BB N$ be chosen so that $2^{-m} \leq \ep \leq 2^{-m+1}$ and let $n_\ep$ be the smallest integer such that $Q(n_\ep-1) - m > 0$. 
Since $ h_{2^{-n}}(v_{\wt{S}})$ is a centered Gaussian with variance $\log 2^n  + O_n(1)$ (with $O_n(1)$ depending only on $U$), the Gaussian tail bound gives
\eqb \label{eqn-gaussian-tail}
\BB P\left[ S\in \mcl S_h^\ep \right] 
\leq 2^{ -  \tfrac{(Q n - m + O_n(1))^2 }{2n + O_n(1) }      }    ,\quad\forall n \geq n_\ep ,\quad\forall S\in \mcl D_n(U) .
\eqe
For $n \leq n_\ep-1$, we use the trivial bound $\BB P\left[ S\in \mcl S_h^\ep \right] \leq 1$.  
By plugging these estimates into~\eqref{eqn-upperbound-split}, we get
\allb \label{eqn-upperbound-split1}
\mathbb{E}\left[  N_h^{\epsilon}(X) \right] 
&\leq \sum_{n=1}^{n_\ep-1} \sum_{S \in \mcl D_n(U) }  \BB P\left[ S \cap X \not=\emptyset \right]
 + \sum_{n = n_\ep}^\infty 2^{ -  \tfrac{(Q n - m + O_n(1))^2 }{2n + O_n(1) } } \sum_{S \in \mcl D_n(U) } \BB P\left[ S \cap X \not=\emptyset \right]   + o_\ep(1)    \notag \\
&= \sum_{n=1}^{n_\ep-1}  \BB E\left[N_0^{2^{-n}}(X) \right]
 + \sum_{n = n_\ep}^\infty 2^{ -  \tfrac{(Q n - m + O_n(1))^2 }{2n + O_n(1) } }  \BB E\left[ N_0^{2^{-n}}(X) \right] + o_\ep(1)   .
\alle
By hypothesis, $ \BB E\left[ N_0^{2^{-n}}(X) \right] = 2^{n(x + o_n(1))}$. Since $n_\ep \rta\infty$ as $\ep\rta 0$ (equivalently, as $m\rta\infty$), we can re-write~\eqref{eqn-upperbound-split1} as
\eqb \label{eqn-upperbound-split2}
\mathbb{E}\left[ N_h^\ep(X) \right] 
\leq 2^{n_\ep(x  + o_m(1))}
 + \sum_{n = n_\ep}^\infty 2^{ n x  -  \tfrac{(Q n - m )^2 }{2n  } + n o_m(1) } , 
\eqe 
where the $o_m(1)$ tends to zero as $m \rta\infty$ (equivalently, as $\ep\rta 0$), at a rate which does not depend on $n$. The first term on the right side of~\eqref{eqn-upperbound-split2} is at most $\ep^{-x/Q + o_\ep(1)}$ by our choice of $n_\ep$. To bound the second term, we first note that it is a series whose summand, viewed as a function of $n$, extends to a function on the positive real line with a unique maximum and no minimum.  This implies that\footnote{ Here we are using the fact that, if $f: \BB{R}^+ \rightarrow \BB{R}$ is a continuous function that is increasing on $(0,a)$ and decreasing on $(a,\infty)$, then the sum $\sum_{n=1}^{\infty} f(n)$ is a lower Riemann sum for the integral of the function $f^*$ that equals $f$ on $(0,a)$, the constant $f(a)$ on $(a,a+1)$, and $f(x-1)$ on $(a+1,\infty)$. Hence, $\sum_{n=1}^{\infty} f(n) \leq \int_0^{\infty} f^*(x) dx = \int_0^{\infty} f(x) dx + f(a)$.}
\allb
\sum_{n = n_\ep}^\infty 2^{ n x  -  \tfrac{(Q n - m )^2 }{2n  } + n o_m(1) } & \nonumber \leq \sum_{n = 1}^\infty 2^{ n x  -  \tfrac{(Q n - m )^2 }{2n  } + n o_m(1) } \\&\leq \int_0^\infty 2^{ t x  -  \tfrac{(Q t - m )^2 }{2n  } + t o_m(1) } dt + \max_{t \in \BB{R}^+} 2^{ t x  -  \tfrac{(Q t - m )^2 }{2n  } + t o_m(1) }.
\label{riemann_approx}
\alle
Moreover, the integrand $2^{ t x  -  \tfrac{(Q t - m )^2 }{2n  } + t o_m(1) }$ is the exponential of a continuously differentiable function with a unique maximum, where it equals $(Qm-m\sqrt{Q^2-2x}) \log{2}  + o_m(m)$. Therefore, by Laplace's method,
\eqb
\int_0^\infty 2^{ t x  -  \tfrac{(Q t - m )^2 }{2n  } + t o_m(1) } dt \leq 2^{Qm-m\sqrt{Q^2-2x}  + o_m(m)}.
\label{laplace}
\eqe
Combining (\ref{riemann_approx}) and (\ref{laplace}) and replacing $m$ by $\ep$, we get that the second term of~\eqref{eqn-upperbound-split2} is at most
\[ \ep^{ - ( Q - \sqrt{ Q^2 - 2x} ) + o_\ep(1)}.\]
Since $2x < Q^2$ implies that $Q - \sqrt{ Q^2 - 2x}  > x/Q$, plugging in our bounds for the two terms on the right side of~\eqref{eqn-upperbound-split2} shows that $\mathbb{E}\left[ N_h^\ep(X) \right]  \leq \ep^{ - ( Q - \sqrt{ Q^2 - 2x} ) + o_\ep(1)} $. This gives expectation bound in~\eqref{eqn-quantum-dim}. The a.s.\ bound then follows immediately from the Chebyshev inequality and the Borel-Cantelli lemma applied at dyadic values of $\ep$ (recall that $N_h^\ep(X)$ is increasing in $\ep$). 
\end{proof}

We now prove the lower bound in Theorem~\ref{thm-kpz}.

\begin{proof}[Proof of Theorem~\ref{thm-kpz}, lower bound.]
Intuitively, to obtain a lower bound, we want to consider a set of points of $X$ where the field $h$ is large, since squares in $\mathcal{S}_h^{\epsilon}$ that intersect this set will necessarily have small Euclidean size, giving a lower bound for the number of such squares in terms of the dimension of the set.  Thus, we want to consider points where the field $h$ is thick.  For $\alpha\in [-2,2]$, we write $\mcl T_\alpha$ for the set of $\alpha$-thick points of $h$, as defined in Definition~\ref{def-thick-pt}.  The dimension of $\alpha$-thick points in $X$ is given by~\cite[Theorem 4.1]{ghm-kpz}, which extends the main result of~\cite{hmp-thick-pts}.

\begin{thm}[\cite{ghm-kpz}] \label{th_GHM}
Let $X$ be a random fractal independent of $h$ such that a.s.\ $\dim_{\mcl H}(X) = x \in [0,2]$. 
For $\alpha >0$ such that $\frac{\alpha^2}{2} \leq x $, a.s.\ 
\begin{equation}
\dim_{\mathcal{H}}(\mathcal{T}_{\alpha} \cap X) =  \dim_{\mathcal{H}}( X) - \frac{\alpha^2}{2}  = x - \frac{\alpha^2}{2}.
\end{equation}
Furthermore if $\frac{\alpha^2}{2} > x $, then a.s. $\mathcal{T}_{\alpha} \cap X = \emptyset$.
\end{thm}

For our proof to work, we want to consider a set of points which are not just thick, but ``uniformly thick''.  More precisely, for $\zeta>0$ arbitrarily small, we consider a subset of $(\alpha - \zeta)$-thick points for which~\eqref{thick_points_def_eq} converges uniformly.  The following lemma, which is~\cite[Lemma 4.3]{ghm-kpz}, asserts that we can find such a set whose dimension differs from the dimension of $\alpha$-thick points in $X$ by at most $\zeta$.

\begin{lem}[\cite{ghm-kpz}]\label{uni}
Let $\alpha \in (0,2]$ and $\zeta  > 0$. Almost surely, there exists a random $ \overline{\delta} >0$ depending on $\alpha $ and $\zeta$ such that the following is true. If we set
\begin{equation}\label{lower_b2}
X^{\alpha, \zeta} := \left\{ z \in X: \frac{h_{\delta}(z)}{\log \delta^{-1}} \geq \alpha - \zeta, \forall \delta \in (0, \overline{\delta}] \right\}
\end{equation}
then $\dim_{\mathcal{H}} (X^{\alpha, \zeta}) \geq \dim_{\mathcal{H}}(\mathcal{T}_{\alpha} \cap X) - \zeta$.
\end{lem}

Together, Theorem~\ref{th_GHM} and Lemma~\ref{uni} imply that, for a fixed $\alpha \in (0,2]$ and  $\zeta  > 0$,
\eqb
\dim_{\mathcal{H}}(X^{\alpha, \zeta}) \geq \dim_{\mathcal{H}}(\mathcal{T}_{\alpha}\cap X) -\zeta = x - \frac{\alpha^2}{2} -\zeta.
\label{dim_alpha_zeta_lower}
\eqe
To prove the lower bound, we consider the set of squares in $\mathcal{S}_h^{\epsilon}$ that intersect $X^{\alpha, \zeta}$ for fixed $\alpha \in (0,2]$ and $\zeta > 0$.
For this purpose we will need the following elementary continuity estimate for the circle average process, see, e.g.,~\cite[Lemma 3.15]{ghm-kpz}. 

\begin{lem}\label{lem_con2}
For each fixed $\zeta \in (0, 1/2)$ and $R>0$, it holds with probability tending to 1 as $\ep \rta 0$ that
\begin{equation}
\label{lem_con2_bound}
| h_{\delta}(w) - h_{\delta^{1-\zeta}}(z) | \leq 3 \sqrt{10r} \log \delta^{-1}, \quad \forall \delta \in (0,\ep], \quad \forall z,w\in B_R(0) \:\text{with} \: |z-w| \leq 2\delta .
\end{equation}   
\end{lem}
 
Since $X$ is bounded, a.s.\ the maximum Euclidean size of the squares of $\mcl S_h^\ep$ which intersect $X$ tends to zero as $\ep\rta 0$. 
Hence, it is a.s.\ the case that for $\epsilon$ less than some random threshold, for each $z\in X^{\alpha,\zeta}$ and each $S \in \mathcal{S}_h^{\epsilon}$ with $z\in S$, the Euclidean size of $S$ is small enough that 
\begin{enumerate}[label=(\alph*)]
\item
 the definition \eqref{lower_b2} gives
\[
 e^{h_{(| S|/2)^{1-\zeta}}(z )}  \geq  \left(\frac{|S|}{2}\right)^{-(1-\zeta)(\alpha -\zeta)},
\]
and 
\item
 we can then apply Lemma~\ref{lem_con2} to recenter the circle average $h_{(| S|/2)^{1-\zeta}}(z)$ at the center $v_S$ of $S$ and get
\[
 e^{h_{| S|/2}(v_S)} \geq  \left(\frac{|S|}{2}\right)^{-\alpha + o_{\zeta}(1)},
\]
where the $o_\zeta(1)$ tends to zero as $\zeta\rta0$ at a rate depending only on $\alpha$.  
\end{enumerate} 
Thus, a.s.\ for small enough $\ep$, 
\eqb
\epsilon \geq  e^{h_{|S|/2}(v_S)} |S|^Q  \geq |S|^{Q-\alpha + o_\zeta(1)}
\label{S_upper_bound} 
\eqe
for all $S \in \mathcal{S}_h^{\epsilon}$ intersecting $X^{\alpha, \zeta}$.
We will now separately treat the case when $x <Q^2/2$ and the case when $x > Q^2/2$. 

If $x < Q^2/2$, then $X^{\alpha, \zeta}$ has nonzero dimension only for $\alpha < Q$. For such $\alpha$, the inequality~\eqref{S_upper_bound} gives (for $\zeta$ sufficiently small) an upper bound for the Euclidean size of a square in $\mathcal{S}_h^{\epsilon}$ that intersect $X^{\alpha, \zeta}$.  Combining this bound with the lower bound~\eqref{dim_alpha_zeta_lower} on the dimension of $X^{\alpha, \zeta}$, we get a lower bound for the number of squares in $\mathcal{S}_h^{\epsilon}$ that intersect $X^{\alpha, \zeta}$; namely, we get that a.s.\ 
\[
N^{\epsilon}_h(X^{\alpha,\zeta}) \geq N_0^{\epsilon^{\frac{1}{Q-\alpha+o_{\zeta}(1)}}}(X^{\alpha,\zeta}) 
\geq \epsilon^{  - \frac{ x - \alpha^2/2}{ Q-\alpha  } + o_\zeta(1)  +o_\epsilon(1) }.
\]
Sending $\zeta \rightarrow 0$, we deduce that, a.s.\
\eqb
N_h^{\epsilon}(X)  \geq  \epsilon^{  - \frac{ x - \alpha^2/2}{ Q-\alpha  }   +o_\epsilon(1) }.
\eqe 
The exponent on the right is minimized at $\alpha = Q - \sqrt{Q^2-2x}$, where it equals $Q - \sqrt{Q^2-2x}$. Making this choice of $\alpha$ gives the desired lower bound for $N_h^\ep(X)$.  
 
If $x > Q^2/2$,  then by~\eqref{dim_alpha_zeta_lower}, the set  $X^{\alpha, \zeta}$ has nonzero dimension for some $\alpha > Q$ and $\zeta$ sufficiently small.  For these values of $\alpha$ and $\zeta$, the exponent on the right side of \eqref{S_upper_bound} is negative. Therefore, \eqref{S_upper_bound} will always be false for $\epsilon$ sufficiently small, so a.s.\ for small enough $\ep$ none of the squares in $\mathcal{S}_h^{\epsilon}$ intersects the set $X^{\alpha, \zeta}$. In other words, every point of $X^{\alpha,\zeta}$ is a singularity for $\mcl S_h^\ep$. 
Therefore, a.s.\ $X$ intersects infinitely many singularities of $\mcl S_h^\ep$ for each small enough $\ep > 0$.  

We will now deduce from this that a.s.\ $N_h^\ep(X) = \infty$ for small enough $\ep$. 
Indeed, since $X$ is independent from $h$ and the probability that any fixed point of $\BB C$ is a singularity is zero, we infer that for any fixed $w\in \BB C$ and $r >0$, 
\eqb
\BB P\left[ \text{$X\cap B_r(w)\not=\emptyset$ and every point of $X\cap B_r(w)$ is a singularity}\right] = 0 .
\eqe
Applying this for $w\in\BB Q^2$ and $r\in \BB Q\cap (0,\infty)$ shows that a.s.\ the following is true. For every singularity $z$ of $\mcl S_h^\ep$ which intersects $X$ and every $r > 0$, there is a point of $X\cap B_r(z)$ which is \emph{not} a singularity for $\mcl S_h^\ep$. Hence each of the infinitely many singularities $z\in X\cap \mcl S_h^\ep$ is the limit of a sequence of points of $  X$ which are \emph{not} singularities for $\mcl S_h^\ep$, and so are contained in squares of $\mcl S_h^\ep$ which intersect $X$. 
The maximal size of the squares of $\mcl S_h^\ep$ which intersect $B_r(z)$ tends to zero as $r\rta 0$ (otherwise $z$ would be contained in a square of $\mcl S_h^\ep$).
Hence a.s.\ each of the singularities of $X\cap \mcl S_h^\ep$ is an accumulation point of arbitrarily small squares of $\mcl S_h^\ep$ which intersect $X$. 
Therefore, a.s.\ $N_h^\ep(X) = \infty$ for each small enough $\ep > 0$.  
\end{proof}

\begin{remark} \label{remark-not-true}  
Theorem~\ref{thm-kpz} is not true with the Hausdorff dimension of $X$ used in the upper bound and/or the Minkowski expectation dimension of $X$ used in the lower bound. 
For a counterexample in the case of the upper bound, take $X = \BB Q^2 \cap \BB D$.
Then the Hausdorff dimension of $X$ is 0 but $N_h^\ep(X) =  \ep^{- (Q - \sqrt{Q^2-4}) + o_\ep(1)}$ if $Q > 2$ or $N_h^\ep(X) = \infty$ for small enough $\ep$ if $Q < 2$. 
In the case of the lower bound, consider $X =\{0\} \cup \{1/n : n\in\BB N\}$, viewed as a subset of $\BB C$. 
Then the Minkowski dimension of $X$ is $1/2$, so if the lower bound in Theorem~\ref{thm-kpz} were true with Minkowski dimension in place of Hausdorff dimension, we would have $N_h^\ep(X) = \infty$ for small enough $\ep$ whenever $Q^2/2 < 1/2$. We claim, however, that a.s.\ $N_h^\ep(X) \leq \ep^{-1/Q + o_\ep(1)}$. 
Indeed, it is easy to see using basic Gaussian estimates that a.s.\ $|S| = \ep^{ 1/Q + o_\ep(1)}$ for each of the (at most 4) squares of $\mcl S_h^\ep$ which contain zero. 
This means that $X$ can be covered by one of these squares plus at most $\ep^{-1/Q + o_\ep(1)}$ additional squares, one for each of the points $1/n \in X$ with $1/n \geq \ep^{1/Q + o_\ep(1)}$. 
\end{remark}

\begin{remark} \label{remark-hausdorff}
One possible way of defining the ``$\ccM$-quantum Hausdorff dimension" $\dim_{\mcl H}^{\ccM} X$ of a set $X$ for a general choice of $\ccM  < 25$ is as follows. For $\Delta > 0$, we define the ``$\ccM$-quantum Hausdorff content" of $X$ by
\eqb
\inf\left\{ \sum_{i=1}^\infty M_h(S_i)^\Delta : \text{$\{S_i\}_{i\in\BB N}$ is a cover of $X$ by squares} \right\}
\eqe
where $M_h(S_i)$ is as in~\eqref{eqn-mass-def}. 
We define $\dim_{\mcl H}^{\ccM} X$ to be the infimum of the set of $\Delta  > 0$ for which this content is zero. 
We expect that if $X$ a deterministic set of a random set sampled independently from $h$, then a.s.\ $\dim_{\mcl H}^{\ccM} X $ is related to the Euclidean Hausdorff dimension of $X$ via the KPZ formula as in Theorem~\ref{thm-kpz}. We do not prove this here, however.
See~\cite[Theorem 5.4]{rhodes-vargas-log-kpz} for a version of the KPZ formula for a class of GMC measures (including $\gamma$-LQG measures for $\gamma \in (0,2)$) using a similar notion of Hausdorff dimension. 
\end{remark}

\section{White noise approximations of the GFF}
\label{sec-prelim}

The rest of the paper is devoted to studying graph distances in $\mcl S_h^\ep$, which will eventually lead to proofs of Theorem~\ref{thm-ball-infty} and Proposition~\ref{prop-ptwise-distance}. Before proceeding with the core arguments in Section~\ref{sec-ball-infty}, in this brief section we make some preliminary definitions and record a few basic estimates.

\subsection{White noise approximation}
\label{sec-wn-decomp}
 
In this subsection we will introduce various white-noise approximations of the Gaussian free field which are often more convenient to work with than the GFF itself. Our exposition closely follows that of~\cite{dg-lqg-dim}. 

Let $W$ be a space-time white noise on $\BB C\times [0,\infty)$, i.e., $\{(W,f) : f\in L^2(\BB C\times [0,\infty))\}$ is a centered Gaussian process with covariances $\BB E[(W,f) (W,g) ]  = \int_\BB C\int_0^\infty f(z,s) g(z,s) \, dz\,ds$. For $f\in L^2(\BB C\times [0,\infty))$ and Borel measurable sets $A\subset\BB C$ and $I\subset [0,\infty)$, 
we slightly abuse notation by writing 
\eqbn
\int_A\int_I f(z,s) \, W(dz,ds) := (W , f \BB 1_{A\times I} ) .
\eqen

Approximations of the GFF can be obtained by integrating the transition density of Brownian motion against $W$. 
For an open set $U \subset \BB C$, let $p_U(s ; z,w)$ be the transition density of Brownian motion stopped at the first time it exits $U$, so that for $s\geq 0$, $z\in \BB C$, and $A\subset \ol U$, the probability that a standard planar Brownian motion $\mcl B$ started from $z$ satisfies $\mcl B([0,s]) \subset U$ and $\mcl B_s \in A$ is $\int_A p_U(s;z,w) \,dw$. We abbreviate
\eqbn
p(s;z,w) := p_{\BB C}(s;z,w) =  \frac{1}{2\pi s} \exp\left( - \frac{|z-w|^2}{2s} \right) .
\eqen  

Following~\cite{rhodes-vargas-log-kpz}, we define the centered Gaussian process
\eqb \label{eqn-wn-decomp}
\wh h_t (z) := \sqrt\pi \int_{\BB C} \int_{t^2}^1 p (s/2 ;z,w) \, W(dw,ds)  ,\quad \forall t\in [0,1] , \quad \forall z\in \BB C .
\eqe 
We abbreviate $\wh h := \wh h_0$. 
By~\cite[Lemma 3.1]{ding-goswami-watabiki} and Kolmogorov's criterion, each $\wh h_t$ for $t > 0$ admits a continuous modification. Henceforth whenever we work with $\wh h_t$ we will assume that it is been replaced by such a modification (we will only ever need to consider at most countably many values of $t $ at a time). 
The process $\wh h$ does not admit a continuous modification, but can be viewed as a random distribution. 
We note that
\eqb \label{eqn-wn-var}
\BB E\left[ \left(\wh h_{\wt t}(z) -  \wh h_t(z)\right)^2 \right] = \log (\wt t/t),\quad\forall z \in \BB C, \quad\forall 0 < t <\wt t < 1 .
\eqe 

The process $\wh h_t$ is a good approximation for the circle-average process of a GFF over circles of radius $t$~\cite[Proposition 3.2]{ding-goswami-watabiki}.  
The former process is in some ways more convenient to work with than the GFF thanks to the following symmetries, which are immediate from the definition. 
\begin{itemize}
\item \textit{Rotation/translation/reflection invariance.} The law of $\{\wh h_t(z) : t\in [0,1] , z\in \BB C\}$ is invariant with respect to rotation, translation, and reflection of the plane.
\item \textit{Scale invariance.} For $\delta \in (0,1]$, one has $\{ (\wh h_{\delta t} - \wh h_{\delta})(\delta z) : t \in [0,1] , z\in \BB C\} \eqD \{\wh h_t(  z) : t\in [0,1] , z\in \BB C\}$. 
\item \textit{Independent increments.} If $0 \leq t_1<t_2 \leq t_3 < t_4$, then $\wh h_{t_2} - \wh h_{t_1}$  and $\wh h_{t_4} - \wh h_{t_3}$ are independent. 
\end{itemize}

For $Q >0$ and $U\subset\BB C$, we define a collection of non-overlapping dyadic squares $\mcl S_{\wh h}^\ep(U)$ and a pseudometric $D_{\wh h}^\ep$ on $\BB C$ in exactly the same manner as in~\eqref{eqn-dyadic-tiling-def} but with $\wh h_{|S|/2}$ used in place of $h_{|S|/2}$. That is, for a square $S\subset \BB C$ we set 
\eqb \label{eqn-mass-def-wn}
M_{\wh h}(S) := e^{\wh h_{|S|/2}(v_S)} |S|^Q  
\eqe  
and for $U\subset \BB C$ we define
\allb \label{eqn-dyadic-tiling-def-wn}
\mcl S_{\wh h}^\ep(U) &:= \big\{ \text{dyadic squares $S\subset U$ with $M_{\wh h}(S) \leq \ep$ and $M_{\wh h}(S') > \ep$, } \notag \\
&\qquad \qquad \qquad \text{$\forall$ dyadic ancestors $S' \subset U$ of $S$} \big\}.
\alle
We view $\mcl S_{\wh h}^\ep(U)$ as a graph with squares considered to be adjacent if their boundaries intersect and for $z,w\in U$ we define $D_{\wh h}^\ep(z,w; U)$ to be the minimal graph distance in $\mcl S_{\wh h }^\ep(U)$ from a square containing $z$ to a square containing $w$, or $\infty$ if either $z$ or $w$ is not contained in a square of $\mcl S_{\wh h}^\ep(U)$. We drop the $U$ from the notation when $U=\BB C$ and we define the $D_{\wh h}^\ep$-distance between sets as in~\eqref{eqn-dist-set}.  
We note that the obvious analogue of~\eqref{eqn-dist-scaling} holds for $D_{\wh h}^\ep$. 

It is sometimes convenient to work with a truncated variant of $\wh h $ where we only integrate over a ball of finite radius. As in~\cite{ding-goswami-watabiki,dzz-heat-kernel,dg-lqg-dim}, we define a truncated version of the white noise approximation of the GFF by
\allb \label{eqn-wn-truncate}
 \wh h_t^{\tr}(z) := \sqrt\pi \int_{\BB C} \int_{t^2}^\infty p_{B_{1/10}(z)}(s/2 ;z,w) \, W(dw,ds)  ,\quad \forall 0 \leq t \leq 1 ,  \quad \forall z\in \BB C .
\alle  
As above, we write $\wh h^\tr := \wh h^\tr_0$. Each $\wh h_t^\tr$ a.s.\ admits a continuous modification, but $\wh h^\tr$ does not, and is instead viewed as a random distribution. 
The process $\wh h^\tr$ lacks the scale invariance property enjoyed by $\wh h$. However, it does possess an exact local independence property, which is immediate from the analogous property of the white noise: if $A,B \subset \BB C$ with $\op{dist}(A,B) \geq 1/5$, then $\{\wh h^\tr_t|_A\}_{t\in [0,1]}$ and $\{\wh h^\tr_t|_B\}_{t\in [0,1]}$ are independent. 

We define $M_{\wh h^\tr}(S)$ for squares $S\subset \BB C$ as well as $\mcl S_{\wh h^\tr}^\ep(U)$ and $D_{\wh h^\tr}^\ep(\cdot,\cdot ;U)$ exactly as above but with $\wh h^\tr_{|S|/2}$ in place of $\wh h_{|S|/2}$.

The following lemma will allow us to use $\wh h^\tr$ or $\wh h$ in place of the GFF in many of our arguments. 
It can be proven using basic estimates for the Brownian transition density which allow one to check Kolmogorov's continuity criterion; see~\cite[Lemma 3.1]{dg-lqg-dim}. 

\begin{lem} \label{lem-gff-compare} 
Suppose $U\subset \BB C$ is a bounded Jordan domain and let $K$ be the set of points in $U$ which lie at Euclidean distance at least $1/10$ from $\bdy U$. 
There is a coupling $(h , h^U,\wh h , \wh h^\tr)$ of a whole-plane GFF normalized so that $h_1(0) = 0$, a zero-boundary GFF on $U$, and the fields from~\eqref{eqn-wn-decomp} and~\eqref{eqn-wn-truncate} such that the following is true. For any $h^1,h^2 \in \{h , h^U,\wh h , \wh h^\tr\}$, the distribution $(h^1-h^2)|_K$ a.s.\ admits a continuous modification and there are constants $c_0,c_1 > 0$ depending only on $U$ such that for $A>1$, 
\eqb \label{eqn-gff-compare}
\BB P\left[\max_{z\in K} |(h^1-h^2)(z)| \leq A \right] \geq 1 - c_0 e^{-c_1 A^2} .
\eqe
In fact, in this coupling one can arrange so that $\wh h$ and $\wh h^\tr$ are defined using the same white noise and $h - h^U$ is harmonic on $U$. 
\end{lem}

We will need the following variant of Lemma~\ref{lem-gff-compare}.

\begin{lem} \label{lem-tr-compare-square}
Let $U$ and $K$ be as in Lemma~\ref{lem-gff-compare}. There is a coupling $(h , h^U,\wh h , \wh h^\tr)$ such that for any $h^1,h^2 \in \{h , h^U , \wh h , \wh h^\tr\}$ and each $C>2$,
\eqb \label{eqn-tr-compare-set}
\BB P\left[ \text{Each square in $\mcl S_{h_1}^{ \ep}(K)$ is contained in a square of $\mcl S_{h_2}^{C\ep}(K)$} \right] \geq 1 - a_0 e^{-a_1 (\log C)^2}
\eqe 
and
\eqb \label{eqn-tr-compare-square}
\BB P\left[ D_{h_1}^{ \ep}\left( z , w ; \BB S \right) \leq   D_{h_2}^{C\ep}\left( z , w ; K \right)    ,\: \forall z,w\in K \right] \geq 1 - a_0 e^{-a_1(\log C)^2} . 
\eqe 
\end{lem}
\begin{proof}
The lemma in the case when $(h^1,h^2) = (h ,h^U)$ is immediate from Lemma~\ref{lem-gff-compare} applied with $A =  \log C $ and~\eqref{eqn-dist-scaling}. 
The case when $(h^1 , h^2) = (\wh h, \wh h^\tr)$ follows since the same argument\footnote{The truncated white noise decomposition is called $\eta$ in~\cite{dzz-heat-kernel} and has a slightly more complicated definition than $\wh h^\tr$, but the same (in fact, a slightly easier) argument works in the case of $\wh h^\tr$.}  used to prove~\cite[Lemma 2.7]{dzz-heat-kernel} shows that if $\wh h$ and $\wh h^\tr$ are defined using the same white noise, then for appropriate constants $a_0,a_1 > 0$ as in the lemma statement,  
\eqbn
\BB P\left[ \max_{z\in K } \max_{j \in \BB N_0} | \wh h_{  2^{-j} }(z) -  \wh h^\tr_{  2^{-j} }(z)| \leq \log C \right] \geq 1 -  a_0  e^{- a_1 (\log C)^2} ,\quad\forall C >2. 
\eqen
Finally, in the case when $(h^1,h^2) = (\wh h , h^U)$, we obtain the needed comparison between $\wh h_t$ and the circle-average process for $h^U$ from~\cite[Proposition 3.3]{ding-goswami-watabiki} and a union bound over dyadic radii. 
\end{proof}

\subsection{Estimates for various fields}
\label{sec-gff-estimates}

Let us now record some basic estimates for the squares of $\mcl S_h^\ep$. 
We first show that none of the squares of $\mcl S_h^\ep$ are macroscopic. 
 
\begin{lem} \label{lem-max-square-diam}
Let $h$ be any of the four fields from Lemma~\ref{lem-tr-compare-square} with $U = \BB S(1) = (-1,2)^2$. For each $\zeta\in (0,1)$, it holds with polynomially high probability as $\ep \rta 0$ that
\eqbn
\max\left\{ |S| : S\in \mcl S_{ h}^\ep(\BB S) \right\} \leq  \ep^{\frac{1}{2+Q } - \zeta} .
\eqen  
\end{lem}
\begin{proof}
By Lemma~\ref{lem-tr-compare-square} applied with $K = \ol{\BB S}$, it suffices to prove the lemma in the case when $h = \wh h$ (note that $\BB S$ lies at distance at least $1/10$ from $\bdy \BB S(1)$).  
For $n\in\BB N$, let $\mcl D_n$ be the set of dyadic squares $S\subset\BB S$ which have side length $|S| = 2^{-n}$. 
Recall that $\wh h_{2^{-n-1}}(v_S)$ is centered Gaussian with variance $\log 2^{n+1}$ for each $S\in \mcl D_n$. By the Gaussian tail bound, if $n\in\BB N$ with $2^{ Q n} \ep  < 1$ and $S\in \mcl D_n$,
\allb \label{eqn-square-in-set-prob}
\BB P\left[ S \in \mcl S_{\wh h}^\ep(\BB S) \right] 
\leq \BB P\left[ e^{\wh h_{|S|/2}(v_S)} |S|^Q  < \ep \right] 
\leq \BB P\left[ \wh h_{2^{-n-1}}(v_S) < \log(2^{Q n} \ep ) \right] 
\leq \exp\left( - \frac{(\log (2^{Q n}\ep))^2}{2 \log 2^{ n + 1}} \right) .
\alle 
By a union bound over $2^{2n}$ elements of $\mcl D_n$,
\eqbn
\BB P\left[ \mcl D_n\cap \mcl S_{\wh h}^\ep(\BB S) \not=\emptyset \right] \leq 2^{2n} \exp\left( - \frac{(\log (2^{Q n}\ep))^2}{2 \log 2^{ n + 1}} \right).
\eqen 
We now conclude by summing this estimate over all $n\in\BB N$ with $2^{-n} \geq \ep^{1/(2+Q ) -\zeta}$. 
\end{proof}

\begin{lem} \label{lem-unit-square-good}
Suppose that $h$ is either a whole-plane GFF, the field $\wh h$ of~\eqref{eqn-wn-decomp}, or the field $\wh h^\tr$ of~\eqref{eqn-wn-truncate}. It holds with polynomially high probability as $\ep\rta 0$ that each square of $\mcl S_h^\ep(\BB S)$ is also a square of $\mcl S_h^\ep$.
\end{lem}
\begin{proof}
It suffices to show that with polynomially high probability as $\ep\rta 0$, neither the square $\BB S$ nor any of its dyadic ancestors belongs to $\mcl S_h^\ep(\BB S)$. 
Equivalently, if we let $\BB S = S_0 \subset S_1\subset\dots$ be the sequence of dyadic ancestors of $\BB S$, it suffices to show that with polynomially high probability as $\ep\rta 0$, it holds that $e^{h_{|S_j|/2}(v_{S_j})} |S_j|^Q  > \ep$ for each $j \in \BB N$. This follows from the Gaussian tail bound and a union bound over all $j\in\BB N_0$ since $|S_j|  =2^j$ and $h_{|S_j|/2}(v_{S_j})$ is Gaussian with variance $\log 2^j + O (1)$. 
\end{proof}

\section{Ball volume growth is superpolynomial}
\label{sec-ball-infty}

In this section we prove Theorem~\ref{thm-ball-infty}. Along the way, we will also obtain Proposition~\ref{prop-ptwise-distance} (see the beginning of Section~\ref{sec-macro-square-dist}). 
To prove Theorem~\ref{thm-ball-infty}, we will first prove a lower bound for the cardinality of the graph-distance ball centered at the origin-containing square in $\mcl S_h^\ep(\BB S)$, then transfer to the statement of Theorem~\ref{thm-ball-infty} using the scaling properties of the GFF. The proof of the estimate for balls in $\mcl S_h^\ep(\BB S)$ consists of four main steps, which are carried out in Sections~\ref{sec-dist-along-line},~\ref{sec-macro-square-dist},~\ref{sec-singularity-dist}, and~\ref{sec-small-square-count}, respectively. 
\begin{enumerate}
\item \textbf{Distance between two sides of a rectangle.} We show that the $D_h^\ep$-distance between two opposite sides of a fixed rectangle in $\BB R^2$ typically grows at most polynomially in $\ep$. \label{item-across-square}
\item \textbf{Maximal distance between large squares.} We show (Proposition~\ref{prop-macro-square-dist}) that for any fixed $\beta > 0$, it holds with high probability as $\ep\rta 0$ that any two squares in $\mcl S_h^\ep(\BB S)$ with side length at least $\ep^\beta$ lie at $\mcl S_h^\ep(\BB S)$-graph distance at most $\ep^{-f(\beta)  - o_\ep(1)}$ from one another, for an exponent $f(\beta)$ depending only on $Q$ and $\beta$. \label{item-macro-square}
\item \textbf{Distance from a small square to a large square.} We show that when $K >1$ is large, a typical ``small" square of $\mcl S_h^\ep(\BB S)$ with side length of order $\ep^K$ is close to some ``large" square of $\mcl S_h^\ep(\BB S)$ with side length at least $\ep^{1/Q + o_\ep(1)}$, in the sense that the distance between the two squares is at most $\ep^s$ for an exponent $s   > 0$ which can be taken to be independent of $K$ if certain parameters are chosen appropriately. \label{item-singularity-dist}
\item \textbf{Lower bound for the number of small squares.} We show that with high probability, the number of squares of $\mcl S_h^\ep(\BB S)$ with side length at most $\ep^K$ which satisfy the condition of step~\ref{item-singularity-dist} grows at least as fast as some power of $\ep^{-K}$ (the power depends only on $Q$).  \label{item-small-square-count}
\end{enumerate}
Combining the last three statements shows that with high probability, the number of squares of side length smaller than $\ep^K$ which belong to the graph distance ball of radius $\ep^{- ( s \wedge f(1/Q) )}$ in $\mcl S_h^\ep(\BB S)$ centered at a fixed point is typically at least a $Q$-dependent power of $\ep^{-K}$.
Given $r\in\BB N$, we will then choose $\ep$ so that $\ep^{-(s\wedge f(1/Q))} = r$, take $K$ to be arbitrarily large, and use the scaling properties of the GFF to get Theorem~\ref{thm-ball-infty}. 

Let us now give slightly more detail as to how each of the above steps is carried out. Step~\ref{item-across-square} is essentially obvious for $Q\in [\sqrt 2 , 2)$, and for $Q < \sqrt 2$ is obtained by considering a path of squares which follows an SLE$_4$ level line of the GFF between the two opposite sides of the rectangle. 

For step~\ref{item-macro-square}, we will use step~\ref{item-across-square} and a percolation argument to build a ``grid" consisting of a polynomial number of paths of squares in $\mcl S_h^\ep(\BB S)$, each of which has at most polynomial length, which intersect every dyadic square in $\BB S$ with side length at least $\ep^\beta$. 
This leads to an upper bound on the maximal distance between two such squares.

To carry out step~\ref{item-singularity-dist}, we first need to specify what we mean by a ``typical" small square of $\mcl S_h^\ep(\BB S)$. 
Roughly speaking, the accumulation points of arbitrarily small squares in $\mcl S_h^\ep(\BB S)$ correspond to the $\alpha$-thick points of $h$ for $\alpha \in (Q,2)$ (Definition~\ref{def-thick-pt}).
The set of $\alpha$-thick points is larger (e.g., in the sense of Hausdorff dimension~\cite{hmp-thick-pts}) for smaller values of $\alpha$, so we expect that most small squares of $\mcl S_h^\ep(\BB S)$ arise from $\alpha$-thick points with $\alpha$ close to $Q$.
This leads us to fix $z\in \BB S$ and $\alpha \in (Q,2)$ and analyze the field $h^\alpha := h - \alpha\log|\cdot - z|$, which essentially amounts to conditioning $z$ to be a thick point of $h$. 
We will prove that for any large exponent $K > 1$, there is a square $S$ of $\mcl S_{h^\alpha}^\ep(\BB S)$ near $z$ with side length at most $\ep^K$ and a square $S'$ of $\mcl S_{h^\alpha}^\ep(\BB S)$ with side length at least $\ep^{1/Q + o_\ep(1)}$ such that $D_{h^\alpha}^\ep(S,S' ; \BB S)$ grows at most polynomially in $\ep$, at a rate which does not depend on $K$, provided we take $\alpha$ sufficiently close to $Q$, depending on $K$ ($\alpha = Q + 1/K$ will suffice).

For step~\ref{item-small-square-count}, we will use the fact that for $\alpha \in (Q,2)$, a point sampled uniformly from the $\alpha$-LQG measure $\mu_h^\alpha$ is a.s.\ $\alpha$-thick~\cite{kahane} (see also~\cite[Section 3.3]{shef-kpz}). We will combine this fact with the analysis of $h^\alpha$ described above and estimates for the $\alpha$-LQG measure (which show that its mass is sufficiently ``spread out" over $\BB S$) to show that there are a large number of small squares in $\mcl S_h^\ep(\BB S)$ which are close to large squares in the sense described above. 

Throughout this section we will use the following notation. 

\begin{defn} \label{def-rectangle}
For a rectangle $R = [a,b] \times [c,d] \subset \BB C$ parallel to the coordinate axes, we write $\bdy_{\op{L}} R$, $\bdy_{\op{R}} R$, $\bdy_{\op{T}} R$, and $\bdy_{\op{B}} R$, respectively, for its left, right, top, and bottom boundaries.  
\end{defn}

\subsection{Distance between two sides of a rectangle grows at most polynomially}
\label{sec-dist-along-line}

Let $h$ be a whole-plane GFF normalized so that $h_1(0) = 0$. In this subsection we will establish that the $D_h^\ep$ distance between two sides of a rectangle grows at most polynomially in $\ep$.

\begin{lem} \label{lem-dist-along-line}
For each $Q \in (0,2)$ and with $\xi = \xi(Q):= \frac{3}{2Q} >  0$ the following is true. For each $\zeta\in (0,1)$ and each $2\times 1$ or $1\times 2$ rectangle $R\subset \BB C$ with sides parallel to the coordinate axes and corners in $\BB Z^2$, it holds that
\eqb \label{eqn-dist-along-line}
\lim_{\ep\rta 0} \BB P\left[ D_h^\ep\left( \bdy_{\op{L}} R, \bdy_{\op{R}} R ; R \right) \leq  \ep^{-\xi - \zeta } \right]  =1 .
\eqe  
The same is true with the field $\wh h$ of~\eqref{eqn-wn-decomp} in place of $h$. 
\end{lem}

The exponent $\xi$ in Lemma~\ref{lem-dist-along-line} is far from optimal; the main point is just to get a polynomial upper bound. In the case when $Q \in (\sqrt 2, 2)$, one can take $\xi = Q - \sqrt{Q^2-2}  $ instead by considering the set of squares of $\mcl S_h^\ep(R)$ which intersect a fixed line segment (Theorem~\ref{thm-kpz}). 
Lemma~\ref{lem-dist-along-line} will be a consequence of the following estimate, which gives a path between two sides of a rectangle on which the values of $h$ are of approximately constant order.

\begin{lem} \label{lem-gff-perc}
Let $h$ be a whole-plane GFF and fix $\zeta\in (0,1)$. 
For each fixed rectangle $R\subset \BB C$ with sides parallel to the coordinate axes and corners in $\BB Z^2$, it holds with probability tending to 1 as $n\rta\infty$ that the following is true. 
There is a simple path $P$ in $(2^{-n}\BB Z^2)\cap  R$ (equipped with its usual nearest-neighbor graph structure) from $\bdy_{\op{L}} R$ to $\bdy_{\op{R}} R$ such that
\eqb \label{eqn-gff-perc}
\# P \leq 2^{(3/2 + \zeta)n } \quad \text{and} \quad |h_{2^{-n-1}}(z)| \vee |h_{2^{-n}}(z)| \leq \zeta \log 2^n ,\quad\forall z\in P .
\eqe
The same is also true if we instead consider paths in the dual lattice $(2^{-n}( \BB Z^2 + (1/2,1/2) )) \cap R$ between the $2^{-n-1}$-neighborhoods of $\bdy_{\op{L}} R$ and $\bdy_{\op{R}} R$. 
\end{lem}

To prove Lemma~\ref{lem-dist-along-line}, we will apply Lemma~\ref{lem-gff-perc} with $2^{-n} \approx \ep^{ 1/Q}$. 
A similar argument to the proof of~\cite[Theorem 1]{ding-li-chem-dist} should show that Lemma~\ref{lem-gff-perc} is true with the circle average process replaced by a zero-boundary discrete GFF on $2^{-n}\BB Z^2$, with the exponent $3/2$ in~\eqref{eqn-gff-perc} replaced by 1. This together with the comparison of the zero-boundary GFF and the circle average process established in~\cite{ang-discrete-lfpp} should lead to a proof of Lemma~\ref{lem-gff-perc} with 1 in place of $3/2$. We will not give such a proof here, however, since~\cite[Theorem 1]{ding-li-chem-dist} is stated for crossings of annuli instead of crossings of rectangles so one would either need to repeat the proof from~\cite{ding-li-chem-dist} or carry out a non-trivial RSW-type argument to get the result as stated in Lemma~\ref{lem-gff-perc}. 

\begin{proof}[Proof of Lemma~\ref{lem-gff-perc}]
	We will prove the statement for $2^{-n}\BB Z^2$; the proof for the dual lattice $ 2^{-n}( \BB Z^2 + (1/2,1/2) ) $ follows from exactly the same argument. 
	The basic idea of the proof is to consider a path which stays close to an SLE$_4$ level line of $h$, in the sense of~\cite{ss-contour}. 
	Since~\cite{ss-contour} only considers level lines of a GFF on a proper simply connected domain with piecewise linear boundary data and since we want our level line to stay in $R$, we will need to consider such a GFF on a domain slightly larger than $R$. 
	
	For $r\geq 0$, let $R_r$ be the rectangle with the same center as $R$ and whose side lengths are $1+r$ times the side lengths of $R$ and let $x_r$ and $y_r$, respectively, be the midpoints of $\bdy_{\op{L}} R_r$ and $\bdy_{\op{R}} R_r$. 
	Let $\lambda$ be the constant from~\cite{ss-contour} ($\lambda= \pi/2$ with our choice of normalization).
	Let $h'$ be a GFF on $R_r$ such that $h'$ has boundary data $\lambda$ (resp.\ $-\lambda$) on the counterclockwise (resp.\ clockwise) boundary arc from $x_r$ to $y_r$, i.e., $h'$ is a zero-boundary GFF on $R_r$ plus the harmonic function on $R_r$ with this boundary data. 
	
	Fix $s\in (0,1)$. We will show that~\eqref{eqn-gff-perc} holds with probability at least $1-s$ if $n$ is sufficiently small. 
	Recall that $h|_{R_r}$ can be written in the form $h|_{R_r}=h^{R_r} + \frk h$, where $h^{R_r}$ is a zero-boundary GFF in $R_r$ and $\frk h$ is a Gaussian harmonic function in $R_r$. Since $\frk h$ is continuous on $R$ and since the boundary data for the field $h'$ above is bounded, we can find a constant $C_0>0$ (depending on $r$ and $s$) and a coupling of $h$ and $h'$ such that a.s.\ $h-h'$ is a continuous function and 
	\eqb \label{eqn-sle-gff-compare}
	\BB P\left[ \sup_{z\in R}|h(z)-h'(z)|  \leq  C_0 \right]  \geq 1 - \frac{s}{10}.
	\eqe
	
	\begin{figure}[t!]
		\begin{center}
			\includegraphics[scale=1]{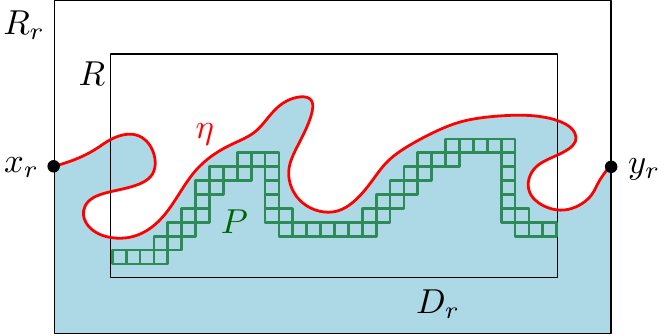}
			\caption{Illustration of the proof of Lemma \ref{lem-gff-perc}. By \cite{ss-contour}, for a particular Dirichlet GFF $h'$ in $R_r$ the level line $\eta$ has the law of an  SLE$_4$. Since $h'$ is smaller near $\eta$ than at a ``typical'' interior point, with high probability we can find a path $P$ in $2^{-n}\BB Z^2$ which connects $\partial_{\op{L}} R$ and $\partial_{\op{R}} R$ along which $|h_{2^{-n-1}}(z)| \vee |h_{2^{-n}}(z)|$ is small. We need $2^{n(3/2 + o_n(1))}$ squares since the Minkowski dimension of $\eta$ is $3/2$. By coupling $h$ and $h'$ so they are close in $R$ we obtain the statement of the lemma. }
			\label{fig-sle-path}
		\end{center} 
	\end{figure} 

	By the main result of~\cite{ss-contour} we may define the level line $\eta = \eta_r$ of $h'$ from $x_r$ to $y_r$ as a chordal SLE$_4$ curve from $x_r$ to $y_r$ in $R_r$ which is determined by $h'$. See Figure~\ref{fig-sle-path}.
	 
Let $\wh R_r$ be the rectangle with the same center point as $R$ (equivalently, $R_r$) whose height is $1-r$ times the height of $R$ and whose width is the same as the width of $R_r$. Note that $R\subset \wh R_r \subset R_r$ and the distance from $\wh R_r$ to each of the top and bottom boundaries of $R$ is bounded above and below by $R$-dependent constants times $r$.
We first claim that for small enough $r > 0$,   
\eqb  \label{eqn-sle-contained}
\BB P\left[ \eta \subset \wh R_r \right] \geq 1 - \frac{s}{10} .
\eqe
To see this $\phi:R_r\to\BB H$ be the conformal map such that $\phi(x_r)=0$, $\phi(y_r)=\infty$, and the two corners of $R_r$ which are closest to $x_r$ are mapped to $\pm 1$. Note that there is a constant $C_1>0$ depending only on $R$ such that $|\phi'|\leq C_1$ on $R_r\setminus\wh R_r$ for all $r<1$. Therefore there is a $C_2>0$ such that $\phi(R_r\setminus \wh R_r )\subset\{z\in\BB H\,:\,C_2^{-1}<|\op{Re}(z)|<C_2,\op{Im}(z)< C_2  r  \}$. By transience of SLE$_4$, and since an SLE$_4$ a.s.\ does not visit the domain boundaries except at the initial and terminal points, the probability that an SLE$_4$ from 0 to $\infty$ in $\BB H$ visits this last set converge to 0 as $r\rta 0$. This gives \eqref{eqn-sle-contained} by the conformal invariance of SLE$_4$. We henceforth fix $r$ so that~\eqref{eqn-sle-contained} holds.

	By the main result of~\cite{lawler-rezai-nat}, $\eta$ has finite $3/2$-dimensional Minkowski content a.s. 
	This implies that we can find a constant $C_1 > 0$ depending only on $s$ and $R$ such that 
	\eqb \label{eqn-sle-dim}
	\BB P\left[ \op{area} B_{2^{-n+10}}(\eta) \leq C_1 2^{-n(1-\zeta)/2} \right] \geq 1 -\frac{s}{10} ,
	\eqe
	where $B_{2^{-n+10}}(\eta)$ denotes the Euclidean $2^{-n+10}$-neighborhood of $\eta$ and $\op{area}$ denotes Lebesgue measure. 
	
	Let $D_r\subset R_r$ denote the simply connected domain whose boundary is the union of $\partial_{\op{B}} R_r$, $\eta$, and the segments of $\partial_{\op{L}} R_r$ and $\partial_{\op{R}} R_r$ which lie below $x_r$ and $y_r$, respectively. 
	The curve $\eta$ is a local set for $h'$ in the sense of~\cite[Lemma 3.9]{ss-contour}. In particular, if we condition on $\eta$, then the field $h'|_{D_r}$ has the law of a zero-boundary GFF on $D_r$ plus the constant function $\lambda$. For any $z\in D_r$ let $\op{CR}(z,D_r)$ denote the conformal radius of $z$ in $D_r$. By \cite[Proposition 3.2]{shef-kpz}, under the conditional law given $\eta$, for each $\delta>0$ such that $B_\delta(z)\subset D_r$, the random variable $h'_\delta(z)$ is Gaussian with mean $\lambda$ and variance $\op{Var}(h'_\delta(z) | \eta)=-\log\delta+\log \op{CR}(z,D_r)$. By the Koebe quarter theorem, $\frac14 \op{dist}(z,\bdy D_r)\leq \op{CR}(z;D_r) \leq 4 \op{dist}(z,\bdy D_r)$. Consequently, 
	\eqbn
	\op{Var}\left( h'_{2^{-n}}(z) | \eta \right) \vee \op{Var}\left( h'_{2^{-n-1}}(z) | \eta \right) \preceq 1 ,\quad\forall z\in \left( B_{2^{-n+10}}(\eta) \setminus B_{2^{-n}}(\eta)  \right) \cap D_r ,
	\eqen
	with a universal implicit constant. 
	If~\eqref{eqn-sle-dim} holds, then the number of points in $(2^{-n}\BB Z^2)\cap B_{2^{-n+10}}(\eta)$ is at most $C_1 2^{n(3+\zeta)/2 + 2}$. By the Gaussian tail bound and a union bound over all such points, we see that if $n$ is sufficiently large, then whenever the event in~\eqref{eqn-sle-dim} occurs, it a.s.\ holds with conditional probability at least $1-s/10$ given $ \eta$ that
	\eqb \label{eqn-max-sle-nbd}
	\max\left\{  |h'_{2^{-n-1}}(z)| \vee |h'_{2^{-n}}(z)|  :   z \in \left(2^{-n}\BB Z^2\right)\cap  \left( B_{2^{-n+10}}(\eta) \setminus B_{2^{-n}}(\eta)  \right) \cap D_r \right\}   \leq  \frac{\zeta}{2} \log 2^n   .
	\eqe 
	
	Now assume that the events in~\eqref{eqn-sle-gff-compare}, \eqref{eqn-sle-contained}, \eqref{eqn-sle-dim}, and~\eqref{eqn-max-sle-nbd} all occur, which happens with probability at least $1-s$. 
	The event in~\eqref{eqn-sle-contained} implies that if $n$ is sufficiently small (depending on $r$) then $\eta$ lies at Euclidean distance at least $2^{-n+10}$ from $\bdy_{\op{B}} R$, so there is a simple path $P$ in $\left(2^{-n}\BB Z^2\right)\cap  \left( B_{2^{-n+10}}(\eta) \setminus B_{2^{-n}}(\eta)  \right)$ from $\bdy_{\op{L}} R $ to $\bdy_{\op{R}} R $ which is contained in $R\cap D_r$. 
	By the occurrence of the event in~\eqref{eqn-sle-dim}, we have $\# P\leq C_1 2^{n(3+\zeta)/2}$. 
	The occurrence of the events in~\eqref{eqn-sle-gff-compare} and~\eqref{eqn-max-sle-nbd} then show that~\eqref{eqn-gff-perc} holds for large enough $n$. 
\end{proof}

\old{
\begin{proof}[Proof of Lemma~\ref{lem-gff-perc}]
We will prove the statement for $2^{-n}\BB Z^2$; the proof for the dual lattice $ 2^{-n}( \BB Z^2 + (1/2,1/2) ) $ follows from exactly the same argument. 
The basic idea of the proof is to consider a path which stays close to an SLE$_4$ level line of $h$, in the sense of~\cite{ss-contour}. 
Since~\cite{ss-contour} only considers level lines of a GFF on a proper simply connected domain with piecewise linear boundary data and since we want our level line to stay in $R$, we will need to consider such a GFF on a domain slightly larger than $R$. 
 
For $r\geq 0$, let $R_r$ be the rectangle with the same center as $R$ and whose side lengths are $1+r$ times the side lengths of $R$ and let $x_r$ and $y_r$, respectively, be the midpoints of $\bdy_{\op{L}} R_r$ and $\bdy_{\op{R}} R_r$. 
Let $\lambda$ be the constant from~\cite{ss-contour} ($\lambda= \pi/2$ with our choice of normalization).
Let $h'$ be a GFF on $R_r$ such that $h'$ has boundary data $\lambda$ (resp.\ $-\lambda$) on the counterclockwise (resp.\ clockwise) boundary arc from $x_r$ to $y_r$, i.e., $h'$ is a zero-boundary GFF on $R_r$ plus the harmonic function on $R_r$ with this boundary data. 
 
Fix $s\in (0,1)$. We will show that~\eqref{eqn-gff-perc} holds with probability at least $1-s$ if $n$ is sufficiently small. 
Recall that $h|_{R_r}$ can be written in the form $h|_{R_r}=h^{R_r} + \frk h$, where $h^{R_r}$ is a zero-boundary GFF in $R_r$, $\frk h$ is a Gaussian harmonic function in $R_r$, and $h^{R_r}$ and $\frk h$ are independent. Since $\frk h$ is continuous on $R$ and since the boundary data for the field $h'$ above is bounded, we can find a constant $C_0>0$ (depending on $r$ and $s$) and a coupling of $h$ and $h'$ such that a.s.\ $h-h'$ is a continuous function and 
\eqb \label{eqn-sle-gff-compare}
\BB P\left[ \sup_{z\in R}|h(z)-h'(z)|  \leq  C_0 \right]  \geq 1 - \frac{s}{10} .
\eqe

\begin{figure}[t!]
	\begin{center}
		\includegraphics[scale=1]{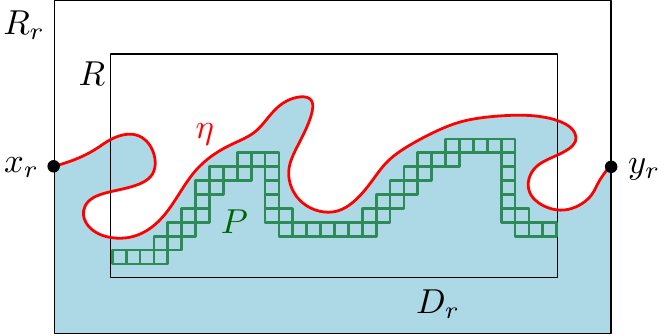}
		\caption{Illustration of the proof of Lemma \ref{lem-gff-perc}. By \cite{ss-contour}, for a particular Dirichlet GFF $h'$ in $R_r$ the level line $\eta$ has the law of an  SLE$_4$. Since $h'$ is smaller near $\eta$ than at a ``typical'' interior point, with high probability we can find a path $P$ in $2^{-n}\BB Z^2$ which connects $\partial_{\op{L}} R$ and $\partial_{\op{R}} R$ along which $|h_{2^{-n-1}}(z)| \vee |h_{2^{-n}}(z)|$ is small. We need $2^{n(3/2 + o_n(1))}$ squares since the Minkowski dimension of $\eta$ is $3/2$. By coupling $h$ and $h'$ so they are close in $R$ we obtain the statement of the lemma. }
		\label{fig-sle-path}
	\end{center} 
\end{figure} 

By the main result of~\cite{ss-contour} we may define the level line $\eta$ of $h'$ from $x_r$ to $y_r$ as a chordal SLE$_4$ curve from $x_r$ to $y_r$ in $R_r$ which is determined by $h'$. See Figure~\ref{fig-sle-path}.
We first claim that for small enough $r > 0$,  
\eqb \label{eqn-sle-contained}
\BB P\left[ \eta \subset R\cup B_{\sqrt r}(x_r) \cup B_{\sqrt r}(y_r) \right] \geq 1 - \frac{s}{10} .
\eqe
Indeed, this follows since $\eta$ a.s.\ does not touch $\partial R_r$ except at $x_r$ and $y_r$ and since the law of $\eta$ is the same for all $r$ modulo scaling. 
Henceforth fix $r$ so that~\eqref{eqn-sle-contained} holds. 

By the main result of~\cite{lawler-rezai-nat}, $\eta$ has finite $3/2$-dimensional Minkowski content a.s. 
This implies that we can find a constant $C_1 > 0$ depending only on $s$ and $R$ such that 
\eqb \label{eqn-sle-dim}
\BB P\left[ \op{area} B_{2^{-n+10}}(\eta) \leq C_1 2^{-n(1-\zeta)/2} \right] \geq 1 -\frac{s}{10} ,
\eqe
where $B_{2^{-n+10}}(\eta)$ denotes the Euclidean $2^{-n+10}$-neighborhood of $\eta$ and $\op{area}$ denotes Lebesgue measure. 
  
Let $D_r\subset R_r$ denote the simply connected domain whose boundary is the union of $\partial_{\op{B}} R_r$, $\eta$, and the segments of $\partial_{\op{L}} R_r$ and $\partial_{\op{R}} R_r$ which lie below $x_r$ and $y_r$, respectively. 
The curve $\eta$ is a local set for $h'$ in the sense of~\cite[Lemma 3.9]{ss-contour}. In particular, if we condition on $\eta$, then the field $h'|_{D_r}$ has the law of a zero-boundary GFF on $D_r$ plus the constant function $\lambda$. For any $z\in D_r$ let $\op{CR}(z,D_r)$ denote the conformal radius of $z$ in $D_r$. By \cite[Proposition 3.2]{shef-kpz}, under the conditional law given $\eta$, for each $\delta>0$ such that $B_\delta(z)\subset D_r$, the random variable $h'_\delta(z)$ is Gaussian with mean $\lambda$ and variance $\op{Var}(h'_\delta(z) | \eta)=-\log\delta+\log \op{CR}(z,D_r)$. By the Koebe quarter theorem, $\frac14 \op{dist}(z,\bdy D_r)\leq \op{CR}(z;D_r) \leq 4 \op{dist}(z,\bdy D_r)$. Consequently, 
\eqbn
\op{Var}\left( h'_{2^{-n}}(z) | \eta \right) \vee \op{Var}\left( h'_{2^{-n-1}}(z) | \eta \right) \preceq 1 ,\quad\forall z\in \left( B_{2^{-n+10}}(\eta) \setminus B_{2^{-n}}(\eta)  \right) \cap D_r ,
\eqen
with a universal implicit constant. 
If~\eqref{eqn-sle-dim} holds, then the number of points in $(2^{-n}\BB Z^2)\cap B_{2^{-n+10}}(\eta)$ is at most $C_1 2^{n(3+\zeta)/2 + 2}$. By the Gaussian tail bound and a union bound over all such points, we see that if $n$ is sufficiently large, then whenever the event in~\eqref{eqn-sle-dim} occurs, it a.s.\ holds with conditional probability at least $1-s/10$ given $ \eta$ that
\eqb \label{eqn-max-sle-nbd}
  \max\left\{  |h'_{2^{-n-1}}(z)| \vee |h'_{2^{-n}}(z)|  :   z \in \left(2^{-n}\BB Z^2\right)\cap  \left( B_{2^{-n+10}}(\eta) \setminus B_{2^{-n}}(\eta)  \right) \cap D_r \right\}   \leq  \frac{\zeta}{2} \log 2^n   .
\eqe 

Now assume that the events in~\eqref{eqn-sle-gff-compare}, \eqref{eqn-sle-contained}, \eqref{eqn-sle-dim}, and~\eqref{eqn-max-sle-nbd} all occur, which happens with probability at least $1-s$. 
The event in~\eqref{eqn-sle-contained} implies that if $n$ is sufficiently small then $\eta$ lies at Euclidean distance at least $2^{-n+10}$ from $\bdy_{\op{B}} R$, so there is a simple path $P$ in $\left(2^{-n}\BB Z^2\right)\cap  \left( B_{2^{-n+10}}(\eta) \setminus B_{2^{-n}}(\eta)  \right)$ from $\bdy_{\op{L}} R $ to $\bdy_{\op{R}} R $ which is contained in $R\cap D_r$. 
By the occurrence of the event in~\eqref{eqn-sle-dim}, we have $\# P\leq C_1 2^{n(3+\zeta)/2}$. 
The occurrence of the events in~\eqref{eqn-sle-gff-compare} and~\eqref{eqn-max-sle-nbd} then show that~\eqref{eqn-gff-perc} holds for large enough $n$. 
\end{proof}}

\begin{proof}[Proof of Lemma~\ref{lem-dist-along-line}]
We will prove the statement for $h$; the statement for $\wh h$ is then immediate from Lemma~\ref{lem-gff-compare}. 
Given $\ep  >0$, choose $n\in\BB N$ so that $2^{-(Q - \zeta) n } \leq \ep \leq 2^{-(Q - \zeta) n + 1}$. 
Lemma~\ref{lem-gff-perc} shows that with probability tending to 1 as $\ep \rta 0$, there is a path $P$ of dyadic squares contained in $R$ with side length $2^{-n}$ such that
\eqb
\# P \leq 2^{(3/2 + \zeta)n } \leq \ep^{-(\xi - o_\zeta(1))}  \quad \text{and} \quad 
|h_{2^{-n-1}}(v_S)|  \leq \zeta \log 2^n   ,\quad\forall S \in P ,
\eqe 
where the $o_\zeta(1)$ tends to 0 as $\zeta\rta 0$ at a rate depending only on $Q$. 
For $S\in P$, the condition that $|h_{2^{-n-1}}(v_S)| \leq \zeta \log 2^n $ implies that
\eqbn
M_h(S)  = e^{h_{2^{-n-1}}(v_S)} 2^{-  Q n} \leq 2^{-(Q-\zeta) n} \leq \ep .
\eqen
Hence, either $S$ or one of its dyadic ancestors is contained in $\mcl S_h^\ep(R)$. This shows that~\eqref{eqn-dist-along-line} holds. 
\end{proof}

\subsection{Maximal distance between large squares}
\label{sec-macro-square-dist}

Throughout the rest of this section, we let $\xi = \xi(Q)>  0$ be any exponent for which the conclusion of Lemma~\ref{lem-dist-along-line} is satisfied. 
Many of our estimates will depend on $\xi$. 

Let $h$ be a whole-plane GFF normalized so that its circle average over $\bdy \BB D$ is zero. 
The goal of this subsection is to show that with high probability, the $D_h^\ep$-distance between any two large squares (of at least polynomial size) in $\mcl S_h^\ep$ is at most a polynomial function of $\ep^{-1}$. 
We will also show that each such square lies at at most polynomial distance from $\bdy\BB S$. 

\begin{prop} \label{prop-macro-square-dist}
Fix $\zeta\in (0,1)$. 
For each $\beta  > 0$, it holds with polynomially high probability as $\ep\rta 0$ that 
\eqb \label{eqn-macro-square-dist}
  D_{h}^\ep(S,\wt S; \BB S) \vee D_{ h}^\ep(S , \bdy \BB S ; \BB S) \leq \ep^{ - f(\beta) - \zeta}   , \quad \forall  S,\wt S \in \mcl S_{ h}^\ep(\BB S)  \quad \text{with} \quad |S|  \wedge |\wt S| \geq \ep^\beta ,
\eqe 
where 
\eqb \label{eqn-macro-square-function}
f(\beta)  =f(\beta,Q) := \xi + (2-Q) \beta \xi + \beta .
\eqe 
Moreover, the same is true with the white noise field $\wh h$ of~\eqref{eqn-wn-decomp} in place of $h$. 
\end{prop}

The particular value of $f(\beta)$ in~\eqref{eqn-macro-square-function} is not important for our purposes. Before proceeding with the proof of Proposition~\ref{prop-macro-square-dist}, we note that it implies Proposition~\ref{prop-ptwise-distance} (which asserts that the $D_h^\ep$-distance between two points grows at most polynomially in $\ep$). 

\begin{proof}[Proof of Proposition~\ref{prop-ptwise-distance}]
By re-scaling space and using that $h_{|z+w|}((z+w)/2)$, say, is Gaussian with variance depending only on $|z+w|/2$, we can assume without loss of generality that each of $z$ and $w$ is contained in the interior of $\BB S$. 
 
To prove the lower bound for $D_h^\ep(z,w)$ with $\ul \xi = 1/(2+Q)$, we first use Lemma~\ref{lem-max-square-diam}, a union bound over dyadic values of $\ep$, and the Borel-Cantelli lemma to find that a.s.\ $\max\{|S| : S\in \mcl S_h^\ep(\BB S)\} \leq \ep^{1/(2+Q) + o_\ep(1)}$. 
Since the sum of the side lengths of the squares along any path in $\mcl S_h^\ep$ from $z$ to $\{w\} \cup \bdy \BB S$ must be at least the Euclidean distance from $z$ to $\{w\} \cup\bdy \BB S$, we get that a.s.\ $D_h^\ep(z,w) \geq \ep^{-1/(2+Q) + o_\ep(1)}$. 

For the upper bound, let $S$ be a dyadic square contained in $\BB S$ with side length in $[\ep^{1/Q+\zeta},2\ep^{1/Q+\zeta})$. Then $h_{|S |/2}(v_{S })$ is Gaussian with variance $\log \ep^{-1/Q -\zeta}  + O_\ep(1)$, so by the Gaussian tail bound $e^{h_{|S |/2}(v_{S })}|S |^Q \leq \ep$ with polynomially high probability as $\ep\rta 0$. 
By applying this to dyadic squares containing each of $z$ and $w$, we find that with polynomially high probability as $\ep\rta 0$, each of $z$ and $w$ is contained in a square of $\mcl S_h^\ep(\BB S)$ with side length at least $\ep^{1/Q +\zeta}$. Combining this with Proposition~\ref{prop-macro-square-dist} and possibly shrinking $\zeta$ shows that with polynomially high probability as $\ep\rta 0$, 
\eqb 
D_h^\ep(z,w ; \BB S) \leq \ep^{- \xi - (2-Q) \xi/Q - 1/Q - \zeta} .
\eqe 
A union bound over dyadic values of $\ep$ now concludes the proof for $h$. 
The statement for $\wh h$ follows from Lemma~\ref{lem-tr-compare-square}. 
\end{proof}

For most of the proof of Proposition~\ref{prop-macro-square-dist}, we will work with $\wh h$ instead of $h$. We will transfer from this field to $h$ at the very end of the proof using Lemma~\ref{lem-tr-compare-square}. 
As discussed in the outline above, we want to use Lemma~\ref{lem-dist-along-line} to build a ``grid" of paths in $\BB S$ which hit every square of side length at least $\ep^\beta$. To do this, we will need to apply the lemma to a large number of different rectangles simultaneously, so we will need a quantitative bound for the probability that the distance between two sides of a rectangle is larger than expected. We will obtain such a concentration bound using a percolation-style argument. The statements and proofs of the next three lemmas are roughly similar to arguments appearing in~\cite[Section 3.2]{dg-lqg-dim} but with some extra complications since a priori we only have an upper bound for the distance between two sides of the rectangle, not for the point-to-point distance.  

We first prove a bound for the distance between the two sides of a large square which holds with superpolynomially high probability.

\begin{lem} \label{lem-rectangle-perc}  
For $n\in\BB N$, let $\mcl R_n := [0,2n]\times [0,n]$. 
For each fixed $\zeta  \in (0,1)$, there exists $a_0,a_1  , A   > 0$ (depending only on $\zeta$ and $Q$) such that for $n\in\BB N$ and $\ep > 0$,  
\eqb \label{eqn-rectangle-perc}
\BB P\left[ D_{\wh h }^\ep\left(\bdy_{\op{L}} \mcl R_n , \bdy_{\op{R}} \mcl R_n ; \mcl R_n \right) \leq  n^2 \max\left\{ A ,  e^{n^{1/2}}  \ep^{ - \xi -\zeta } \right\} \right] \geq 1 -  a_0 e^{-a_1 n  }  .
\eqe 
\end{lem} 

When we apply Lemma~\ref{lem-rectangle-perc}, we will typically take $n$ to be of order $(\log \ep^{-1})^p$ for $1 < p < 2$, so that $e^{-a_1 n}$ is smaller than any power of $\ep$ and $n^2$ and $e^{n^{1/2}}$ are of order $\ep^{o_\ep(1)}$. 

Lemma~\ref{lem-rectangle-perc} will be proven using a percolation-style argument which requires exact independence for the values of the field in squares which lie at macroscopic distance from one another. So, we will need to work with the truncated white-noise field $\wh h^\tr$ instead of $\wh h$ itself. 
We will apply the following lemma to bound the length of a ``hard direction" crossing for $2\times 1$ or $1 \times 2$ rectangles contained in $\mcl R_n$, then concatenate such crossings to prove Lemma~\ref{lem-rectangle-perc}.

\begin{lem} \label{lem-local-dist}  
Let $\xi = \xi(Q)$ be as in Lemma~\ref{lem-dist-along-line} and fix $\zeta \in (0,1)$. For a $1\times 1$ square $S\subset \BB C$ with corners in $\BB Z^2$ and $\ep\in (0,1)$, let $E_S^\ep$ be the event that the following is true:  
For each of the four $2\times 1$ rectangles $R$ with corners in $\BB Z^2$ which contain $S$, one has
\eqbn
  D_{\wh h^\tr}^\ep\left( \bdy_{\op{L}} R, \bdy_{\op{R}} R ; R \right) \leq  \ep^{-\xi - \zeta }  ;
\eqen  
and the same holds for $1\times 2$ rectangles with $\bdy_{\op{B}} R$ and $\bdy_{\op{T}} R$ in place of $\bdy_{\op{L}} R$ and $\bdy_{\op{R}} R$.
Then $\BB P[E_S^\ep] \rta 1$ as $\ep\rta 0$, uniformly over all such squares $S$.  
\end{lem}
\begin{proof}
In the case when $S = \BB S$ is the standard unit square, this follows from four applications of Lemma~\ref{lem-dist-along-line} combined with Lemma~\ref{lem-tr-compare-square}. For a general choice of $S$, the lemma follows from the case $S = \BB S$ and the fact that the law of $\wh h^\tr$ is invariant under spatial translations.
\end{proof}

\begin{proof}[Proof of Lemma~\ref{lem-rectangle-perc}]
We will show that there are constants $a_0,a_1 , A > 0$ as in the statement of the lemma such that for $n\in\BB N$ and $\ep > 0$, 
\eqb \label{eqn-rectangle-perc-truncated}
\BB P\left[ D_{\wh h^\tr}^\ep\left(\bdy_{\op{L}} \mcl R_n , \bdy_{\op{R}} \mcl R_n  ; \mcl R_n \right) \leq  n^2 \max\left\{ A ,    \ep^{-\xi -\zeta } \right\} \right] \geq 1 -  a_0 e^{-a_1 n}  .
\eqe 
Combining this with~\cite[Lemma 3.5]{dg-lqg-dim} (applied with $r = 4 n$ and $C = e^{c n^{1/2}}$ for an appropriate constant $c>0$) will yield~\eqref{eqn-rectangle-perc}. 

We will prove~\eqref{eqn-rectangle-perc-truncated} via a percolation argument. 
Let $p\in (0,1)$ be a small universal constant to be chosen later. 
We assume without loss of generality that $n\geq 3$ and let $\mcl S(\mcl R_n)$ be the set of unit side length squares\footnote{The reason for considering $[1,2n-1]\times [1,n-1]$ instead of $\mcl R_n $ is so that if $S\in \mcl S(\mcl R_n)$, then each  of the four $1\times 2$ or $2\times 1$ rectangles with corners in $\BB Z^2$ which contain $S$ are contained in $\mcl R_n $.} 
$S\subset [1,2 n-1] \times [1,n-1]$ with corners in $\BB Z^2$. 
We view $\mcl S(\mcl R_n)$ as a graph with two squares considered to be adjacent if they share an edge. We define the \emph{left boundary} of $\mcl S(\mcl R_n)$ to be the set of squares in $\mcl S(\mcl R_n)$ which intersect the left boundary of $[-1,2n+1]\times [0,n]$. We similarly define the right, top, and bottom boundaries of $\mcl S(\mcl R_n)$.  

For each square $S\in \mcl S(\mcl R_n)$ and $\ep \in (0,1)$, let $E_S^\ep$ be the event of Lemma~\ref{lem-local-dist}.  
By Lemma~\ref{lem-local-dist}, we can find $\ep_* = \ep_*(p,\zeta,Q) > 0$ such that  
\eqb \label{eqn-perc-prob}
\BB P[E_S^\ep ] \geq 1-p , \quad \forall S\in \mcl S(\mcl R_n) , \quad \forall \ep\in (0,\ep_*] .
\eqe 

We claim that if $p$ is chosen sufficiently small, then for appropriate constants $a_0,a_1 > 0$ as in the statement of the lemma, it holds for each $\ep \in (0,\ep_*]$ and $n\in\BB N$ that with probability at least $1- a_0 e^{-a_1 n}$, we can find a path $\mcl P$ in $\mcl S(\mcl R_n)$ from the left boundary of $\mcl S(\mcl R_n)$ to the right boundary of $\mcl S(\mcl R_n)$ consisting of squares for which $E_S^\ep$ occurs. 

Assume the claim for the moment. 
On the event that a path $\mcl P$ as in the claim exists, for each $S \in \mcl P$ and each of the four $2\times 1$ or $1\times 2$ rectangles $R \supset S$ from the definition of $E_S^\ep$, one can find a path of squares $\wt P_R$ in $\mcl S_{\wh h^\tr}^\ep(R)$ between the two shorter sides of $R$ with length at most $\ep^{-\xi-\zeta}$. Since each such rectangle $R$ is contained in $\mcl R_n$, the set of squares $P_R \in \mcl S_h^\ep(\mcl R_n)$ which contain the squares in $\wt P_R$ contains a path of squares in $\mcl S_h^\ep(\mcl R_n)$ between the two shorter sides of $R$ of length at most $\ep^{-\xi - \zeta}$. 
Let $P_S$ be the union of the sets of squares $P_R$ for the four rectangles $R$ which contain $S$. Then $P_S$ is connected.

If $S,\wt S \in \mcl S(\mcl R_n)$ are two squares which share a side, then topological considerations show that $P_S\cap P_{\wt S} \not=\emptyset$. Hence the union of the sets $P_S$ over all $S\in \mcl P$ contains a path of squares in $\mcl S_h^\ep(\mcl R_n)$ between the left and right boundaries of $\mcl R_n$ of length at most a universal constant times $ n^2 \ep^{-1-\xi}$ (note that there are $2n^2$ squares in $\mcl S(\mcl R_n)$).   
Since this happens with probability at least $1-a_0 e^{-a_1 n}$ whenever $\ep\in (0,\ep_*]$, we get that~\eqref{eqn-rectangle-perc-truncated} holds for $\ep \in (0,\ep_*]$. 
Since $  D_{ \wh h_1^{\tr}}^\ep\left(\bdy_{\op{L}} \mcl R_n , \bdy_{\op{R}} \mcl R_n  ; \mcl R_n \right) $ can only increase when $\ep$ decreases, it follows that~\eqref{eqn-rectangle-perc-truncated} is true for general $\ep > 0$ with $A \asymp \ep_*^{- \xi - \zeta} $. 

It remains only to prove the above claim. 
Let $\mcl S^*(\mcl R_n)$ be the graph whose squares are the same as the squares of $\mcl S(\mcl R_n)$, but with two squares considered to be adjacent if they share a corner or an edge, instead of only considering squares to be adjacent if they share an edge.
By planar duality, it suffices to show that if $p ,a_0, a_1$ are chosen appropriately, then for $\ep\in (0,\ep_*]$ it holds with probability at least $1-a_0 e^{-a_1 n}$ that there does \emph{not} exist a simple path in $\mcl S^*(\mcl R_n)$ from the top boundary to the bottom boundary of $\mcl S (\mcl R_n)$ consisting of squares for which $E_S^\ep$ does not occur. This will be proven by a standard argument for subcritical percolation.

By the definition~\eqref{eqn-wn-truncate} of $\wh h^{\tr}$, the event $E_S^\ep$ is a.s.\ determined by the restriction of the white noise $W$ to $B_2(S) \times \BB R$ (here $B_2(S)$ denotes the Euclidean neighborhood of radius 2). In particular, $E_S^\ep$ and $E_{\wt S}^\ep$ are independent whenever $B_2(S) \cap B_2(\wt S) = \emptyset$. 
For each fixed deterministic simple path $ \mcl P^*$ in $\mcl S^*(\mcl R_n)$, we can find a set of at least $|\mcl P^*|/100$ squares $S$ hit by $\mcl P^*$ for which the neighborhoods $B_2(S)$ are disjoint. 
By~\eqref{eqn-perc-prob}, applied once to each of these $|\mcl P^*|/100$ squares, if $\ep\in (0,\ep_*]$ then the probability that $E_S^\ep$ fails to occur for every square in $ \mcl P^*$ is at most $p^{| \mcl P^*|/100}$.

We now take a union bound over all paths $\mcl P^*$ in $\mcl S^*(\mcl R_n)$ connecting the top and bottom boundaries. For $k \in [n,2n^2]_{\BB Z}$, the number of such paths with $|\mcl P^*| = k$ is at most $ n 8^{k+1}$ since there are $2 n$ possible initial squares adjacent in the top boundary of $\mcl R_n$ and 8 choices for each step of the path. Combining this with the estimate in the preceding paragraph, we find that for $\ep \in (0,\ep_*]$ the probability of a top-bottom crossing of $\mcl S^*(\mcl R_n)$ consisting of squares for which $E_S^\ep$ does not occur is at most
\eqbn
  n \sum_{k=n}^{2 n^2} p^{k/100} 8^{k+1} , 
\eqen
which is bounded above by an exponential function of $n$ provided we take $p < 8^{-100}$.  
\end{proof}

Re-scaling space and applying Lemma~\ref{lem-rectangle-perc} in several different rectangles contained in $\BB S$ leads to the following. 
 
\begin{lem} \label{lem-rectangle-dist}
For each $\zeta\in (0,1)$, there exists $\lambda= \lambda(\zeta,Q) > 0$ such that the following is true.  
For $m\in\BB N$ and $\ep  > 0$, it holds with probability $1-O_m(e^{-\lambda m})$ as $m\rta\infty$, at a rate which is uniform in $\ep$, that for each $ 2^{-m+1} \times 2^{-m}$ rectangle $R\subset \BB S$ with corners in $2^{-m} \BB Z^2$, 
\eqb \label{eqn-rectangle-dist}
 D_{\wh h}^\ep\left( \bdy_{\op{L}} R , \bdy_{\op{R}} R ; R  \right)  
\leq \max\left\{ m^3 , \, \ep^{-\xi - \zeta }  2^{- (\xi Q-\zeta) m  }  \exp\left( \xi   \min_{z\in R } \wh h_{2^{-m} }(z)  \right) \right\}.
\eqe  
\end{lem}
\begin{proof}
Fix $\wt\zeta\in (0,1)$ to be chosen later, in a manner depending only on $\zeta$ and $Q$. 
Also set 
\eqbn
n_m := \lfloor \log_2   m^{3/2} \rfloor ,\quad \forall m \in \BB N  
\eqen
(in fact, $n_m = \lfloor \log_2   m^p \rfloor$ for any $1 < p < 2$ would suffice). 

By~\cite[Lemmas 3.5 and 3.6]{dg-lqg-dim} (the latter is applied with $\delta = 2^{-m}$ and $A = 2^{n_m} \asymp m^{3/2}$), it holds with exponentially high probability as $m\rta\infty$ that 
\eqb \label{eqn-use-mid-scale-compare}  
\max_{z\in \BB S} |\wh h_{2^{-m}}(z)| \leq (2+\wt\zeta) \log 2^m \quad \op{and} \quad
  \max_{z,w\in \BB S : |z-w| \leq   2^{-m+2}   } |\wh h_{   2^{-m - n_m  } }(z) - \wh h_{2^{-m}}(w)| \leq \wt\zeta \log 2^m .
\eqe
If $R$ is a $  2^{-m+1}  \times 2^{-m}$ rectangle and $u_R$ denotes its bottom left corner, then $   2^{m+n_m} (R-u_R) $ is the rectangle $\mcl R_{2^{n_m} }$ of Lemma~\ref{lem-rectangle-perc}. Moreover, the field $\{ \wh h_{  2^{-m - n_m} t,  2^{-m - n_m} }  ( 2^{-m - n_m} \cdot  + u_R)\}_{t \in [0,1]  }$ agrees in law with $\{ \wh h_t\}_{t \in [0,1]} $ and is independent from $\wh h_{2^{-m-n_m}   }$, which means that the associated graph distance $D_{\wh h_{0,2^{-m-n_m}}(2^{-m - n_m}\cdot + u_R)}^\ep$ agrees in law with $D_{\wh h}^\ep$ and is independent from $\wh h_{2^{-m-n_m}}$. 
Using~\eqref{eqn-dist-scaling} with $\wh h$ in place of $h$ and with $C =  2^{ m +  n_m}$ and $f = \wh h_{2^{-m-n_m}}$, we therefore get that the conditional law of $ D_{\wh h}^\ep\left( \bdy_{\op{L}} R , \bdy_{\op{R}} R ; R  \right)$ given $\wh h_{  2^{-m - n_m} }$ is stochastically dominated by the law of 
\eqbn
 D_{\wh h }^{T_R\ep}\left(\bdy_{\op{L}} \mcl R_{2^{n_m} } , \bdy_{\op{R}} \mcl R_{2^{n_m} } ; \mcl R_{2^{n_m} }  \right) 
\quad \text{for} \quad 
T_R := 2^{Q ( m+n_m) } \exp\left( -    \max_{z\in R }  \wh h_{ 2^{-m -  n_m} }(z)  \right) .
\eqen
If~\eqref{eqn-use-mid-scale-compare} holds, then 
\allb \label{eqn-rectangle-dist-T}
T_R  
&\geq 2^{ \left(Q + o_m(1) + o_{\wt\zeta}(1) \right) m  }  \exp\left( -  \min_{z\in R }   \wh h_{2^{-m} }(z) \right)   \notag \\
&\geq 2^{ \left(Q + o_m(1) + o_{\wt\zeta}(1) \right) m  }  \exp\left( - \frac{\xi}{\xi-\zeta}  \min_{z\in R }   \wh h_{2^{-m} }(z) \right)   .
\alle
where the $o_{\wt\zeta}(1)$ and $o_m(1)$ are each deterministic and independent of $\ep$ and the $o_{\wt\zeta}(1)$ error is also independent of $m$. Note that in the first inequality in~\eqref{eqn-rectangle-dist-T}, we switched from the maximum of $\wh h_{  2^{-m - n_m}}(z)$ to the minimum of $\wh h_{2^{-m}}(z)$ (which gives a stronger estimate than the maximum) using the second inequality in~\eqref{eqn-use-mid-scale-compare}. Also, in the second inequality in~\eqref{eqn-rectangle-dist-T}, we absorbed a small multiple of $\min_{z\in R }   \wh h_{2^{-m} }(z)$ into the $m o_{\wt\zeta}(1)$ error inside the first exponential using the first inequality in~\eqref{eqn-rectangle-dist-T}. 
 By Lemma~\ref{lem-rectangle-perc} (applied with $T_R \ep$ in place of $\ep$ and $2^{n_m}$ in place of $n$) and a union bound over $O_m(2^{2m})$ rectangles $R$, we obtain the statement of the lemma upon choosing $\wt\zeta$ sufficiently small, in a manner depending only on $\zeta$ and $Q$. 
\end{proof}

By replacing $\min_{z\in R } \wh h_{2^{-m} }(z)$ by the maximum possible value of $\wh h_{2^{-m}}(z)$ on $R$, we can eliminate the dependence on $\wh h_{2^{-m}}$ in Lemma~\ref{lem-rectangle-dist}.

\begin{lem} \label{lem-rectangle-dist-max}
Fix $\zeta \in (0,1)$. With polynomially high probability as $\ep\rta 0$, the following is true. For each $m \in \BB N$ with $2^{-m} \leq \ep^\zeta$ and each $2^{-m+1} \times 2^{-m }$ rectangle $R\subset \BB S$ with corners in $2^{-m} \BB Z^2$,  
\eqb \label{eqn-rectangle-dist-max}
 D_{\wh h}^\ep\left( \bdy_{\op{L}} R , \bdy_{\op{R}} R ; R  \right)  
\leq  \ep^{-\xi - \zeta }  2^{ \left( (2-Q)\xi   -\zeta \right) m  } .
\eqe  
Moreover, the same holds for $2^{-m} \times 2^{-m+1}$ rectangles with corners in $2^{-m}\BB Z^2$, but with $\bdy_{\op{T}} R$ and $ \bdy_{\op{B}} R$ in place of $\bdy_{\op{L}} R$ and $\bdy_{\op{R}} R$. 
\end{lem}

We note that the right side of~\eqref{eqn-rectangle-dist-max} tends to $\infty$ as $m\rta\infty$ ($\ep$ fixed) if $Q < 2$ and tends to 0 as $m\rta\infty$ if $Q >2$.  
This is related to the fact that there are arbitrarily small squares in $\mcl S_{\wh h}^\ep(\BB S)$ when $Q < 2$, hence there are arbitrarily small rectangles with large $D_{\wh h}^\ep$-diameter.

\begin{proof}[Proof of Lemma~\ref{lem-rectangle-dist-max}]
By~\cite[Lemma 3.5]{dg-lqg-dim} and a union bound over all $m\in\BB N$ with $2^{-m} \geq \ep^\zeta$, it holds with polynomially high probability as $\ep\rta 0$ that for each such $m$, $\max_{z\in\BB S} |\wh h_{2^{-m}}(z)| \leq (2+\zeta) \log 2^m$. By combining this with Lemma~\ref{lem-rectangle-dist} and possibly shrinking $\zeta$, we get that with polynomially high probability as $\ep\rta 0$, it holds for each $2^{-m} \leq \ep^\zeta$ and each rectangle $R$ as in the statement of the lemma that 
\eqb \label{eqn-rectangle-dist0}
 D_{\wh h}^\ep\left( \bdy_{\op{L}} R , \bdy_{\op{R}} R ; R  \right)  
\leq \max\left\{ m^3 , \, \ep^{-\xi - \zeta }  2^{ \left( (2-Q)\xi   -\zeta \right) m  }  \right\}.
\eqe 
Since $2^{-m} \leq \ep^\zeta$, the second term in the maximum in~\eqref{eqn-rectangle-dist0} is larger than the first for a small enough choice of $\zeta \in (0,1)$ and $\ep >0$. This gives~\eqref{eqn-rectangle-dist-max}. The statement for $2^{-m} \times 2^{-m+1}$ rectangles follows since the law of $\wh h$ is invariant under rotations by $\pi/2$.
\end{proof}

\begin{figure}[t!]
 \begin{center}
\includegraphics[scale=.85]{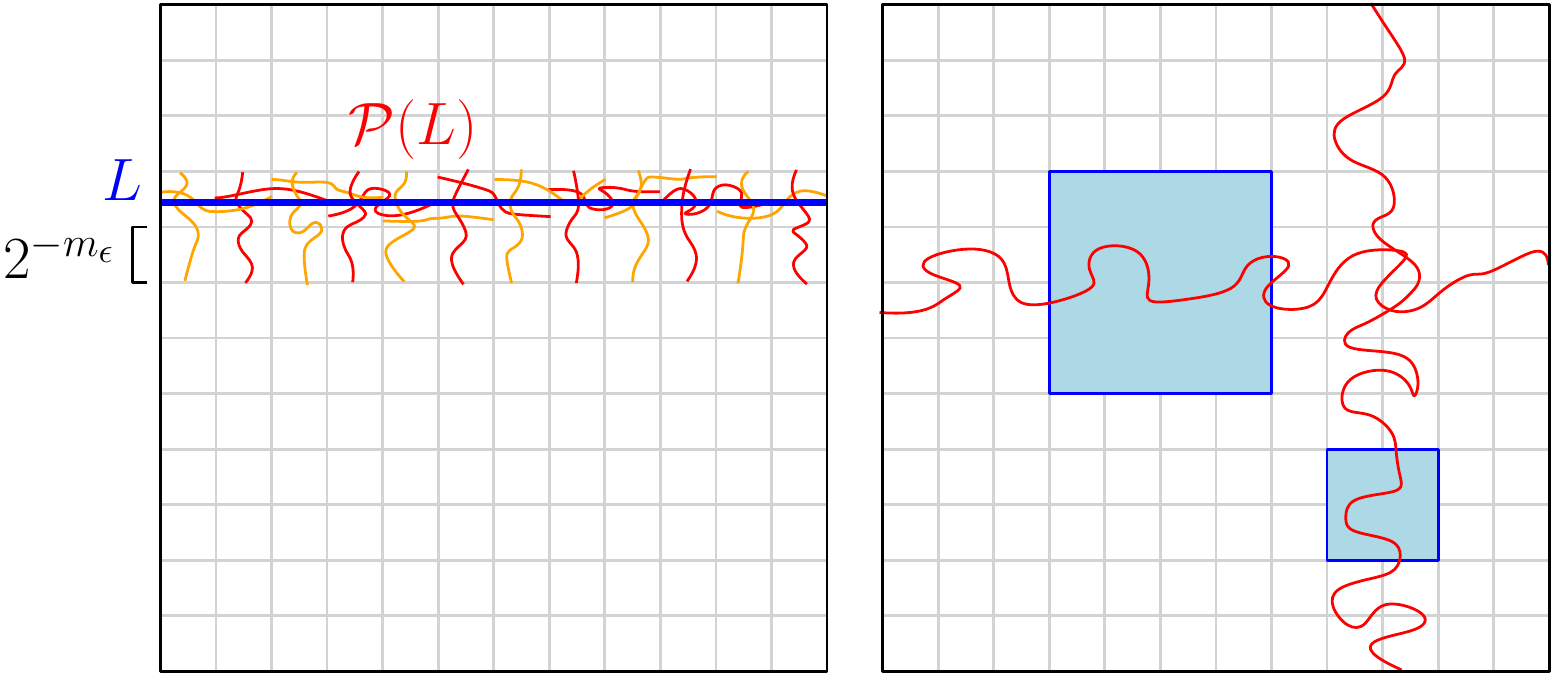}
\vspace{-0.01\textheight}
\caption{ \textbf{Left.} Illustration of the construction of a path of squares near the line segment $L$ in the proof of Proposition~\ref{prop-macro-square-dist}. The paths of squares $P_R$ in some of the $2^{-m_\ep+1} \times 2^{-m_\ep}$ and $2^{-m_\ep } \times 2^{-m_\ep  + 1}$ rectangles $R$ which intersect $L$ are shown in red and orange (the individual squares in the paths are not shown). The union $\mcl P(L)$ of these paths is connected and contains a path near $L$ from the left boundary to the right boundary of $\BB S$. 
\textbf{Right.} To get a path between two given squares of side length at least $\ep^\beta$, we consider the union of the paths (red) corresponding to two line segments $L$ and $L'$ which pass through the two given squares. This union contains a path of squares in $\mcl S_{\wh h}^\ep(\BB S)$. between the two given squares. 
}\label{fig-line-path}
\end{center}
\vspace{-1em}
\end{figure}

\begin{proof}[Proof of Proposition~\ref{prop-macro-square-dist}] 
We will prove the proposition statement with $\wh h$ in place of $h$. The statement for $h$ follows immediately from this and Lemma~\ref{lem-tr-compare-square}. See Figure~\ref{fig-line-path} for an illustration of the proof.

Fix $\beta  >0$ and for $\ep \in (0,1)$, let $m_\ep \in \BB N$ be chosen so that $2^{-m_\ep + 100} \leq \ep^\beta \leq 2^{-m_\ep + 101}$. By Lemma~\ref{lem-rectangle-dist-max}, it holds with polynomially high probability as $\ep\rta 0$ that for each $2^{-m_\ep +1} \times 2^{-m_\ep }$ rectangle $R\subset \BB S$ with corners in $2^{-m_\ep} \BB Z^2$, there is a path of squares $\wt P_R$ in $\mcl S_{\wh h}^\ep(R)$ joining the left and right boundaries of $R$ with 
\eqbn
|\wt P_R| \leq \ep^{- \xi - (2-Q) \beta \xi - \zeta } ;
\eqen
and the same holds for $2^{-m_\ep} \times 2^{-m_\ep+1}$ rectangles but with $\bdy_{\op{T}} R$ and $ \bdy_{\op{B}} R$ in place of $\bdy_{\op{L}} R$ and $\bdy_{\op{R}} R$. Henceforth assume that this is the case. 

For each rectangle $R$ as above, let $P_R$ be the set of squares of $\mcl S_{\wh h}^\ep(\BB S)$ which contain the squares in $\wt P_R$ (note that each square of $\mcl S_{\wh h}^\ep(R)$ is contained in a square of $\mcl S_{\wh h}^\ep(\BB S)$ since $R\subset \BB S$). 
For a horizontal or vertical line segment $L$ joining the left and right or top and bottom boundaries of $S$, let $\mcl P(L)$ be the union of the paths of squares $P_R$ over all $2^{-m_\ep +1} \times 2^{-m_\ep }$ or $2^{-m_\ep} \times 2^{-m_\ep +1 }$ rectangles $R$ contained in $\BB S$ with corners in $2^{-m_\ep}\BB Z^2$ which intersect $L$. Then $\mcl P(L)$ is connected and contains a path in $\mcl S_{\wh h}^\ep(\BB S)$ between the left and right boundaries of $\BB S$ consisting of squares which all intersect the Euclidean $\ep^{-\beta}/20$-neighborhood of $L$ (see Figure~\ref{fig-line-path}, left). Furthermore, the number of rectangles $R$ involved in the above union is at most $O_\ep(\ep^{-\beta})$, so the length of this path is at most $O_\ep(\ep^{-f(\beta)-\zeta})$. 
 
If $S,\wt S \in \mcl S_{\wh h}^\ep(\BB S)$ are squares with side length at least $\ep^\beta$, then we can find a horizontal line segment $L$ and a vertical line segment $L'$ as above whose union is connected and contains the centers of each of $S$ and $\wt S$. The union of the corresponding sets of squares $\mcl P(L)$ and $\mcl P(L')$ as above is connected in $\mcl S_{\wh h}^\ep(\BB S)$ and passes through (and therefore contains) each of $S$ and $\wt S$. Thus~\eqref{eqn-macro-square-dist} holds with $\wh h$ in place of $h$.
\end{proof}

\subsection{Distance from a small square to a large square}
\label{sec-singularity-dist}

Fix $\alpha \in (Q,2)$ and $z\in \BB S$ with $\op{dist}(z,\bdy\BB S) \geq 1/4$.  As in the discussion at the beginning of this section, we will eventually take $\alpha$ to be very close to $Q$. Let $h$ be a whole-plane GFF normalized so that $h_1(0) = 0$ and let $\wh h$ be the white-noise decomposition field as in~\eqref{eqn-wn-decomp}. Define the fields 
\eqb \label{eqn-singularity-fields}
h^\alpha := h - \alpha \log|\cdot - z| \quad \op{and} \quad \wh h_t^\alpha := \wh h_t - \alpha \log |\cdot  -z| ,\quad \forall t \in (0,1).
\eqe
We note that the condition that $\alpha  > Q$ ensures that $\mcl S_{h^\alpha}^\ep(\BB S)$ has a singularity at $z$ with high probability and the condition that $\alpha < 2$ ensures that $h$ possesses $\alpha$-thick points a.s.\ (see~\cite{hmp-thick-pts} and Section~\ref{sec-small-square-count} below). The goal of this section is to prove the following statement, which gives an upper bound for the $\mcl S_{h^\alpha}^\ep(\BB S)$-graph distance from a large square to a neighborhood of the singularity $z$.

\begin{prop} \label{prop-singularity-dist}
Fix $\zeta\in (0,1)$ and $K  > 1$. With polynomially high probability as $\ep\rta 0$, at a rate which is uniform in $z$, there exists a square $S\in \mcl S_{h^\alpha}^\ep(\BB S)$ with $|S| \geq \ep^{1/(Q-\zeta) }$ and
\eqb \label{eqn-singularity-dist}
D_{ h^\alpha}^\ep\left( S , \bdy B_{\ep^K}(z) ; \BB S     \right)   \leq   \ep^{ - g(\alpha ,  K ) - \zeta}    ,
\eqe
where
\eqb \label{eqn-singularity-dist-function}
g(\alpha,  K ) = g(\alpha,K,Q) := \max\left\{ (\alpha - Q  ) \xi  K  + \xi  ,    \xi +  \frac{2-Q}{Q}   \xi + \frac{1}{Q}      \right\} .
\eqe  
The same is true with $\wh h^\alpha$ in place of $h^\alpha$. 
\end{prop}

We will eventually send $K \rta \infty$ and $\alpha \rta Q$ in order to get an arbitrarily large growth exponent for a graph-distance ball. The only important feature of the function $g(\alpha,K)$ of~\eqref{eqn-singularity-dist-function} for our argument is that when we set $\alpha = Q + 1/K$, it holds that $g(Q+1/K , K)$ is bounded above by a constant which depends only on $Q$. 

As in Section~\ref{sec-macro-square-dist}, for the proof of Proposition~\ref{prop-singularity-dist} we will mostly work with $\wh h^\alpha$ instead of $h^\alpha$. 
To prove the proposition, we will first establish for each $n\in\BB N$ an upper bound for the minimum length of a path in $\mcl S_{\wh h^\alpha}(\BB S)$ between the inner and outer boundaries of a square annulus consisting of points which lie at Euclidean distance of order $2^{-n}$ from $z$; and an upper bound for the minimum length of a path in $\mcl S_{\wh h^\alpha}(\BB S)$ which disconnects the inner and outer boundaries of such an annulus. This is done in Lemma~\ref{lem-annulus-path}, with Lemma~\ref{lem-annulus-rect} as an intermediate step. We will then concatenate paths of minimum length in a logarithmic number of such annuli to produce a path from a large square (contained in one of the outer annuli) to $B_{\ep^{K}}(z)$. See Figure~\ref{fig-annulus-path} for an illustration.

\begin{figure}[t!]
 \begin{center}
\includegraphics[scale=.85]{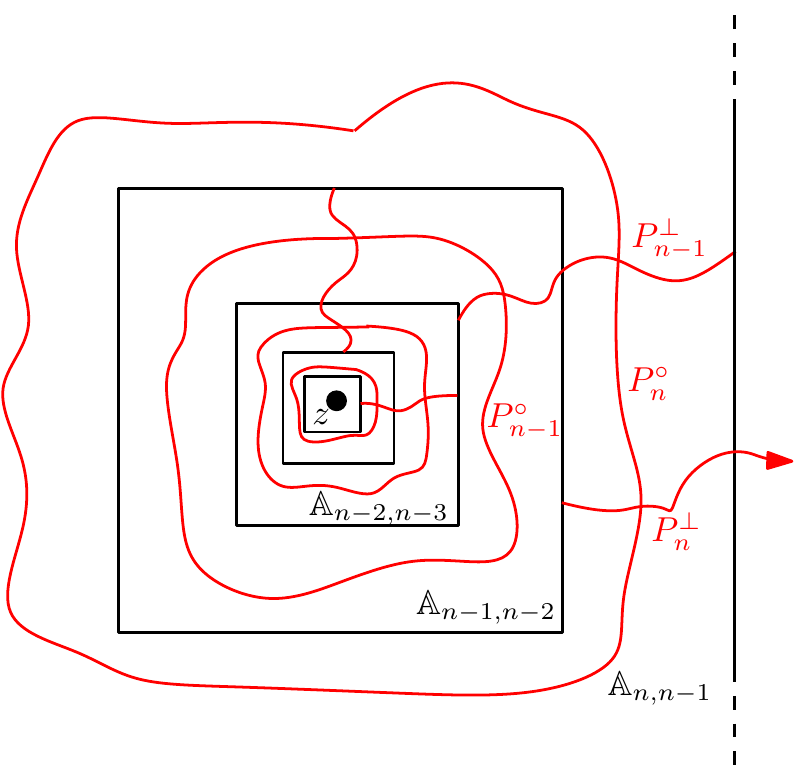}
\vspace{-0.01\textheight}
\caption{ The annuli $\BB A_{n,n-1}$ defined in~\eqref{eqn-annulus-def} and the paths of squares $P_n^\perp$ and $P_n^\circ$ of Lemma~\ref{lem-annulus-path} for several values of $n$ (individual squares along the paths are not shown). Note that the inner and outer boundaries of the annuli are not exactly concentric. The union of these paths over all $n\in [\log_2 \ep^{-\wt\zeta/(2\alpha)} , \log_2 \ep^{-K}]_{\BB Z}$ contains a path in $\mcl S_{\wh h^\alpha}^\ep$ from a nearly macroscopic annulus to $B_{\ep^K}(z)$. 
}\label{fig-annulus-path}
\end{center}
\vspace{-1em}
\end{figure}

Before proceeding with the proof, we define the annuli we will consider. We want to study nested square annuli surrounding $z$ with dyadic radii, but we want the corners of the squares involved to be dyadic since $\mcl S_{\wh h^\alpha}^\ep$ is defined using dyadic squares. To this end, let $\{\mathsf S_n\}_{n\in\BB Z}$ be the sequence of dyadic squares containing $z$, enumerated so that $|\mathsf S_n| = 2^{-n}$. 
Also let $\mathsf S_n(2)$ be the dyadic square with the same center as $\mathsf S_n$ and five times the side length as $\mathsf S_n $, equivalently $\mathsf S_n(2)$ is the 2-neighborhood of $\mathsf S_n$ with respect to the $L^\infty$ metric on $\BB C$. 
Note that $\mathsf S_m(2) $ is contained in the interior of $ \mathsf S_n(2)$ whenever $m\geq n+1$. 
For $n,m\in\BB Z$ with $m\geq n$, define the possibly non-concentric square annulus
\eqb \label{eqn-annulus-def}
\BB A_{m,n} := \mathsf S_m(2) \setminus \mathsf S_n(2) .
\eqe 

\begin{lem} \label{lem-annulus-rect} 
For each $\zeta \in (0,1)$, it holds with exponentially high probability as $n\rta\infty$, at a rate which is uniform over all choices of $z \in \BB S$ with $\op{dist}(z,\bdy \BB S) \geq 1/4$ and all choices of $\ep \in (0,1)$, that for each $2^{-n}\times 2^{-n-1}$ rectangle $R\subset \BB A_{n,n-2}$ with corners in $2^{-n-1}\BB Z^2$,
\eqb \label{eqn-annulus-rect}
D_{\wh h^\alpha}^\ep\left( \bdy_{\op{L}} R , \bdy_{\op{R}} R ; R \right) \leq   2^{(\alpha-Q+\zeta) \xi n} \ep^{-\xi -\zeta } . 
\eqe
Moreover, the same holds for $2^{-n} \times 2^{-n+1}$ rectangles with corners in $2^{-n}\BB Z^2$, but with $\bdy_{\op{T}} R$ and $ \bdy_{\op{B}} R$ in place of $\bdy_{\op{L}} R$ and $\bdy_{\op{R}} R$. 
\end{lem}
\begin{proof}
Fix a $2^{-n}\times 2^{-n-1}$ rectangle $R\subset \BB A_{n,n-2}$ with corners in $2^{-n-1}\BB Z^2$. Since there are only a constant order number of such rectangles, it suffices to prove that~\eqref{eqn-annulus-rect} holds with exponentially high probability for this fixed choice of $R$ (we can then take a union bound over all of the possibilities for $R$, and use the rotational symmetry of the law of $\wh h$ to get the analogous statement for $2^{n } \times 2^{-n+1}$ rectangles).
We note that for $t \in (0,1)$ and $w \in \BB A_{n,n-2}$, we have 
\eqb \label{eqn-diff-on-annulus}
\wh h_t^\alpha(w) -  \wh h_t(w)  = \alpha \log(|w-z|^{-1}) \in \left[ \alpha \log 2^n  - 2\alpha \log 2 ,  \alpha \log 2^n      \right] .
\eqe 

The proof of the desired bound for $R$ is similar to the proof of Lemma~\ref{lem-rectangle-dist}. Let
\eqbn
s_n := \lfloor \log_2 n \rfloor \quad \text{so that} \quad 2^{-s_n} \leq n \leq 2^{-s_n+1}. 
\eqen
Let $u_R$ be the bottom left corner of $R$, so that $2^n n (R-u_R)$ is the rectangle $\mcl R_n$ of Lemma~\ref{lem-rectangle-perc}. By~\eqref{eqn-dist-scaling} (with $C = 2^{ n + s_n}$ and $f = \wh h_{2^{-n - s_n} }$) and the scale invariance properties of $\wh h$, applied exactly as in the proof of Lemma~\ref{lem-rectangle-dist}, together with~\eqref{eqn-diff-on-annulus}, the conditional law of $D_{\wh h^\alpha}^\ep(\bdy_{\op{L}} R , \bdy_{\op{R}} R ; R)$ given $\wh h_{2^{-n-s_n}}$ is stochastically dominated by the law of 
\eqbn
D_{\wh h}^{T_R\ep} \left( \bdy_{\op{L}} \mcl R_n , \bdy_{\op{R}} \mcl R_n ; \mcl R_n \right) \quad\op{for} \quad
T_R := 2^{(n+s_n) Q} 2^{- \alpha (n+2) } \exp\left( - \max_{z\in R} \wh h_{2^{-n - s_n}}(z) \right) .
\eqen
By~\cite[Lemma 3.6]{dg-lqg-dim} and since $\wh h_{2^{-n}}(u_R)$ is centered Gaussian with variance $\log 2^n$, it holds with exponentially high probability as $n\rta\infty$ that $\max_{z\in R} |\wh h_{2^{-n} n^{-1}}(z)| \leq \zeta \log 2^n$, in which case 
\eqbn
T_R \geq 2^{ - \left( \alpha - Q  +  \zeta - o_n(1) \right) n}  .
\eqen
After possibly shrinking $\zeta$, we can now infer from Lemma~\ref{lem-rectangle-perc} (applied with $T_R\ep$ in place of $\ep$ and $2^{s_n}$ in place of $n$) that with exponentially high probability as $n\rta\infty$,
\eqb \label{eqn-annulus-rect0}
D_{\wh h^\alpha}\left( \bdy_{\op{L}} R , \bdy_{\op{R}} R ; R \right) \leq n^2 \max\left\{ A ,  2^{(\alpha-Q+\zeta) \xi n}\ep^{-\xi -\zeta }\right\} , 
\eqe
where $A$ is the constant from that lemma. Since $\alpha > Q$, if we choose $\zeta$ sufficiently small and $n$ sufficiently large such that $2^{(\alpha-Q + \zeta) \xi n} \geq A$, then the second term inside the max in~\eqref{eqn-annulus-rect0} is larger than the first. Absorbing the factor of $n^2$ into a factor of $2^{\zeta n}$ and again shrinking $\zeta$ now shows that~\eqref{eqn-annulus-rect} holds with exponentially high probability as $n\rta\infty$. 
\end{proof}

In the next lemma, we write $\bdy_{\op{in}} A$ and $\bdy_{\op{out}} A$, respectively, for the inner and outer boundaries of a square annulus $A$. 

\begin{lem} \label{lem-annulus-path}
Fix $\zeta\in (0,1)$. With exponentially high probability as $n\rta\infty$, at a rate which is uniform over all $\ep \in (0,1)$, there exists a path of squares $P_{n}^\perp$ in $\mcl S_{\wh h^\alpha}^\ep\left(\BB A_{n,n-2} \right)$ from $\bdy_{\op{in}} \BB A_{n,n-2}$ to $\bdy_{\op{out}} \BB A_{n,n-2}$ and a path of squares $P_n^\circ$ in $\mcl S_{\wh h^\alpha}^\ep\left( \BB A_{n,n-1} \right)$ which separates $\bdy_{\op{in}} \BB A_{n,n-1}$ from $\bdy_{\op{out}} \BB A_{n,n-1}$, each of which has length at most $2^{(\alpha-Q+\zeta) \xi n}\ep^{-\xi -\zeta } $. 
\end{lem}
\begin{proof}
By Lemma~\ref{lem-annulus-rect}, it holds with exponentially high probability as $n\rta\infty$, as a rate which is uniform over all $\ep \in (0,1)$, that for each $2^{-n}\times 2^{-n-1}$ rectangle $R\subset \BB A_{n,n-2}$ with corners in $2^{-n-1}\BB Z^2$, there is a path of squares $P_R$ in $\mcl S_{\wt h^\alpha}^\ep(R)$ from $\bdy_{\op{L}} R$ to $\bdy_{\op{R}} R$ with length at most $2^{(\alpha-Q+\zeta) \xi n}\ep^{-\xi -\zeta }$, and the analogous statement holds for $2^{-n-1}\times 2^{-n }$ rectangles. Let $\wt P_n$ be the union of the paths $P_R$ over all of the $2^{-n}\times 2^{-n-1}$ or $2^{-n-1} \times 2^{-n}$ rectangles $R$ with corners in $2^{-n-1}\BB Z^2$ which are contained in $\BB A_{n,n-2}$. The number of such rectangles $R$ is at most a universal constant $c>0$ and each square of $\mcl S_{\wh h^\alpha}^\ep(R)$ for each such rectangle $R$ is contained in a square of $\mcl S_{\wh h^\alpha}^\ep(\BB A_{n,n-2})$. Consequently, $\wt P_n$ is contained in the union of at most $c 2^{(\alpha-Q+\zeta) \xi n}\ep^{-\xi -\zeta } $ squares of $\mcl S_{\wh h^\alpha}^\ep(\BB A_{n,n-2})$. Furthermore, since the rectangles $R$ above overlap, it is easily seen that this union contains paths $P_n^\perp$ and $P_n^\circ$ as in the statement of the lemma. This gives the statement of the lemma after slightly shrinking $\zeta$ to get rid of the factor of $c$.
\end{proof}

\begin{proof}[Proof of Proposition~\ref{prop-singularity-dist}]
We will prove the statement of the proposition for $\wh h^\alpha$. The statement for $h^\alpha$ is immediate from this and Lemma~\ref{lem-tr-compare-square}. 
Fix $\wt\zeta \in (0,\zeta)$ to be chosen later, in a manner depending only on $K$, $\zeta$, $\alpha$, and $Q$, and for $\ep\in (0,1)$ let $n_\ep \in \BB N$ be chosen so that \eqbn
2^{-n_\ep-1} \leq \ep^{\wt\zeta/(2\alpha)} \leq 2^{-n_\ep} . 
\eqen

\noindent\textit{Step 1: Finding a path to a nearly macroscopic annulus.}
By Lemma~\ref{lem-annulus-path} and a union bound over all $n \geq n_\ep$, it holds with polynomially high probability as $\ep\rta 0$ that for each $n\geq n_\ep$, there is a path of squares $P_{n}^\perp$ in $\mcl S_{\wh h^\alpha}^\ep\left(\BB A_{n,n-2} \right)$ from $\bdy_{\op{in}} \BB A_{n,n-2}$ to $\bdy_{\op{out}} \BB A_{n,n-2}$ and a path of squares $P_n^\circ$ in $\mcl S_{\wh h^\alpha}^\ep\left( \BB A_{n,n-1} \right)$ which separates $\bdy_{\op{in}} \BB A_{n,n-1}$ from $\bdy_{\op{out}} \BB A_{n,n-1}$, each of which has length at most $2^{(\alpha-Q+ \wt \zeta) \xi  n}\ep^{-\xi -\wt\zeta } $. Henceforth assume that this is the case.

Since $\op{dist}(z,\bdy\BB S)\geq 1/4$, if $\ep$ is at most some constant depending only on $\wt\zeta$ and $\alpha$, then each of the annuli $\BB A_{n,n-2}$ for $n\geq n_\ep$ is contained in $\BB S$. Therefore, each of the squares in each of the paths $P_n^\perp$ and $P_n^\circ$ for $n\geq n_\ep$ is contained in a square of $\mcl S_{\wh h^\alpha}^\ep(\BB S)$. Let $P$ be the set of squares of $\mcl S_{\wh h^\alpha}^\ep(\BB S)$ which contain the squares in 
\eqbn
 \bigcup_{n=n_\ep}^{\lceil \log_2 \ep^{-K} \rceil} \left( P_n^\perp \cup P_n^\circ \right). 
\eqen
Topological considerations show that $P_n^\perp$ intersects $P_{n-1}^\circ$ and $P_n^\circ$ for each $n\in \BB N$, so $P$ is a connected set of squares $S$ with $M_{\wh h^\alpha}(S) \leq \ep$ (see Figure~\ref{fig-annulus-path}). 
If $\ep$ is chosen sufficiently small so that $2^{-n_\ep} \leq 1/100$, then each of the annuli $\BB A_{n , n-2}$ for $n\geq n_\ep$, and hence also each of the squares in $P$, is contained in $\BB S$. 
Therefore, in this case $P$ is a connected set of squares in $\mcl S_{\wh h^\alpha}^\ep(\BB S)$. Moreover, $P$ includes a square which intersects $B_{\ep^K}(z)$ and $P$ disconnects $\bdy_{\op{in}} \BB A_{n_\ep,n_\ep-1}$ from $\bdy_{\op{out}} \BB A_{n_\ep ,n_\ep-1}$.
By summing the above bound for $|P_n^\perp|$ and $|P_n^\circ|$ over all $n \in [n_\ep , \lceil \log_2 \ep^{-K} \rceil]_{\BB Z}$, we find that
\eqb \label{eqn-singularity-path}
\# P \preceq \ep^{  -  (\alpha-Q+ \wt \zeta) \xi K -\xi -  \wt \zeta }  ,
\eqe 
with the implicit constant independent of $\ep$. 
We have therefore bounded the $D_{\wh h}^\ep$-distance from $B_{\ep^K}(z)$ to $\BB A_{n_\ep , n_\ep-1}$. 
\medskip

\noindent\textit{Step 2: Finding a path to a large square.}
We will now find a square in $S \in \mcl S_{\wh h^\alpha}^\ep(\BB S)$ with side length at least $\ep^{1/Q + \wt\zeta}$ which is contained in $\BB A_{n_\ep,n_\ep-1}$ and which is not too far (in the sense of $\mcl S_{\wh h^\alpha}^\ep$-graph distance) from $P$. 

Choose a deterministic dyadic square $S\subset \BB S \cap \BB A_{n_\ep+1,n_\ep}$ with side length $|S|  \in [ \ep^{1/Q + \wt\zeta  }  , 2  \ep^{1/Q + \wt\zeta } ]$.
Recalling that $2^{-n_\ep} \asymp \ep^{\wt\zeta/(2\alpha)}$, we find that $\wh h^\alpha_{|S|/2}(v_S)$ is Gaussian with mean at most $ \alpha \log(2\ep^{-\wt\zeta/(2\alpha)})$ and variance $ \log \ep^{ - 1/Q + \wt\zeta} + O(1)$. We compute
\allb  \label{eqn-one-macroscopic-square}
&\BB P\left[ \text{$S$ properly contains a square of $ \mcl S_{\wh h^\alpha}^\ep(\BB S)$} \right] 
\leq \BB P\left[ e^{\wh h_{|S|/2}^\alpha(v_S)} |S|^Q  >  \ep \right] \notag \\
&\qquad \qquad \qquad \leq \BB P\left[ \wh h_{2^{-n-1}}^\alpha(v_S) -  \alpha \log(2 \ep^{-\wt\zeta/(2\alpha)} )  > \log\left(  |S|^{-Q} \ep^{1+\wt\zeta/2} \right)   - \log 4 \right]   \notag \\
&\qquad \qquad \qquad \leq \BB P\left[ \wh h_{2^{-n-1}}^\alpha(v_S) -  \alpha \log(2 \ep^{-\wt\zeta/(2\alpha)} )  > \log\left( \ep^{-  \wt\zeta/2}  \right)   - \log 4 \right]  , 
\alle 
which decays polynomially in $\ep$ by the Gaussian tail bound. Consequently, with polynomially high probability either the square $S$ or one of its dyadic ancestors belongs to $\mcl S_{\wh h^\alpha}^\ep(\BB S)$. The same argument used to conclude the proof of the bound $D_{\wh h^\alpha}^\ep(S , \bdy \BB S ; \BB S) \leq \ep^{-f(\beta)- \zeta}$ in the proof of Proposition~\ref{prop-macro-square-dist} applies with the field $\wh h^\alpha$ in place of $\wh h$ and $\BB A_{n_\ep , n_\ep-1}$ in place of $\BB S$ to give that with polynomially high probability as $\ep\rta 0$, we have $D_{\wh h^\alpha}^\ep(S , \bdy_{\op{out}} \BB A_{n_\ep , n_\ep-1} ; \BB S) \leq \ep^{-f(1/Q-\wt\zeta)-100 \wt\zeta}$. Since $P$ disconnects $\bdy_{\op{in}} \BB A_{n_\ep  , n_\ep -1}$, and hence also $S$, from $\bdy_{\op{out}} \BB A_{n_\ep  , n_\ep -1}$ we get that with polynomially high probability as $\ep\rta 0$,
\eqb \label{eqn-one-square-dist}
D_{\wh h^\alpha}^\ep\left( S , P  ; \BB S\right) \leq \ep^{-f(1/Q-\wt\zeta)-100\wt\zeta} \quad \op{and} \quad |S| \geq \ep^{ 1/Q + \wt\zeta} ,
\eqe
where $f(\cdot)$ is as in~\eqref{eqn-macro-square-function}. 
Combining this with~\eqref{eqn-singularity-path} and choosing $\wt\zeta$ sufficiently small (in a manner depending only on $K$, $\zeta$, $\alpha$, and $Q$) and using the triangle inequality shows that $D_{\wh h^\alpha}^\ep(S,B_{\ep^K}(z) ; \BB S) \leq \ep^{-g(\alpha,K)-\zeta}$ with polynomially high probability as $\ep\rta0$, as required. 
\end{proof}

Proposition~\ref{prop-singularity-dist} is not quite sufficient for our purposes since we will eventually want an upper bound for the distance from a large square to a small \emph{square}, rather than a small ball. Such a bound follows from Proposition~\ref{prop-singularity-dist} together with the following lemma.

\begin{lem} \label{lem-small-squares}
Fix $\alpha \in (Q,2)$, $K > 1$, and $\zeta\in (0,1)$ and define $h^\alpha$ and $\wh h^\alpha$ as in~\eqref{eqn-singularity-fields}. With polynomially high probability as $\ep\rta 0$, at a rate which is uniform in $z$, each square in $\mcl S_{ h^\alpha}^\ep(\BB S)$ which intersects $B_{\ep^K}(z)$ has side length at most $\ep^{K }$. The same is true with $\wh h^\alpha$ in place of $h^\alpha$.
\end{lem}
\begin{proof}
By Lemma~\ref{lem-tr-compare-square}, it suffices to prove the statement of the lemma with $\wh h^\alpha$ in place of $h^\alpha$.

Let $n_\ep \in\BB N$ be chosen so that $2^{-n_\ep -1 } \leq \ep^K \leq 2^{-n_\ep }$. Let $\mcl D_{\ep,0}$ be the set of 9 dyadic squares with side length $2^{-n_\ep }$ which either contain $z$ or which share a side or a corner with a square which contains $z$. Note that $B_{\ep^K}(z) \subset \bigcup_{S\in \mcl D_{\ep,0} } S$. 
For $j \in \BB N$, inductively let $\mcl D_{\ep,j}$ be the set of 9 dyadic squares of side length $2^{-n_\ep+j}$ which are the dyadic parents of the 9 squares in $\mcl D_{\ep,j-1}$. 
We will show that with polynomially high probability as $\ep\rta 0$, 
\eqb \label{eqn-no-square-contained}
e^{\wh h^\alpha_{|S|/2}(v_S)} |S|^Q >  \ep , \quad \forall S\in \mcl D_{\ep,j} ,\quad \forall j \in [0,n_\ep]_{\BB Z} .
\eqe
This shows that none of the squares in $\mcl D_{\ep,j}$ for $j\in [0,n_\ep]_{\BB Z}$ is contained in $\mcl S_{\wh h^\alpha}^\ep(\BB S)$, and hence that the squares of $\mcl S_{\wh h^\alpha}^\ep(\BB S)$ which intersect $B_{\ep^K}$ must have side length smaller than $2^{-n_\ep}$. 
  
For each $j\in\BB N$ and each $S\in \mcl D_{\ep,j}$, the center $v_S$ of $S$ lies at Euclidean distance at most $2^{-n_\ep + j + 2}$ from $z$. Therefore, $\wh h^\alpha_{|S|/2}(v_S)$ is Gaussian with mean at least $\alpha \log 2^{n_\ep - j - 2}$ and variance $\log 2^{n_\ep - j + 1}$. 
By the Gaussian tail bound,
\allb \label{eqn-square-in-set-alpha}
  \BB P\left[ e^{\wh h_{|S|/2}^\alpha(v_S)} |S|^Q  \leq \ep \right] 
&\leq \BB P\left[ \wh h_{|S|/2}^\alpha(v_S) -  \alpha \log 2^{n_\ep - j - 2}   < \log\left(2^{-(\alpha -Q) (n_\ep - j) } \ep \right) + \log 8 \right] \notag\\
&\leq \exp\left( - \frac{\left( \log\left(2^{-(\alpha -Q) (n_\ep - j) } \ep \right) \right)^2}{2 \log 2^{ n_\ep - j + 1}} \right)
\alle
We now conclude by means of a union bound over all $S\in \mcl D_{\ep,j}$ and all $j \in [0,n_\ep]_{\BB Z}$. 
\end{proof}

\subsection{Lower bound for the number of small squares}
\label{sec-small-square-count}

Throughout this subsection we let $h$ be a whole-plane GFF normalized so that $h_1(0) = 0$ (we no longer need to consider $\wh h$). 

\begin{prop} \label{prop-small-square-count}
For each $K\in\BB N$, $\zeta\in (0,1)$, and $\alpha \in (Q,2)$, it holds with polynomially high probability as $\ep\rta 0$ that the following is true. There is a collection $\mcl C^\ep \subset \mcl S^\ep_h(\BB S)$ of at least $\ep^{-(2-\alpha)^2 K/2  + 3\zeta}$ distinct squares in $\mcl S_h^\ep(\BB S)$ each of which has side length at most $\ep^K$, lies at Euclidean distance at least $1/8$ from $\bdy\BB S$, and lies at $\mcl S^\ep_h(\BB S)$-graph distance at most $\ep^{ - g(\alpha , K) - \zeta}$ from some square $S \in\mcl S_h^\ep(\BB S)$ with $|S| \geq \ep^{1/Q+\zeta}$, where $g(\alpha,K)$ is as in~\eqref{eqn-singularity-dist-function}.
\end{prop}

The only important features of the exponents appearing in Proposition~\ref{prop-small-square-count} for our purposes are that $(2-\alpha)^2 K/2$ tends to $\infty$ as $K\rta\infty$ provided $\alpha$ is bounded away from $2$ and that $g(Q+1/K , K)$ is bounded above independently of $K$, as noted just after Proposition~\ref{prop-singularity-dist}. 

To prove Proposition~\ref{prop-small-square-count}, we will use Proposition~\ref{prop-singularity-dist} and Lemma~\ref{lem-small-squares} to show that a certain regularity event, defined just below, holds with high probability at a typical point sampled from the $\alpha$-LQG measure associated with $h$ (Lemma~\ref{lem-uniform-pt}). We will then combine this with the fact that the mass of the $\alpha$-LQG measure is not too ``spread out" to deduce Proposition~\ref{prop-small-square-count}. 

Fix $K\in\BB N$, $\zeta \in (0,1)$, and $\alpha\in (Q,2)$ and for $z\in\BB S $, let $E^\ep(z) = E^\ep(z;K,\zeta,\alpha)$ be the event that the following is true. 
\begin{enumerate}
\item There is a square $S \in\mcl S_h^\ep(\BB S)$ with $|S| \geq \ep^{1/Q + \zeta}$ and $D_{h}^\ep\left( S , \bdy B_{\ep^K}(z) ; \BB S    \right)  \leq \ep^{ - g(\alpha,K) -\zeta}$. \label{item-pt-dist}
\item Each square in $\mcl S_{h}^\ep(\BB S)$ which intersects $B_{\ep^K}(z)$ has side length at most $\ep^{K }$. \label{item-pt-diam}
\end{enumerate}
Also let $\BB S' \subset \BB S$ be the closed square consisting of points lying at $L^\infty$ distance at least $1/4$ from $\bdy\BB S$ and let $h^{\BB S'}$ be the zero-boundary part of $h|_{\BB S'}$, so that $h^{\BB S'}$ is a zero-boundary GFF on $\BB S'$ and $(h-h^{\BB S'})|_{\BB S'}$ is a random harmonic function independent from $h^{\BB S'}$.

\begin{lem} \label{lem-uniform-pt}
Conditionally on $h$, let $\BB z$ be sampled uniformly from the $\alpha$-LQG measure $\mu_{h^{\BB S'}}^\alpha$, normalized to be a probability measure.
Then $E^\ep(\BB z)$ occurs with polynomially high probability as $\ep\rta 0$.
\end{lem}
\begin{proof}
Let $\wt{\BB P}$ be the law of $(h,\BB z)$ weighted by the total mass $\mu_{h^{\BB S'}}^\alpha(\BB S')$, so that under $\wt{\BB P}$, $h$ is sampled from its marginal law weighted by $\mu_{h^{\BB S'}}^\alpha(\BB S')$ and conditionally on $h$, $\BB z$ is sampled from $\mu_{h^{\BB S'}}^\alpha$, normalized to be a probability measure.
By a well-known property of the $\gamma$-LQG measure (see, e.g.,~\cite[Lemma A.10]{wedges}), a sample $(h,\BB z)$ from the law $\wt{\BB P}$ can be equivalently produced by first sampling $\wt h$ from the unweighted marginal law of $h$, then independently sampling $\BB z$ uniformly from Lebesgue measure on $\BB S'$ and setting $h = \wt h - \alpha \log|\cdot-\BB z| + \frk g_{\BB z}$, where $\frk g_{\BB z} : \BB C\rta\BB R$ is a continuous function whose absolute value is bounded on $\BB S$ by constants depending only on $\alpha$.
By this, Proposition~\ref{prop-singularity-dist}, and Lemma~\ref{lem-small-squares}, we find that $E^\ep(\BB z)$ occurs with polynomially high $\wt{\BB P}$-probability as $\ep\rta 0$. 

We now transfer from $\wt{\BB P}$ to $\BB P$. 
Since $\mu_{h^{\BB S'}}^\alpha(\BB S')$ has finite moments of all negative orders~\cite[Theorem 2.11]{rhodes-vargas-review}, we can apply the Cauchy-Schwarz inequality to get that for some constants $a,p>0$ depending only on $\alpha$ and $Q$,
\alb
\BB P\left[ E^\ep(\BB z)^c \right]
&= a\wt{\BB E}\left[ \frac{ E^\ep(\BB z)^c } { \mu_{h^{\BB S'}}^\alpha(\BB S') } \right] 
\leq a\wt{\BB P}\left[  E^\ep(\BB z)^c   \right]^{1/2} \wt{\BB E}\left[ \left(  \mu_{h^{\BB S'}}^\alpha(\BB S')  \right)^{-2} \right]^{1/2}\notag \\
&=  a\wt{\BB P}\left[  E^\ep(\BB z)^c   \right]^{1/2}  \BB E\left[ \left(  \mu_{h^{\BB S'}}^\alpha(\BB S')  \right)^{-1} \right]^{1/2} = O_\ep(\ep^p ) .
\ale  
\end{proof}

\begin{proof}[Proof of Proposition~\ref{prop-small-square-count}]
Define $E^\ep(z)$, $\BB S'$, and $h^{\BB S'}$ as above.  
Conditionally on $h$, let $\BB z$ be sampled uniformly from the $\alpha$-LQG measure $\mu_{h^{\BB S'}}^\alpha$, normalized to be a probability measure.
By Lemma~\ref{lem-uniform-pt}, we can find $s = s(\alpha, K , \zeta,Q) \in (0,\zeta/2]$ such that $\BB P[E^\ep(\BB z) ] \geq 1 - O_\ep(\ep^{ s})$. By Markov's inequality, 
\eqbn  
\BB P\left[  \BB P\left[ E^\ep(\BB z) \,|\, h \right]  \geq \ep^{\zeta/2} \right] \geq 1 - O_\ep(\ep^s) ,
\eqen
Since $\BB P\left[ E^\ep(\BB z) \,|\, h \right] = \mu_{h^{\BB S'}}^\alpha\left( z\in \BB S' : \text{$E^\ep(z)$ occurs} \right) /  \mu_{h^{\BB S'}}^\alpha(\BB S')$ and $ \mu_{h^{\BB S'}}^\alpha(\BB S') \geq \ep^{\zeta/2}$ with polynomially high probability as $\ep\rta 0$~\cite[Lemma 4.5]{shef-kpz}, we infer that with polynomially high probability as $\ep\rta 0$, 
\eqb \label{eqn-singularity-cond}
\mu_{h^{\BB S'}}^\alpha\left( z\in \BB S' : \text{$E^\ep(z)$ occurs} \right)  \geq \ep^\zeta .
\eqe
 
By standard estimates for the $\alpha$-LQG measure (see, e.g.,~\cite[Lemma 3.8]{dg-lqg-dim}), it holds with polynomially high probability as $\ep\rta 0$ that
\eqb \label{eqn-use-min-ball}
\max_{z\in\BB S'} \mu_{h^{\BB S'}}^\alpha\left( B_{4\ep^K}(z) \right) \leq \ep^{(2-\alpha)^2 K/2 - \zeta } . 
\eqe 
We claim that if~\eqref{eqn-singularity-cond} and~\eqref{eqn-use-min-ball} both occur (which happens with polynomially high probability), then we can find a finite collection $Z^\ep \subset \BB S'$ of at least $\ep^{-(2-\alpha)^2 K/2  + 3\zeta}$ points such that $E^\ep(z)$ occurs for each $z\in Z^\ep$ and the balls $B_{3\ep^K}(z)$ for $z\in Z^\ep$ are disjoint. To see this, we observe that the balls $B_{4\ep^K}$ for $w \in \left(  \ep^{-K}\BB Z^2 \right) \cap \BB S'$ cover $\BB S'$ and each of these balls has $\mu_{h^{\BB S'}}^\alpha$-mass at most $\ep^{(2-\alpha)^2 K/2 - \zeta }$ by~\eqref{eqn-use-min-ball}. Since $\mu_{h^{\BB S'}}^\alpha\left( z\in \BB S' : \text{$E^\ep(z)$ occurs} \right) \geq \ep^\zeta$ by~\eqref{eqn-singularity-cond}, it follows that at least $\ep^{-(2-\alpha)^2 K/2  + 2\zeta}$ of these balls must contain a point $z$ for which $E^\ep(z)$ occurs. 
Since each point in $\BB S'$ is contained in at most a constant order number of the balls $B_{7\ep^K}(w)$ for $w \in (\ep^{-K}\BB Z^2) \cap \BB S'$, our claim follows.

By the definition of $E^\ep(z)$, if $Z^\ep$ is as in the above claim then we can find for each $z\in Z^\ep$ a dyadic square $S^\ep(z) \in \mcl S_h^\ep(\BB S)$ which intersects $B_{\ep^K}(z)$, lies at $\mcl S_h^\ep(\BB S)$-graph distance at most $\ep^{ -g(\alpha,K)-\zeta}$ from some square $S \in\mcl S_h^\ep(\BB S)$ with $|S| \geq \ep^{1/Q+\zeta}$, and which has side length at most $\ep^K$. By our choice of $Z^\ep$, the squares $Z^\ep(z)$ corresponding to different points $z\in Z^\ep$ lie at positive distance from one another, so are distinct. Thus the proposition statement holds with $\mcl C^\ep = \{S^\ep(z) : z\in Z^\ep\}$. 
\end{proof}

\subsection{Proof of Theorem~\ref{thm-ball-infty}}
\label{sec-ball-infty-proof}

Throughout this section, we let $h$ be a whole-plane GFF normalized so that $h_1(0) = 0$ and we fix $Q\in (0,2)$. 
The following lemma collects the relevant information from the preceding subsections.

\begin{lem} \label{lem-ball-poly}
For each $K\in\BB N$, $\zeta\in (0,1)$, and $\alpha \in (Q,2)$, it holds with polynomially high probability as $\ep\rta 0$ that the following is true.
There are at least $\ep^{-(2-\alpha)^2 K/2  + 3\zeta}$ distinct squares of $\mcl S_h^\ep(\BB S)$ which each lie 
at $\mcl S_h^\ep(\BB S)$-graph distance at most $\ep^{ - g(\alpha,K)-\zeta}$ from a square of $\mcl S_h^\ep(\BB S)$ containing the center point $v_{\BB S}$, where $g(\alpha,K)$ is the exponent from~\eqref{eqn-singularity-dist-function}.
\end{lem}
\begin{proof}  
The Gaussian tail bound shows that with polynomially high probability as $\ep\rta 0$, $v_{\BB S}$ is contained in a square of $\mcl S_h^\ep(\BB S)$ with side length at least $\ep^{1/Q+\zeta}$. 
By Proposition~\ref{prop-macro-square-dist}, it holds with polynomially high probability as $\ep\rta 0$ that each square $S\in \mcl S_h^\ep(\BB S)$ with $|S| \geq \ep^{1/Q + \zeta}$ lies at $\mcl S_h^\ep(\BB S)$-graph distance at most $\ep^{-f(1/Q + \zeta) - \zeta}$ from a square of $\mcl S_h^\ep(\BB S)$ containing $v_{\BB S}$.
By combining this with Proposition~\ref{prop-small-square-count} and the triangle inequality, we get the statement of the lemma (here we note that $g(\alpha,K) \geq f(1/Q)$).
\end{proof}

To extract Theorem~\ref{thm-ball-infty} from Lemma~\ref{lem-ball-poly}, we will use a scaling argument for the GFF $h$. 
The law of the GFF is only scale invariant modulo additive constant, i.e., for $R > 0$ we have $h(R\cdot) - h_R(0) \eqD h$. 
The following elementary lemma will be used to control how much effect subtracting $h_R(0)$ has on the objects we are interested in. 

\begin{lem} \label{lem-circle-avg-rn}
Let $h$ be a whole-plane GFF normalized so that $h_1(0) = 0$ and fix $\rho \in (0,1)$. For $\delta \in (0,\rho)$ and $\frk a \in \BB R$, the conditional law of $h|_{\BB C\setminus B_\rho(0)}$ given $\{h_\delta(0) = \frk a\}$ is absolutely continuous with respect to the marginal law of $h|_{\BB C\setminus B_\rho(0)}$. 
The Radon-Nikodym derivative of the former law with respect to the latter law is given by
\eqb \label{eqn-circle-avg-rn} 
\sqrt{\frac{\log(1/\delta)}{\log(\rho/\delta)} } \exp\left(   \frac{\frk a^2 \log(\rho ) + (     2\frk a h_\rho(0)  - h_\rho(0)^2 )  \log(1/\delta)      }{2\log(1/\delta) \log(\rho/\delta)}   \right)  .
\eqe
\end{lem}
\begin{proof} 
By standard results for the GFF (see, e.g.,~\cite[Section 4.1.5]{wedges}, $B_t :=  h_{e^{-t}}(0)$ is a standard two-sided Brownian motion and $h - h_{|\cdot|}(0)$ is independent from $B$. Hence we get the absolute continuity in the statement of the lemma, and the desired Radon-Nikodym derivative is the same as the Radon-Nikodym derivative of the conditional law of $B|_{[0,\log (1/\rho) ]}$ given $\{B_{\log(1/\delta) }  = \frk a\}$ with respect to the marginal law of $B|_{[0,\log(1/\rho) ]}$. 
This Radon-Nikodym derivative can be computed using, e.g., Bayes' rule and is given by~\eqref{eqn-circle-avg-rn}.
\end{proof}

\begin{proof}[Proof of Theorem~\ref{thm-ball-infty}]
Fix $Q\in (0,2)$ and $p > 1$. We will show that a.s.\ $\#\mcl B_r^{\mcl S_h^1}(0) \geq r^p$ for large enough $r$. 
\medskip

\noindent\textit{Step 1: Scaling argument.}
We will first use the scaling properties of the GFF to reduce our problem to proving an estimate for balls in $\mcl S_h^\ep$ for a small $\ep  =\ep_r > 0$ instead of for balls in $\mcl S_h^1$. 
Fix a small exponent $\theta   \in (0,1)$, to be chosen later in a manner depending only on $p$ and $Q$. 
Given $r\in \BB N$, let $\delta_r$ be the smallest dyadic integer such that $\delta_r \leq r^{-\theta}$ and let $h^r := h(\delta_r \cdot) - h_{\delta_r}(0)$.
Then $h^r \eqD h$ and for any dyadic square $S\subset \BB C$, 
\eqbn
M_{h^r}(S) = e^{h^r_{|S|/2}(v_S)} |S|^Q = e^{-h_{\delta_r}(0)} e^{h_{|\delta_r S|/2}(v_{\delta_r S})} |S|^Q =  e^{-h_{\delta_r}(0)} \delta_r^{- Q} M_h(\delta_r S) .
\eqen
Consequently, 
\eqb \label{eqn-ball-scale-field}
\mcl S_{h^r}^1 = \delta_r^{-1} \mcl S_h^{\ep_r} \quad \text{for} \quad \ep_r :=  e^{ h_{\delta_r}(0)} \delta_r^{  Q} .
\eqe
We will show that if $\theta$ is chosen appropriately (depending only on $p$ and $Q$), then   
\eqb \label{eqn-ep-ball-bound}
\#\mcl B_r^{\mcl S_h^{\ep_r}}(0) \geq r^p  ,\quad \text{with polynomially high probability as $r\rta \infty$}. 
\eqe 
Combined with~\eqref{eqn-ball-scale-field}, we then get that~\eqref{eqn-ep-ball-bound} holds with 1 in place of $\ep_r$. A union bound over dyadic values of $r$ then concludes the proof of the theorem. Thus we only need to prove~\eqref{eqn-ep-ball-bound}. 

We note that $\ep_r$ from~\eqref{eqn-ball-scale-field} is random. However, $h_{\delta_r}(0)$ is Gaussian with variance $\log(1/\delta_r)$ and $\delta_r \asymp r^{-\theta}$, so for any $\zeta \in (0,1)$,
\eqb \label{eqn-ep-r-compare}
r^{-\theta Q - \zeta} \leq \ep_r \leq r^{-\theta Q  + \zeta} \quad \text{with polynomially high probability as $r \rta\infty$.}
\eqe
\medskip

\noindent\textit{Step 2: Application of Lemma~\ref{lem-ball-poly}. }
It will be convenient to first work with $\wt{\BB S} := [1/2,1]^2$ instead of with $\BB S$ since $\wt{\BB S}$ lies at positive distance from 0, which will allow us to apply Lemma~\ref{lem-circle-avg-rn} when we condition on the value of $\ep_r$ from~\eqref{eqn-ep-ball-bound}. 
All of the arguments of this section work equally well with $\wt{\BB S}$ in place of $\BB S$, and in particular Lemma~\ref{lem-ball-poly} is true with this replacement. 
 
By Lemma~\ref{lem-ball-poly} with $\wt{\BB S}$ in place of $\BB S$, for any $K > 1$, $\zeta\in (0,1)$, and $\alpha \in (Q,2)$, it holds with polynomially high probability as $\ep\rta 0$ that
\eqb \label{eqn-ep-ball} 
\#\mcl B_s^{\mcl S^\ep_h(\wt{\BB S} )}(v_{\wt{\BB S}} ) \geq  \ep^{-(2-\alpha)^2 K/2  + 3\zeta} , \quad \forall s \geq \ep^{ - g(\alpha,K) - \zeta} .
\eqe 
We now want to transfer from a deterministic choice of $\ep$ to the random choice $\ep_r$. 
For this, we observe that $\{\mcl S^\ep_h(\wt{\BB S})\}_{\ep > 0}$ is a.s.\ determined by $h|_{\wt{\BB S}}$, which in turn is a.s.\ determined by $h|_{\BB C\setminus B_{1/2}(0)}$. 
We may therefore apply Lemma~\ref{lem-circle-avg-rn} with $\rho = 1/2$ and $\delta =\delta_r$ to find that for $\frk a \in \BB R$, the Radon-Nikodym derivative of the conditional law of $\{\mcl S^\ep_h(\wt{\BB S})\}_{\ep > 0}$ given $\{h_{\delta_r}(0) = \frk a\}$, or equivalently $\{\ep_r = e^{\frk a} \delta_r^{ Q}\}$, with respect to the marginal law of $\{\mcl S^\ep_h(\wt{\BB S})\}_{\ep > 0}$ is given by~\eqref{eqn-circle-avg-rn}. 
This Radon-Nikodym derivative is bounded above and below by constants depending only on $\theta$ provided $|h_{1/2}(0)| \leq \sqrt{\log r}$ and $|h_{\delta_r}(0)|  \leq   \log r$, which happens with polynomially high probability as $r \rta \infty$. 
Combining this with~\eqref{eqn-ep-r-compare}, we get that~\eqref{eqn-ep-ball} holds with $\ep_r$ in place of $\ep$ with polynomially high probability as $r\rta\infty$.
\medskip

\noindent\textit{Step 3: Transferring from $\mcl S_h^{\ep_r}(\wt{\BB S})$ to $\mcl S_h^{\ep_r}$. } 
We will now argue that an analogue of~\eqref{eqn-ep-ball} holds with high probability with $\ep_r$ in place of $\ep$, $\mcl S_h^{\ep_r}$ in place of $\mcl S_h^\ep(\wt{\BB S})$, and 0 in place of $v_{\wt{\BB S}}$. Here we will deal with the fact that $\ep_r$ is random using~\eqref{eqn-ep-r-compare} and the monotonicity of various quantities in $\ep$. We could not do this above since the quantity $\#\mcl B_r^{\mcl S_h^\ep(U)}(z)$ for fixed $U\subset \BB C$, $z\in U$, and $r\in\BB N$ does not depend monotonically on $\ep$. 

Fix a small exponent $\wt\zeta \in (0,\theta Q)$.
By Lemma~\ref{lem-unit-square-good}, it holds with polynomially high probability as $r\rta\infty$ that each square of $\mcl S_h^{r^{-\theta Q + \wt\zeta}}(\wt{\BB S})$ is also a square of $\mcl S_h^{r^{-\theta Q +\wt\zeta}}$ (equivalently, $\wt{\BB S}$ is not properly contained in a square of $\mcl S_h^{r^{-\theta Q + \wt\zeta}}$). Together with~\eqref{eqn-ep-r-compare} and the monotonicity of $\mcl S_h^\ep$ in $\ep$, this shows that with polynomially high probability as $r\rta\infty$, each square of $\mcl S_h^{\ep_r}(\wt{\BB S})$ is also a square of $\mcl S_h^{\ep_r}$. This implies that~\eqref{eqn-ep-ball} holds with polynomially high probability as $r\rta\infty$ with $\ep_r$ in place of $\ep$ and $\mcl S_h^{\ep_r}$ in place of $\mcl S_h^\ep(\wt{\BB S})$. 

We now transfer from balls centered at $v_{\BB S}$ to balls centered at 0. 
By Proposition~\ref{prop-macro-square-dist} (applied as in the proof of Lemma~\ref{lem-ball-poly}), if we choose $\wt\zeta$ sufficiently small (depending on $Q$, $\theta$, and $\zeta$) the with polynomially high probability as $r\rta\infty$, we have $D_h^{r^{-\theta Q - \wt\zeta}}( 0,  v_{\wt{\BB S} }) \leq r^{\theta Q f(1/Q) + \zeta/2}$, where $f(\cdot)$ is as in~\eqref{eqn-macro-square-function}. 
Again using~\eqref{eqn-ep-r-compare} and monotonicity in $\ep$, we get that with polynomially high probability as $r\rta\infty$, $D_h^{\ep_r}(0,v_{\wt{\BB S}}) \leq \ep_r^{f(1/Q) + \zeta}$.

Combining the conclusions of the two preceding paragraphs with~\eqref{eqn-ep-ball} for $\ep_r$ in place of $\ep$ and using that $g(\alpha,K) \geq f(1/Q)$ by definition shows that with polynomially high probability as $r\rta\infty$,
\eqb \label{eqn-ep-ball-plane} 
\#\mcl B_s^{\mcl S^{\ep_r}_h }( 0 ) \geq  \ep_r^{-(2-\alpha)^2 K/2  + 3\zeta} , \quad \forall s \geq \ep^{ - g(\alpha,K)      - \zeta}  .
\eqe 
\medskip

\noindent\textit{Step 4: Choosing the parameter values. }
For a given $p > 1$, we will now choose the parameters $\theta$, $K$, $\alpha$, and $\zeta$ above in a manner depending on $p$ and $Q$ in order to extract~\eqref{eqn-ep-ball-bound} from~\eqref{eqn-ep-ball-plane}.  
We first set $\alpha = Q  + 1/K$, so that by~\eqref{eqn-singularity-dist-function}, $q := g(Q+1/K ,K) $ depends only on $Q$. 
Then, we choose $\theta$ sufficiently small that $  2 q \theta Q < 1$, so that by~\eqref{eqn-ep-r-compare}, it holds with polynomially high probability as $r\rta\infty$ that $r \geq \ep_r^{-2q}$. 
Finally, we choose $K$ sufficiently large such that $\frac12 (2-\alpha)^2 K  \times \theta Q = \frac12 (2-Q - 1/K)^2 K  \times \theta Q \geq 2p $ and $\zeta$ sufficiently small such that $2p-3\zeta \geq  p$ and $ q+\zeta \leq 2q$. 
Making this choice of parameters and applying~\eqref{eqn-ep-ball-plane} yields that~\eqref{eqn-ep-ball-bound} holds with polynomially high probability as $r\rta\infty$. 
\end{proof}

\section{Conjectures and open questions}
\label{sec-open-problems}
 
One of the most interesting problems to solve concerning the model of this paper would be to prove that $\mcl S_h^\ep$ converges to a metric space in the scaling limit (Conjecture~\ref{conj-metric-lim}). This would allow us to define LQG with matter central charge $\ccM$ as a metric space. 
It is not clear to us, even at a heuristic level, whether it is possible to define LQG for $\ccM \in (1,25)$ as a metric \emph{measure} space. 

\begin{ques}[Measure scaling limit] \label{ques-measure-lim}
For $\ccM \in (1,25)$, do the graphs $\mcl S_h^\ep$, equipped with their graph distance and the counting measure on squares, converge in law in the scaling limit as $\ep \rta 0$ to a non-trivial random metric measure space $(X , \mu_h , \frk d_h)$  with respect to the local Gromov-Hausdorff-Prokhorov topology~\cite{adh-ghp}? 
\end{ques}

Question~\ref{ques-measure-lim} could possibly have an affirmative answer even though the total number of squares is infinite, since the number of squares contained in a graph metric ball is typically finite. On the other hand, we know from Theorem~\ref{thm-ball-infty} that the number of such squares grows superpolynomially in the radius of the ball, so to get a non-trivial scaling limit one would need a superpolynomial scaling factor for the measure as compared to the metric (distance function).  

Another potential way to construct a ``measure" associated to LQG with matter central charge $\ccM \in (1,25)$ is to directly extend the definition of Gaussian multiplicative chaos in the case when $\ccM < 1$ (see also Section~\ref{sec-related}).
\begin{ques}[Complex Gaussian Multiplicative Chaos] \label{ques-complex-measure}
If $\ccM \in (1,25)$, so that $\gamma$ is complex with $|\gamma | =2$, can one make sense of $\exp\left( \gamma h(z) - \frac{\gamma}{2} \op{Var} h(z) \right) d^2 z$ (where $d^2 z$ is Lebesgue measure on $\BB C$) as a complex measure, or at least as a complex distribution? 
\end{ques}

For $\gamma \in (0,2)$ and corresponding $\kappa \in \{\gamma^2,16/\gamma^2\}$, there are a number of important theorems describing various properties of SLE$_\kappa$ curves on an independent $\gamma$-LQG surface: see, e.g.,~\cite{shef-zipper,wedges,sphere-constructions}. The SLE parameter $\kappa$ is related to the matter central charge by
\eqb
\ccM   =\frac{(6-\kappa)(3\kappa-8)}{2\kappa} .
\eqe
It would be of substantial interest to generalize the theory of SLE and its relationship to LQG to the case when $\ccM \in (1,25)$. 

\begin{ques}[SLE for $c>1$] \label{ques-sle}
Is there a natural extension of SLE to the $\ccM \in (1,25)$ regime? If so, are there $\ccM \in (1,25)$-analogs of the relationships between SLE and LQG from~\cite{shef-zipper,wedges,sphere-constructions}? 
\end{ques}

A potential discrete analog of SLE with $\ccM \in (1,25)$ is the so-called \emph{$\lambda$-self avoiding walk} with $\lambda = -\ccM/2$, which was introduced by Kennedy and Lawler in~\cite{kl-lambda-saw}. This model is expected to converge to SLE$_\kappa$ for $\kappa \in (0,4]$ in the case when $\lambda \geq -1/2$, equivalently, $\ccM \leq 1$, but also makes sense for $\ccM  > 1$.    

Let us now discuss some further questions which we expect to be easier.  

Recall from Section~\ref{sec-laplacian} that random planar maps with a fixed number of vertices, weighted by the $(-\ccM/2)$-power of the Laplacian determinant, are expected to behave like continuum random trees (CRTs) when $\ccM>1$. As explained in Section~\ref{sec-laplacian}, we do not believe that these random planar maps describe the behavior of central charge LQG with matter central charge $\ccM >1$. Rather, based on forthcoming work by Ang, Park, Pfeffer, and Sheffield, we believe their behavior corresponds to that of $\mcl S_h^\ep$ conditioned to be large but finite. Therefore, we expect an affirmative answer to the following question:

\begin{ques}[Relationship to the CRT] \label{ques-crt}
Let $h$ be a GFF on a bounded domain $U\subset\BB C$. Fix $\ep>0$. If we condition on $\{\#\mcl S_h^\ep = n\}$ and send $n \rta \infty$, does $\mcl S_h^\ep$ converge in law (say, with respect to the Gromov-Hausdorff distance) to some version of the CRT?
\end{ques}

Our results suggest that when we analytically continue a formula for a dimension associated with LQG from the case when $\ccM \leq 1$ to the case when $\ccM \in (1,25)$ and get a complex answer, the corresponding dimension should be infinite. 

\begin{ques}[Interpretation of complex dimensions] \label{ques-complex-dim}
Does the particular value of the complex number appearing in analytic continuations of dimension formulas have an interpretation?
For example, in the setting of Theorem~\ref{thm-kpz}, when the Euclidean dimension satisfies $x   > Q^2/2$ does the particular complex number $Q - \sqrt{Q^2-2x}$ tell us anything about the behavior of the set of squares which intersect the fractal $X$? 
\end{ques}
 
Another class of open problems concern the simple random walk on $\mcl S_h^\ep$. 

\begin{conj}[Transience of random walk] \label{conj-transient}
For $\ccM \in (1,25)$, the simple random walk on $\mcl S_h^1$ is a.s.\ transient. 
\end{conj}

Intuitively, the reason why we expect Conjecture~\ref{conj-transient} to be true is that the set of singularities of $\mcl S_h^1$ is sufficiently ``tree-like" that the random walk will typically get stuck in a small neighborhood of one of these singularities. 
Note that, in contrast to Conjecture~\ref{conj-transient}, the simple random walk is known to be recurrent on many different random planar maps in the matter central charge $\ccM$ universality class for $\ccM \leq 1$~\cite{gn-recurrence}. 

When the simple random walk is approaching a singularity, it will be traversing arbitrarily small squares. It may therefore be possible to extend the definition of the walk so that it still converges to Brownian motion when parametrized by ``Euclidean" speed rather than by its intrinsic speed.

\begin{ques}[Convergence of random walk to Brownian motion] \label{ques-clt}
For $\ccM \in (1,25)$, is it possible to extend the definition of the re-parametrized simple random walk on $\mcl S_h^1$ in such a way that it spends $|S|^2$ units of time in each square $S\in \mcl S_h^1$ and is defined for all time? If so, does the extended, re-parametrized walk converge in law to Brownian motion?
\end{ques}

We emphasize that the first part of Question~\ref{ques-clt} is non-trivial since we expect (due to Conjecture~\ref{conj-transient}) that the walk parametrized by $|S|^2$ will traverse infinitely many squares in a finite amount of time, so one has to find a way to ``reflect" the walk off of the singularities.  
For $\ccM <1$, we expect that one can show convergence of the random walk on $\mcl S_h^1$ to Brownian motion using the ideas of~\cite{gms-random-walk,gms-tutte} (see Remark~\ref{remark-c<1-ball-growth}).

Recall from the discussion just after Proposition~\ref{prop-ptwise-distance} that $D_{ h}^\ep(z,w)$ for fixed $z,w\in\BB C$ is typically bounded above and below by powers of $\ep$, and that for $\ccM <1$ we have $D_h^\ep(z,w) \approx \ep^{-\gamma/d_{\ccM} + o_\ep(1)}$, where $\gamma$ is as in~\eqref{eqn-Q-c-gamma} and $d_{\ccM}$ is the Hausdorff dimension of the $\gamma$-LQG metric space.
If we plug in Watabiki's prediction~\eqref{eqn-watabiki} for $d_{\ccM}$, and then analytically continue to $\ccM  \in (1,25)$, we find that $\lim_{\ccM\rta 25^-} \gamma/d_{\ccM} = 1$. 
If we instead plug in the alternative guess~\eqref{eqn-dg-guess}, we get $\lim_{\ccM\rta 25^-} \gamma/d_{\ccM} =\sqrt 6$.  

\begin{ques}[Behavior of point-to-point distance as $\ccM\rta 25$]  \label{ques-Q-to-0}
For fixed $z,w\in\BB C$, is the limit
\eqb
\lim_{\ccM \rta 25^-} \lim_{\ep\rta 0} \frac{D_h^\ep(z,w)}{\log \ep^{-1}} 
\eqe
a.s.\ well-defined and finite? If so, what is its value?
\end{ques}

Note that the model of the present paper, as described in Section~\ref{sec-tiling-def}, makes sense for $\ccM  = 25$ ($Q = 0$) and it is easy to see using basic Gaussian estimates that in this case, each fixed point of $\mathbb{C}$ is a.s.\ contained in a square of $\mcl S_h^\ep$ (i.e., the set of singularities has zero Lebesgue measure). We have neglected this case throughout the present paper. The following is a natural first question for $\ccM = 25$. 

\begin{ques}[Behavior of point-to-point distance for $\ccM = 25$] \label{ques-Q=0}
For $\ccM = 25$, is $\mcl S_h^\ep$ a.s.\ connected, i.e., is it a.s.\ possible to get from any square to any other square via a finite path in $\mcl S_h^\ep$? 
If so, does the distance between two fixed points a.s.\ grow at most polynomially in $\ep$? 
\end{ques}

We expect that determining whether the distance in Question~\ref{ques-Q=0} has polynomial growth is closely related to determining whether the limit in Question~\ref{ques-Q-to-0} is finite. 

Our last question concerns the relationship between the model considered in the present paper and certain natural random planar maps for $\ccM \in (1,25)$. 

\begin{ques}[Random planar map connection] \label{ques-discrete}
Is there a natural variant of random planar maps weighted by the $(-\ccM/2)$-th power of the Laplacian determinant for $\ccM \in (1,25)$ (or by some other partition function with similar asymptotic behavior) wherein the random planar maps are a.s.\ infinite, with infinitely many ends? If so, can one prove any rigorous relationship between such random planar maps and the maps $\mcl S_h^\ep$ considered in the present paper?
\end{ques}

\bibliography{cibiblong,cibib,extrabib}

\newcommand{\etalchar}[1]{$^{#1}$}
\def\cprime{$'$}
\begin{thebibliography}{DRSV14b}

\bibitem[AB14]{ambjorn-budd-lqg-geodesic}
J.~Ambj\"{o}rn and T.~G. Budd.
\newblock Geodesic distances in {Liouville} quantum gravity.
\newblock {\em Nuclear Physics B}, 889:676--691, 2014, \arxiv{1405.3424}.

\bibitem[ADF86]{adf-critical-dimensions}
J.~Ambj{\o}rn, B.~Durhuus, and J.~Fr\"{o}hlich.
\newblock The appearance of critical dimensions in regulated string theories.
  {II}.
\newblock {\em Nuclear Phys. B}, 275(2):161--184, 1986. \MR{858659}

\bibitem[ADH13]{adh-ghp}
R.~Abraham, J.-F. Delmas, and P.~Hoscheit.
\newblock A note on the {G}romov-{H}ausdorff-{P}rokhorov distance between
  (locally) compact metric measure spaces.
\newblock {\em Electron. J. Probab.}, 18:no. 14, 21, 2013,
  \arxiv{arXiv:1202.5464}. \MR{3035742}

\bibitem[ADJT93]{adjt-c-ge1}
J.~{Ambj{\o}rn}, B.~{Durhuus}, T.~{J{\'o}nsson}, and G.~{Thorleifsson}.
\newblock {Matter fields with c $>$ 1 coupled to 2d gravity}.
\newblock {\em Nuclear Physics B}, 398:568--592, June 1993,
  \arxiv{hep-th/9208030}.

\bibitem[Ald91a]{aldous-crt1}
D.~Aldous.
\newblock The continuum random tree. {I}.
\newblock {\em Ann. Probab.}, 19(1):1--28, 1991. \MR{1085326 (91i:60024)}

\bibitem[Ald91b]{aldous-crt2}
D.~Aldous.
\newblock The continuum random tree. {II}. {A}n overview.
\newblock In {\em Stochastic analysis ({D}urham, 1990)}, volume 167 of {\em
  London Math. Soc. Lecture Note Ser.}, pages 23--70. Cambridge Univ. Press,
  Cambridge, 1991. \MR{1166406 (93f:60010)}

\bibitem[Ald93]{aldous-crt3}
D.~Aldous.
\newblock The continuum random tree. {III}.
\newblock {\em Ann. Probab.}, 21(1):248--289, 1993. \MR{1207226 (94c:60015)}

\bibitem[Amb94]{ambjorn-remarks}
J.~Ambj{\o}rn.
\newblock Remarks about {$c>1$} and {$D>2$}.
\newblock {\em Teoret. Mat. Fiz.}, 98(3):326--336, 1994. \MR{1304731}

\bibitem[Ang19]{ang-discrete-lfpp}
M.~Ang.
\newblock {Comparison of discrete and continuum Liouville first passage
  percolation}.
\newblock {\em ArXiv e-prints}, Apr 2019, \arxiv{1904.09285}.

\bibitem[Aru15]{aru-kpz}
J.~Aru.
\newblock K{PZ} relation does not hold for the level lines and {SLE$_\kappa$}
  flow lines of the {G}aussian free field.
\newblock {\em Probab. Theory Related Fields}, 163(3-4):465--526, 2015,
  \arxiv{1312.1324}. \MR{3418748}

\bibitem[{Aru}17]{aru-gmc-survey}
J.~{Aru}.
\newblock {Gaussian multiplicative chaos through the lens of the 2D Gaussian
  free field}.
\newblock {\em ArXiv e-prints}, Sep 2017, \arxiv{1709.04355}.

\bibitem[BB19]{bb-lqg-dim}
J.~{Barkley} and T.~{Budd}.
\newblock {Precision measurements of Hausdorff dimensions in two-dimensional
  quantum gravity}.
\newblock {\em ArXiv e-prints}, Aug 2019, \arxiv{1908.09469}.

\bibitem[BD86]{bd-triangulated-surfaces}
A.~Billoire and F.~David.
\newblock Scaling properties of randomly triangulated planar random surfaces: a
  numerical study.
\newblock {\em Nuclear Phys. B}, 275(4):617--640, 1986. \MR{865231}

\bibitem[Bef08]{beffara-dim}
V.~Beffara.
\newblock The dimension of the {SLE} curves.
\newblock {\em Ann. Probab.}, 36(4):1421--1452, 2008, \arxiv{math/0211322}.
  \MR{2435854 (2009e:60026)}

\bibitem[Ber17]{berestycki-gmt-elementary}
N.~Berestycki.
\newblock An elementary approach to {G}aussian multiplicative chaos.
\newblock {\em Electron. Commun. Probab.}, 22:Paper No. 27, 12, 2017,
  \arxiv{1506.09113}. \MR{3652040}

\bibitem[BGRV16]{grv-kpz}
N.~Berestycki, C.~Garban, R.~Rhodes, and V.~Vargas.
\newblock K{PZ} formula derived from {L}iouville heat kernel.
\newblock {\em J. Lond. Math. Soc. (2)}, 94(1):186--208, 2016,
  \arxiv{1406.7280}. \MR{3532169}

\bibitem[BH92]{bh-c-ge1-matrix}
E.~{Br{\'e}zin} and S.~{Hikami}.
\newblock {A naive matrix-model approach to 2D quantum gravity coupled to
  matter of arbitrary central charge}.
\newblock {\em Physics Letters B}, 283:203--208, June 1992,
  \arxiv{hep-th/9204018}.

\bibitem[BJ92]{bj-potts-sim}
C.~F. {Baillie} and D.~A. {Johnston}.
\newblock {A Numerical Test of Kpz Scaling:. Potts Models Coupled to
  Two-Dimensional Quantum Gravity}.
\newblock {\em Modern Physics Letters A}, 7:1519--1533, 1992,
  \arxiv{hep-lat/9204002}.

\bibitem[BJM14]{bjm-uniform}
J.~Bettinelli, E.~Jacob, and G.~Miermont.
\newblock The scaling limit of uniform random plane maps, {\it via} the
  {A}mbj\o rn-{B}udd bijection.
\newblock {\em Electron. J. Probab.}, 19:no. 74, 16, 2014, 1312.5842.
  \MR{3256874}

\bibitem[BJRV13]{bjrv-gmt-duality}
J.~Barral, X.~Jin, R.~Rhodes, and V.~Vargas.
\newblock Gaussian multiplicative chaos and {KPZ} duality.
\newblock {\em Comm. Math. Phys.}, 323(2):451--485, 2013, \arxiv{1202.5296}.
  \MR{3096527}

\bibitem[BKKM86]{bkkm-analytical-study}
D.~V. Boulatov, V.~A. Kazakov, I.~K. Kostov, and A.~A. Migdal.
\newblock Analytical and numerical study of a model of dynamically triangulated
  random surfaces.
\newblock {\em Nuclear Phys. B}, 275(4):641--686, 1986. \MR{865232}

\bibitem[BS09]{benjamini-schramm-cascades}
I.~Benjamini and O.~Schramm.
\newblock K{PZ} in one dimensional random geometry of multiplicative cascades.
\newblock {\em Comm. Math. Phys.}, 289(2):653--662, 2009, \arxiv{0806.1347}.
  \MR{2506765 (2010c:60151)}

\bibitem[{Cat}88]{cates-branched-polymer}
M.~E. {Cates}.
\newblock {The Liouville field theory of random surfaces: when is the bosonic
  string a branched polymer?}
\newblock {\em EPL (Europhysics Letters)}, 7:719, December 1988.

\bibitem[CKR92]{ckr-c-ge1}
S.~{Catterall}, J.~{Kogut}, and R.~{Renken}.
\newblock {Numerical study of c $>$ 1 matter coupled to quantum gravity}.
\newblock {\em Physics Letters B}, 292:277--282, 1992.

\bibitem[Cur16]{curien-peeling-notes}
N.~Curien.
\newblock Peeling random planar maps. {N}otes du cours {P}eccot.
\newblock Available at
  \url{https://www.math.u-psud.fr/~curien/cours/peccot.pdf}, 2016.

\bibitem[Dav88]{david-conformal-gauge}
F.~David.
\newblock Conformal field theories coupled to {2-D} gravity in the conformal
  gauge.
\newblock {\em {M}od. {P}hys. {L}ett. {A}}, (3), 1988.

\bibitem[{Dav}97]{david-c>1-barrier}
F.~{David}.
\newblock {A scenario for the $c > 1$ barrier in non-critical bosonic strings}.
\newblock {\em Nuclear Physics B}, 487:633--649, February 1997,
  \arxiv{hep-th/9610037}.

\bibitem[DD18]{ding-dunlap-lgd}
J.~{Ding} and A.~{Dunlap}.
\newblock {Subsequential scaling limits for Liouville graph distance}.
\newblock {\em ArXiv e-prints}, December 2018, \arxiv{1812.06921}.

\bibitem[DDDF19]{dddf-lfpp}
J.~Ding, J.~Dub{\'e}dat, A.~Dunlap, and H.~Falconet.
\newblock {Tightness of Liouville first passage percolation for $\gamma \in
  (0,2)$}.
\newblock {\em ArXiv e-prints}, Apr 2019, \arxiv{1904.08021}.

\bibitem[DFG{\etalchar{+}}19]{lqg-metric-estimates}
J.~Dub{\'e}dat, H.~Falconet, E.~Gwynne, J.~Pfeffer, and X.~Sun.
\newblock {W}eak {LQG} metrics and {L}iouville first passage percolation.
\newblock {\em ArXiv e-prints}, May 2019, \arxiv{1905.00380}.

\bibitem[DFJ84]{dfj-critical-behavior}
B.~Durhuus, J.~Frohlich, and T.~Jonsson.
\newblock {Critical Behavior in a Model of Planar Random Surfaces}.
\newblock {\em Nucl. Phys.}, B240:453, 1984.
\newblock [Phys. Lett.137B,93(1984)].

\bibitem[DG16]{ding-goswami-watabiki}
J.~{Ding} and S.~{Goswami}.
\newblock {Upper bounds on Liouville first passage percolation and Watabiki's
  prediction}.
\newblock {\em {C}ommunications in {P}ure and {A}pplied {M}athematics}, to
  appear, 2016, \arxiv{1610.09998}.

\bibitem[DG18]{dg-lqg-dim}
J.~{Ding} and E.~{Gwynne}.
\newblock {The fractal dimension of {L}iouville quantum gravity: universality,
  monotonicity, and bounds}.
\newblock {\em {C}ommunications in {M}athematical {P}hysics}, to appear, 2018,
  \arxiv{1807.01072}.

\bibitem[DJKP87]{djkp-critical-exponents}
F.~David, J.~Jurkiewicz, A.~Krzywicki, and B.~Petersson.
\newblock Critical exponents in a model of dynamically triangulated random
  surfaces.
\newblock {\em Nuclear Physics B}, 290:218 -- 230, 1987.

\bibitem[DK89]{dk-qg}
J.~Distler and H.~Kawai.
\newblock Conformal field theory and {2D} quantum gravity.
\newblock {\em {N}ucl.{P}hys. {B}}, (321), 1989.

\bibitem[DKRV16]{dkrv-lqg-sphere}
F.~David, A.~Kupiainen, R.~Rhodes, and V.~Vargas.
\newblock Liouville quantum gravity on the {R}iemann sphere.
\newblock {\em Comm. Math. Phys.}, 342(3):869--907, 2016, \arxiv{1410.7318}.
  \MR{3465434}

\bibitem[DL18]{ding-li-chem-dist}
J.~Ding and L.~Li.
\newblock Chemical distances for percolation of planar {G}aussian free fields
  and critical random walk loop soups.
\newblock {\em Comm. Math. Phys.}, 360(2):523--553, 2018, \arxiv{1605.04449}.
  \MR{3800790}

\bibitem[DMS14]{wedges}
B.~{Duplantier}, J.~{Miller}, and S.~{Sheffield}.
\newblock {Liouville quantum gravity as a mating of trees}.
\newblock {\em ArXiv e-prints}, September 2014, \arxiv{1409.7055}.

\bibitem[DO94]{do-dozz}
H.~{Dorn} and H.-J. {Otto}.
\newblock {Two- and three-point functions in Liouville theory}.
\newblock {\em Nuclear Physics B}, 429:375--388, October 1994,
  \arxiv{hep-th/9403141}.

\bibitem[DP86]{dp-multiloop}
E.~D'Hoker and D.~H. Phong.
\newblock Multiloop amplitudes for the bosonic {P}olyakov string.
\newblock {\em Nuclear Phys. B}, 269(1):205--234, 1986. \MR{838673}

\bibitem[DRSV14a]{shef-deriv-mart}
B.~Duplantier, R.~Rhodes, S.~Sheffield, and V.~Vargas.
\newblock Critical {G}aussian multiplicative chaos: convergence of the
  derivative martingale.
\newblock {\em Ann. Probab.}, 42(5):1769--1808, 2014, \arxiv{1206.1671}.
  \MR{3262492}

\bibitem[DRSV14b]{shef-renormalization}
B.~Duplantier, R.~Rhodes, S.~Sheffield, and V.~Vargas.
\newblock Renormalization of critical {G}aussian multiplicative chaos and {KPZ}
  relation.
\newblock {\em Comm. Math. Phys.}, 330(1):283--330, 2014, \arxiv{1212.0529}.
  \MR{3215583}

\bibitem[DRV16]{drv-torus}
F.~David, R.~Rhodes, and V.~Vargas.
\newblock Liouville quantum gravity on complex tori.
\newblock {\em J. Math. Phys.}, 57(2):022302, 25, 2016, \arxiv{1504.00625}.
  \MR{3450564}

\bibitem[DS11]{shef-kpz}
B.~Duplantier and S.~Sheffield.
\newblock Liouville quantum gravity and {KPZ}.
\newblock {\em Invent. Math.}, 185(2):333--393, 2011, \arxiv{1206.0212}.
  \MR{2819163 (2012f:81251)}

\bibitem[DS19]{ds-ricci-flow}
J.~{Dub{\'e}dat} and H.~{Shen}.
\newblock {Stochastic Ricci Flow on Compact Surfaces}.
\newblock {\em ArXiv e-prints}, Apr 2019, \arxiv{1904.10909}.

\bibitem[Dup10]{dup-dual-lqg}
B.~Duplantier.
\newblock A rigorous perspective on {L}iouville quantum gravity and the {KPZ}
  relation.
\newblock In {\em Exact methods in low-dimensional statistical physics and
  quantum computing}, pages 529--561. Oxford Univ. Press, Oxford, 2010.
  \MR{2668656}

\bibitem[DZZ18]{dzz-heat-kernel}
J.~{Ding}, O.~{Zeitouni}, and F.~{Zhang}.
\newblock {Heat kernel for Liouville Brownian motion and Liouville graph
  distance}.
\newblock {\em {C}ommunications in {M}athematical {P}hysics}, to appear, 2018,
  \arxiv{1807.00422}.

\bibitem[FK02]{fk-c>1-II}
L.~D. {Faddeev} and R.~M. {Kashaev}.
\newblock {Strongly coupled quantum discrete Liouville theory: II. Geometric
  interpretation of the evolution operator}.
\newblock {\em Journal of Physics A Mathematical General}, 35:4043--4048, May
  2002, \arxiv{hep-th/0201049}.

\bibitem[FKV01]{fkv-c>1-I}
L.~D. {Faddeev}, R.~M. {Kashaev}, and A.~Y. {Volkov}.
\newblock {Strongly Coupled Quantum Discrete Liouville Theory.I: Algebraic
  Approach and Duality}.
\newblock {\em Communications in Mathematical Physics}, 219:199--219, 2001,
  \arxiv{hep-th/0006156}.

\bibitem[{Gar}18]{garban-dynamical}
C.~{Garban}.
\newblock {Dynamical Liouville}.
\newblock {\em ArXiv e-prints}, May 2018, \arxiv{1805.04507}.

\bibitem[GGN13]{gn-recurrence}
O.~Gurel-Gurevich and A.~Nachmias.
\newblock Recurrence of planar graph limits.
\newblock {\em Ann. of Math. (2)}, 177(2):761--781, 2013, \arxiv{1206.0707}.
  \MR{3010812}

\bibitem[GHM15]{ghm-kpz}
E.~{Gwynne}, N.~{Holden}, and J.~{Miller}.
\newblock {An almost sure KPZ relation for SLE and Brownian motion}.
\newblock {\em {A}nnals of {P}robability}, to appear, 2015, \arxiv{1512.01223}.

\bibitem[GHS17]{ghs-map-dist}
E.~{Gwynne}, N.~{Holden}, and X.~{Sun}.
\newblock {A mating-of-trees approach for graph distances in random planar
  maps}.
\newblock {\em ArXiv e-prints}, November 2017, \arxiv{1711.00723}.

\bibitem[GHS19]{ghs-mating-survey}
E.~{Gwynne}, N.~{Holden}, and X.~{Sun}.
\newblock {Mating of trees for random planar maps and Liouville quantum
  gravity: a survey}.
\newblock {\em ArXiv e-prints}, Oct 2019, \arxiv{1910.04713}.

\bibitem[GM17]{gwynne-miller-char}
E.~{Gwynne} and J.~{Miller}.
\newblock {Characterizations of SLE$_{\kappa}$ for $\kappa \in (4,8)$ on
  Liouville quantum gravity}.
\newblock {\em ArXiv e-prints}, January 2017, \arxiv{1701.05174}.

\bibitem[GM19a]{gm-coord-change}
E.~Gwynne and J.~Miller.
\newblock Conformal covariance of the {L}iouville quantum gravity metric for
  {$\gamma \in (0,2)$}.
\newblock {\em ArXiv e-prints}, May 2019, \arxiv{1905.00384}.

\bibitem[GM19b]{gm-uniqueness}
E.~Gwynne and J.~Miller.
\newblock Existence and uniqueness of the {L}iouville quantum gravity metric
  for {$\gamma \in (0,2)$}.
\newblock {\em ArXiv e-prints}, May 2019, \arxiv{1905.00383}.

\bibitem[GMS17]{gms-tutte}
E.~{Gwynne}, J.~{Miller}, and S.~{Sheffield}.
\newblock {The Tutte embedding of the mated-CRT map converges to Liouville
  quantum gravity}.
\newblock {\em ArXiv e-prints}, May 2017, \arxiv{1705.11161}.

\bibitem[GMS18]{gms-random-walk}
E.~{Gwynne}, J.~{Miller}, and S.~{Sheffield}.
\newblock {An invariance principle for ergodic scale-free random environments}.
\newblock {\em ArXiv e-prints}, July 2018, \arxiv{1807.07515}.

\bibitem[GP19a]{gp-lfpp-bounds}
E.~{Gwynne} and J.~{Pfeffer}.
\newblock {Bounds for distances and geodesic dimension in Liouville first
  passage percolation}.
\newblock {\em {E}lectronic {C}ommunications in {P}robability}, 24:no. 56, 12,
  2019, \arxiv{1903.09561}.

\bibitem[GP19b]{gp-kpz}
E.~{Gwynne} and J.~{Pfeffer}.
\newblock {KPZ formulas for the Liouville quantum gravity metric}.
\newblock {\em ArXiv e-prints}, May 2019, \arxiv{1905.11790}.

\bibitem[GRV16]{grv-higher-genus}
C.~{Guillarmou}, R.~{Rhodes}, and V.~{Vargas}.
\newblock {Polyakov's formulation of $2d$ bosonic string theory}.
\newblock {\em ArXiv e-prints}, July 2016, \arxiv{1607.08467}.

\bibitem[{Gwy}19]{gwynne-ball-bdy}
E.~{Gwynne}.
\newblock {The dimension of the boundary of a Liouville quantum gravity metric
  ball}.
\newblock {\em arXiv e-prints}, Sep 2019, \arxiv{1909.08588}.

\bibitem[HMP10]{hmp-thick-pts}
X.~Hu, J.~Miller, and Y.~Peres.
\newblock Thick points of the {G}aussian free field.
\newblock {\em Ann. Probab.}, 38(2):896--926, 2010, \arxiv{0902.3842}.
  \MR{2642894 (2011c:60117)}

\bibitem[HRV18]{hrv-disk}
Y.~Huang, R.~Rhodes, and V.~Vargas.
\newblock Liouville quantum gravity on the unit disk.
\newblock {\em Ann. Inst. Henri Poincar\'{e} Probab. Stat.}, 54(3):1694--1730,
  2018, \arxiv{1502.04343}. \MR{3825895}

\bibitem[HS19]{hs-cardy-embedding}
N.~{Holden} and X.~{Sun}.
\newblock {Convergence of uniform triangulations under the Cardy embedding}.
\newblock {\em ArXiv e-prints}, May 2019, \arxiv{1905.13207}.

\bibitem[{Hua}18]{huang-complex-insertion}
Y.~{Huang}.
\newblock {Path integral approach to analytic continuation of Liouville theory:
  the pencil region}.
\newblock {\em ArXiv e-prints}, September 2018, 1809.08650.

\bibitem[IJS16]{ijs16}
Y.~Ikhlef, J.~L. Jacobsen, and H.~Saleur.
\newblock Three-point functions in $c\ensuremath{\le}1$ {L}iouville theory and
  conformal loop ensembles.
\newblock {\em Phys. Rev. Lett.}, 116:130601, Mar 2016.

\bibitem[JSW18a]{jsw-imaginary-gmc}
J.~{Junnila}, E.~{Saksman}, and C.~{Webb}.
\newblock {Imaginary multiplicative chaos: Moments, regularity and connections
  to the Ising model}.
\newblock {\em ArXiv e-prints}, June 2018, \arxiv{1806.02118}.

\bibitem[JSW18b]{jsw-decompositions}
J.~{Junnila}, E.~{Saksman}, and C.~{Webb}.
\newblock {Decompositions of log-correlated fields with applications}.
\newblock {\em ArXiv e-prints}, Aug 2018, \arxiv{1808.06838}.

\bibitem[Kah85]{kahane}
J.-P. Kahane.
\newblock Sur le chaos multiplicatif.
\newblock {\em Ann. Sci. Math. Qu\'ebec}, 9(2):105--150, 1985. \MR{829798
  (88h:60099a)}

\bibitem[KL13]{kl-lambda-saw}
T.~Kennedy and G.~F. Lawler.
\newblock Lattice effects in the scaling limit of the two-dimensional
  self-avoiding walk.
\newblock In {\em Fractal geometry and dynamical systems in pure and applied
  mathematics. {II}. {F}ractals in applied mathematics}, volume 601 of {\em
  Contemp. Math.}, pages 195--210. Amer. Math. Soc., Providence, RI, 2013,
  \arxiv{1109.3091}. \MR{3203863}

\bibitem[{Kle}95]{klebanov-touching}
I.~R. {Klebanov}.
\newblock {Touching random surfaces and Liouville gravity}.
\newblock {\em {P}hys. {R}ev. {D}}, 51:1836--1841, February 1995,
  \arxiv{hep-th/9407167}.

\bibitem[KPZ88]{kpz-scaling}
V.~Knizhnik, A.~Polyakov, and A.~Zamolodchikov.
\newblock {Fractal structure of 2D-quantum gravity}.
\newblock {\em {Modern Phys. Lett A}}, 3(8):819--826, 1988.

\bibitem[KRV15]{krv-local}
A.~{Kupiainen}, R.~{Rhodes}, and V.~{Vargas}.
\newblock {Local Conformal Structure of Liouville Quantum Gravity}.
\newblock {\em ArXiv e-prints}, December 2015, \arxiv{1512.01802}.

\bibitem[KRV17]{krv-dozz}
A.~{Kupiainen}, R.~{Rhodes}, and V.~{Vargas}.
\newblock {Integrability of Liouville theory: proof of the DOZZ Formula}.
\newblock {\em {A}nnals of {M}athematics}, to appear, 2017, \arxiv{1707.08785}.

\bibitem[{Le }13]{legall-uniqueness}
J.-F. {Le Gall}.
\newblock Uniqueness and universality of the {B}rownian map.
\newblock {\em Ann. Probab.}, 41(4):2880--2960, 2013, \arxiv{1105.4842}.
  \MR{3112934}

\bibitem[{Le }14]{legall-sphere-survey}
J.-F. {Le Gall}.
\newblock {Random geometry on the sphere}.
\newblock {\em Proceedings of the {ICM}}, 2014, \arxiv{1403.7943}.

\bibitem[LR15]{lawler-rezai-nat}
G.~F. Lawler and M.~A. Rezaei.
\newblock Minkowski content and natural parameterization for the
  {S}chramm-{L}oewner evolution.
\newblock {\em Ann. Probab.}, 43(3):1082--1120, 2015, \arxiv{1211.4146}.
  \MR{3342659}

\bibitem[LRV13]{rhodes-vargas-complex-gmt}
H.~{Lacoin}, R.~{Rhodes}, and V.~{Vargas}.
\newblock {Complex Gaussian multiplicative chaos}.
\newblock {\em ArXiv e-prints}, July 2013, \arxiv{1307.6117}.

\bibitem[LRV19]{lrv-sine-gordon}
H.~{Lacoin}, R.~{Rhodes}, and V.~{Vargas}.
\newblock {A probabilistic approach of ultraviolet renormalisation in the
  boundary Sine-Gordon model}.
\newblock {\em ArXiv e-prints}, Mar 2019, \arxiv{1903.01394}.

\bibitem[Mie13]{miermont-brownian-map}
G.~Miermont.
\newblock The {B}rownian map is the scaling limit of uniform random plane
  quadrangulations.
\newblock {\em Acta Math.}, 210(2):319--401, 2013, \arxiv{1104.1606}.
  \MR{3070569}

\bibitem[Mie14]{miermont-st-flour}
G.~Miermont.
\newblock {\em Aspects of random maps}.
\newblock 2014.
\newblock St. Flour lecture notes. Available at
  \href{http://perso.ens-lyon.fr/gregory.miermont/coursSaint-Flour.pdf}{http://perso.ens-lyon.fr/gregory.miermont/coursSaint-Flour.pdf}.

\bibitem[MS15]{lqg-tbm1}
J.~{Miller} and S.~{Sheffield}.
\newblock {Liouville quantum gravity and the Brownian map I: The QLE(8/3,0)
  metric}.
\newblock {\em Inventiones Mathematicae}, to appear, 2015, \arxiv{1507.00719}.

\bibitem[MS16a]{lqg-tbm2}
J.~{Miller} and S.~{Sheffield}.
\newblock {Liouville quantum gravity and the Brownian map II: geodesics and
  continuity of the embedding}.
\newblock {\em ArXiv e-prints}, May 2016, \arxiv{1605.03563}.

\bibitem[MS16b]{lqg-tbm3}
J.~{Miller} and S.~{Sheffield}.
\newblock {Liouville quantum gravity and the Brownian map III: the conformal
  structure is determined}.
\newblock {\em ArXiv e-prints}, August 2016, \arxiv{1608.05391}.

\bibitem[MS16c]{ig1}
J.~Miller and S.~Sheffield.
\newblock Imaginary geometry {I}: interacting {SLE}s.
\newblock {\em Probab. Theory Related Fields}, 164(3-4):553--705, 2016,
  \arxiv{1201.1496}. \MR{3477777}

\bibitem[MS17]{ig4}
J.~Miller and S.~Sheffield.
\newblock Imaginary geometry {IV}: interior rays, whole-plane reversibility,
  and space-filling trees.
\newblock {\em Probab. Theory Related Fields}, 169(3-4):729--869, 2017,
  \arxiv{1302.4738}. \MR{3719057}

\bibitem[MS19]{sphere-constructions}
J.~Miller and S.~Sheffield.
\newblock Liouville quantum gravity spheres as matings of finite-diameter
  trees.
\newblock {\em Ann. Inst. Henri Poincar\'{e} Probab. Stat.}, 55(3):1712--1750,
  2019, \arxiv{1506.03804}. \MR{4010949}

\bibitem[Pol81]{polyakov-qg1}
A.~M. Polyakov.
\newblock Quantum geometry of bosonic strings.
\newblock {\em Phys. Lett. B}, 103(3):207--210, 1981. \MR{623209 (84h:81093a)}

\bibitem[Rem18]{remy-annulus}
G.~Remy.
\newblock Liouville quantum gravity on the annulus.
\newblock {\em J. Math. Phys.}, 59(8):082303, 26, 2018, \arxiv{1711.06547}.
  \MR{3843631}

\bibitem[{Rib}14]{ribault-cft}
S.~{Ribault}.
\newblock {Conformal field theory on the plane}.
\newblock {\em ArXiv e-prints}, June 2014, \arxiv{1406.4290}.

\bibitem[Rib18]{ribault2018minimal}
S.~Ribault.
\newblock Minimal lectures on two-dimensional conformal field theory.
\newblock {\em SciPost Physics Lecture Notes}, page 001, 2018.

\bibitem[RS15]{rs15}
S.~Ribault and R.~Santachiara.
\newblock Liouville theory with a central charge less than one.
\newblock {\em Journal of High Energy Physics}, 2015(8):109, 2015.

\bibitem[RV11]{rhodes-vargas-log-kpz}
R.~Rhodes and V.~Vargas.
\newblock K{PZ} formula for log-infinitely divisible multifractal random
  measures.
\newblock {\em ESAIM Probab. Stat.}, 15:358--371, 2011, \arxiv{0807.1036}.
  \MR{2870520}

\bibitem[RV14]{rhodes-vargas-review}
R.~Rhodes and V.~Vargas.
\newblock Gaussian multiplicative chaos and applications: {A} review.
\newblock {\em Probab. Surv.}, 11:315--392, 2014, \arxiv{1305.6221}.
  \MR{3274356}

\bibitem[She07]{shef-gff}
S.~Sheffield.
\newblock Gaussian free fields for mathematicians.
\newblock {\em Probab. Theory Related Fields}, 139(3-4):521--541, 2007,
  \arxiv{math/0312099}. \MR{2322706 (2008d:60120)}

\bibitem[She16]{shef-zipper}
S.~Sheffield.
\newblock Conformal weldings of random surfaces: {SLE} and the quantum gravity
  zipper.
\newblock {\em Ann. Probab.}, 44(5):3474--3545, 2016, \arxiv{1012.4797}.
  \MR{3551203}

\bibitem[SS13]{ss-contour}
O.~Schramm and S.~Sheffield.
\newblock A contour line of the continuum {G}aussian free field.
\newblock {\em Probab. Theory Related Fields}, 157(1-2):47--80, 2013,
  \arxiv{math/0605337}. \MR{3101840}

\bibitem[Suz97]{suzuki1997note}
T.~Suzuki.
\newblock A note on quantum liouville theory via the quantum group an approach
  to strong coupling liouville theory.
\newblock {\em Nuclear Physics B}, 492(3):717--742, 1997.

\bibitem[Tes04]{teschner04}
J.~Teschner.
\newblock A lecture on the {L}iouville vertex operators.
\newblock {\em International Journal of Modern Physics A}, 19(supp02):436--458,
  2004.

\bibitem[Wat93]{watabiki-lqg}
Y.~Watabiki.
\newblock {Analytic study of fractal structure of quantized surface in
  two-dimensional quantum gravity}.
\newblock {\em Progr. Theor. Phys. Suppl.}, (114):1--17, 1993.
\newblock Quantum gravity (Kyoto, 1992).

\bibitem[Zam05]{zam05}
A.~B. Zamolodchikov.
\newblock Three-point function in the minimal {L}iouville gravity.
\newblock {\em Theoretical and mathematical physics}, 142(2):183--196, 2005.

\bibitem[ZZ96]{zz-dozz}
A.~{Zamolodchikov} and A.~{Zamolodchikov}.
\newblock {Conformal bootstrap in Liouville field theory}.
\newblock {\em Nuclear Physics B}, 477:577--605, February 1996,
  \arxiv{hep-th/9506136}.

\end{thebibliography}
\bibliographystyle{hmralphaabbrv}

\end{document}